\documentclass{article}
\usepackage{amsmath,amsthm,amssymb,amscd}

\newcommand{\ga}{\alpha}
\newcommand{\gb}{\beta}
\renewcommand{\gg}{\gamma}
\newcommand{\gd}{\delta}
\newcommand{\gw}{\omega}

\newcommand{\gS}{\Sigma}
\newcommand{\gs}{\sigma}
\newcommand{\eps}{\varepsilon}

\newcommand{\R}{\mathbb{R}}

\newcommand{\cH}{\mathcal{H}}

\newcommand{\cG}{\mathcal{G}}

\newcommand{\liff}{\leftrightarrow}

\newcommand{\cantor}{2^\gw}
\newcommand{\baire}{\gw^\gw}
\newcommand{\bintree}{2^{<\gw}}
\newcommand{\gwtree}{\gw^{<\gw}}

\newcommand{\dottgen}{{\dot T}_{\mathit{gen}}}

\newcommand{\dotxgen}{{\dot x}_{\mathit{gen}}}
\newcommand{\dotbgen}{{\dot b}_{\mathit{gen}}}
\newcommand{\dotygen}{{\dot y}_{\mathit{gen}}}
\newcommand{\vecygen}{{\vec y}_{\mathit{gen}}}
\newcommand{\xgen}{x_{\mathit{gen}}}

\newcommand{\cmin}{c_{\mathrm{min}}}

\newcommand{\dom}{\mathrm{dom}}

\newcommand{\rng}{\mathrm{rng}}
\newcommand{\power}{\mathcal{P}}
\newcommand{\pioneoneonsigmaoneone}{${\mathbf{\Pi}}^1_1$ on ${\mathbf{\gS}}^1_1$}

\newcommand{\Prod}{\Pi}

\newtheorem{theorem}{Theorem}[section]

\newtheorem{claim}[theorem]{Claim}
\newtheorem{corollary}[theorem]{Corollary}
\newtheorem{fact}[theorem]{Fact}
\newtheorem{proposition}[theorem]{Proposition}

\theoremstyle{definition}
\newtheorem{definition}[theorem]{Definition}
\newtheorem{example}[theorem]{Example}
\newtheorem{question}[theorem]{Question}

\title{Hypergraphs and proper forcing}

\author{
Jind{\v r}ich Zapletal\\
University of Florida}

\begin{document}
\maketitle

\begin{abstract}
Given a Polish space $X$ and a countable collection of analytic hypergraphs on $X$, I consider the $\gs$-ideal generated by Borel anticliques for the hypergraphs in the family. It turns out that many of the quotient posets are proper. I investigate the forcing properties of these posets, certain natural operations on them, and prove some related dichotomies.
\end{abstract}

\section{Introduction}

The purpose of this paper is to introduce a class of proper forcing notions which can be in a natural way associated with hypergraphs on Polish spaces. There is a number of theorems which connect simple combinatorial properties of the hypergraphs with deep forcing properties of the resulting forcing notions. The story begins with the concept of analytic hypergraphs and their associated $\gs$-ideals: 

\begin{definition}
Let $X$ be a Polish space.

\begin{enumerate}
\item A \emph{hypergraph} on $X$ is a subset of $X^{\leq \gw}$;
\item if $G$ is a hypergraph on $X$, a $G$-\emph{anticlique} is a set $B\subset X$ such that $B^{\leq\gw}\cap G=0$;
\item if $\cG$ is a countable collection of hypergraphs on $X$, then $I_{\cG}$ is the $\gs$-ideal on $X$ $\gs$-generated by the Borel subsets of $X$ which happen to be anticliques in at least one of the hypergraphs in $\cG$. $\gs$-ideals generated by countable collection of analytic hypergraphs are called \emph{hypergraphable}.
\end{enumerate}
\end{definition}

In the spirit of \cite{z:book2}, I will be interested in the quotient posets $P_{I_\cG}$ of Borel sets positive with respect to the ideal $I_{\cG}$, ordered by inclusion. Such posets may fail to be proper already in very simple circumstances. Still, a great majority of definable, proper forcings preserving Baire category in the literature can be conveniently presented as quotient forcings of hypergraphable ideals. They can be so presented, but invariably they are not, since the authors have not been aware of the existence and the great advantages and comfort of such a presentation. The purpose of this paper is to change this unfortunate situation.

In Section~\ref{featuresection}
I isolate two several broad classes of hypergraphs for which the quotient is proper: the actionable hypergraphs, associated with a countable group action on the underlying Polish space and the nearly open hypergraphs.
Theorem~\ref{maintheorem} shows the properness of the actionable posets; these in fact can be canonically decomposed into (as opposed to regularly embedded to) a two step $\gs$-closed*c.c.c.\ iteration--Theorem~\ref{intermediatetheorem}.  
The nearly open hypergraphs give posets with quite different properties, one of the distictions being that they generate a minimal forcing extension unless they add a Cohen real.

Section~\ref{fubinisection} shows that there is an enormous amount of general information that one can derive just from the very basic properties of the generating hypergraphs. For example, if the generating hypergraphs all have arity two, then in the resulting extension, every element of $\baire$ is a branch through some ground model $2$-branching tree on $\gw$--Corollary~\ref{localizationcorollary}. If each of the generating hypergraphs has finite arity, then the poset has the Sacks property--Corollary~\ref{sackscorollary}.
If the edges of the hypergraphs can be diagonalized in a natural sense, then the resulting forcing is bounding--Corollary~\ref{bounding1corollary}. If the hypergraphs have a simple Fubini property, then the quotient forcing preserves outer Lebesgue measure--Corollary~\ref{measurecorollary}. In any case, the hypergraphable forcings preserve the Baire category--Corollary~\ref{categorycorollary}. There are useful combinatorial criteria for dealing with compact anticliques in closed graphs, such as the ones presented in Theorem~\ref{closedanticliquetheorem} or~\ref{millertheorem}. Other such criteria concern adding an independent real, Corollary~\ref{boocorollary} or Theorem~\ref{independenttheorem}. Each of these theorems is adorned with a number of examples, some new, others well-known. Other similar theorems will appear in forthcoming work.

It may seem that one could classify quotient posets coming from very restrictive hypergraph classes. Section~\ref{invariantsection} attempts to do just that for some types of invariant graphs. There are two well-known examples, the Silver forcing (Example~\ref{silvergraphexample}) and the Vitali forcing (Example~\ref{vitaligraphexample}), and all the other posets in the category are in a precise sense between these two. Many invariant graphs give rise to quotient posets which are naturally isomorphic to one of these two--Theorems~\ref{seconddichotomytheorem} and ~\ref{thirddichotomytheorem}. However, there are some intermediate examples such as the Vitali-odd forcing (Example~\ref{vitalioddexample}), and many examples where I cannot determine what exactly is happening, typically connected with fine additive or group combinatorics (Subsection~\ref{groupsubsection}). Subsection~\ref{kstsubsection} shows that there is a natural proper poset derived from the (non-invariant) KST graph and proves an attendant dichotomy.

Section~\ref{othersection} studies operations on hypergraphs that lead to operations on partial orders. The countable support product is studied in Section~\ref{productsubsection}, and it turns out that in the case of posets given by actionable families of finitary hypergraphs, the product is hypergraphable again and the computation of the associated ideal is so simple that it gives rise to many preservation properties for product
that seem to be very awkward to obtain in any other way. One can also represent countable support iterations, even illfounded ones, using natural operations on hypergraphs, Subsection~\ref{iterationsubsection}.

The last section of the paper provides the descriptive complexity computations necessary to push all the proofs through. The computations were mostly known previously, either as folklore theorems or as published results.

There is an overwhelming number of open questions; I will mention two of strategic nature. The classes  of analytic hypergraps isolated in this paper do not exhaust the class of hypergraphable proper forcings by any stretch. The most natural question in the area is wide open:

\begin{question}
Characterize the analytic hypergraphs $G$ such that the quotient poset $P_{I_G}$ is proper.
\end{question}

\noindent It is just as unclear to me which proper partial orders can be presented as hypergraphable. It is clear from the work in this paper that such a presentation advances the understanding of the forcing properties of the poset more than any other piece of information.
The only two significant restrictions on the class of hypergraphable forcing are that the associated $\gs$-ideal is \pioneoneonsigmaoneone\ and the poset preserves Baire category. However, within these limitations there are many posets which I suspect cannot be presented as hypergraphable, such as the product of two copies of Sacks forcing. Thus, another question begs an answer:

\begin{question}
Characterize the idealized forcings on Polish spaces which can be presented as hypergraphable.
\end{question}

The notation follows the set theoretic standard of \cite{jech:set}.  If $I$ is a $\gs$-ideal on a Polish space, the symbol $P_I$ denotes the poset of Borel $I$-positive subsets of the space, ordered by inclusion. A hypergraph is finitary if all its edges are finite, possibly of arbitrarily large finite sizes. If $C$ is a class of hypergraphs, a $\gs$-ideal is $C$-hypergraphable if there is a a countable family $\cG$ of analytic hypergraphs in the class $C$ such that $I=I_{\cG}$; thus, I speak of finitary hypergraphable ideals, graphable ideals, nearly open hypergraphable ideals etc. I will need precise terminology regarding maps between hypergraphs: suppose that $X, Y$ are Polish spaces with respective hypergraphs $G, H$ on them, and suppose that $h\colon X\to Y$ is a continuous injection. I will say that a function $h\colon X\to Y$ is a \emph{homomorphism} of $G$ to $H$ if for every edge $e\in G$, the sequence $h\circ e$ belongs to $H$. The function $h$ is a \emph{reduction} of $G$ to $H$ if for every sequence $e\in X^{\leq\gw}$, $e\in G\liff h\circ e\in H$. The function $h$ is a \emph{near reduction}
if $X$ can be decomposed into countably many Borel sets $B_n$ for $n\in\gw$ such that the function $h\restriction B_n$ is a reduction of $G$ to $H$ for every number $n\in\gw$. Note that if $h$ is a near reduction of $G$ to $H$, then the $h$- image of any Borel $G$-anticlique
decomposes into countably many Borel $H$-anticliques and the $h$-preimage of any Borel $H$-anticlique decomposes into countably many Borel $G$-anticliques. Thus, any near reduction $h$ of $G$ to $H$ transports the ideal $I_G$ to $I_H$ restricted to the range of $h$.
If $\cG$ and $\cH$ are countable families of hypergraphs on the respective Polish spaces $X, Y$, then a \emph{near reduction} of $\cG$ to $\cH$ is a continuous injection $h\colon X\to Y$ together with a bijection $\pi\colon \cG\to\cH$ such that
for every $G\in\cG$, $h$ is a near reduction of $G$ to $\pi(G)$. As before, a near reduction clearly transports the ideal $I_{\cG}$ to the ideal $I_{\cH}$ restricted to the range of $h$.

\section{Properness theorems}
\label{featuresection}

\subsection{Actionable hypergraphs}

\begin{definition}
Let $X$ be a Polish space and $\cG$ a countable set of analytic hypergraphs on $X$. Say that $\cG$ is \emph{actionable} if there is a countable group $\Gamma$ and a Borel action of $\Gamma$ on $X$ such that 

\begin{enumerate}
\item every edge of every hypergraph in $\cG$ is a subset of a single orbit;
\item for every $G\in\cG$ and every $\gg\in\Gamma$, $\gg\cdot G\in\cG$. 
\end{enumerate}

\noindent A $\gs$-ideal $I$ on $X$ is \emph{actionable} if there is an actionable collection $\cG$ of hypergraphs on $X$ such that $I$ is $\gs$-generated by Borel sets which are anticliques in at least one of the hypergraphs in $\cG$.
\end{definition}

\begin{theorem}
\label{maintheorem}
Suppose that $I$ is an actionable $\gs$-ideal on a Polish space $X$.  Then
the poset $P_I$ of Borel $I$-positive sets ordered by inclusion is proper.
\end{theorem}

\begin{proof}
Let $\cG$ be a countable family of analytic hypergraphs and let $\Gamma$ be a countable group with an action witnessing the assumption that the ideal $I$ is actionable. Let $\dotxgen$ be the $P_I$-name for the generic element of the space $X$. The next claim yields a key homogeneity feature of the poset $P_I$.

\begin{claim}
\label{prepclaim2}
Let $\gamma\in\Gamma$ be any element. The map $x\mapsto\gamma\cdot x$ induces an automorphism of the ordering $P_I$.
\end{claim}

\begin{proof}
Note that if a Borel set $B\subset X$ is a $G$-anticlique for some graph $G\in\cG$ and $\gamma\in\Gamma$ is any element, then $\gamma\cdot B$ is a Borel $\gamma\cdot G$-anticlique and $\gamma\cdot G\in\cG$. This means that the set of generators of $I$ and consequently the whole ideal $I$ is invariant under the action of the group. The claim immediately follows.
\end{proof}

\noindent The second preparatory claim describes a key derivative operation on conditions in the poset $P_I$.

\begin{claim}
\label{prepclaim3}
Let $A\subset X$ be a Borel $I$-positive set and $G\in\cG$. Then the set $B=\{x\in A\colon \exists e\in G\cap A^\gw\ x\in\rng(e)\}$ is analytic and $I$-positive.
\end{claim}

\begin{proof}
Suppose that the set $B$ belongs to the ideal $I$. Then, since the ideal  $I$ is by definition generated by Borel sets, there must be a Borel set $B'\in I$ which is a superset of $B$. The definition of the set $B$ shows that the set $A\setminus B$ is a $G$-anticlique and so the set $A\setminus B'$ is a Borel $G$-anticlique.  By the definitions, the set $A$ is in the ideal $I$, contradicting the initial assumptions.
\end{proof}

Let $M$ be a countable elementary submodel of a large structure. The poset $P_I\cap M$ still adds a single point $\dotxgen$ such that the generic filter is the collection of those sets in $P_I\cap M$ which  contain $\dotxgen$, and the action still induces automorphisms of the poset $P_I\cap M$. The following claim is central.

\begin{claim}
\label{keyclaim}
If $G\in\cG$ and $C\subset X$ is a Borel $G$-anticlique, then $P_I\cap M\Vdash\dotxgen\notin\dot C$.
\end{claim}

\begin{proof}
Suppose that $C\subset X$ is a Borel set and $A\in P_I\cap M$ is a condition forcing $\dotxgen\in\dot C$; I must find a $G$-edge in $C^\gw$. Let $B=\{x\in A\colon\exists e\in G\cap A^\gw\ x\in\rng(e)\}$; by Claim~\ref{prepclaim3}, this is an analytic $I$-positive subset of $A$ in the model $M$. By Fact~\ref{millerfact}, it contains a Borel $I$-positive subset $B'$. Now, suppose that $H\subset P_I\cap M$ be a generic filter containing the condition $B'$ and $x\in X$ the associated generic point. Then $x\in B$, and so there is an edge $e\in G\cap A^\gw$ such that $x$ is in its range. By the initial assumptions on the hypergraph $G$, the range of the edge $e$ is contained in the orbit of $x$, and so each point in the range of $e$ is also generic for the poset $P_I$. Since the points on the range of $e$ all belong to the set $A$, by the forcing theorem they must all belong to the set $C$. Thus, the Borel set $C$ fails to be a $G$-anticlique in the extension $V[H]$, and by the Mostowski absoluteness it cannot be a $G$-anticlique in the ground model either.
\end{proof}

\begin{claim}
\label{lesskeyclaim}
Let $A\in P_I\cap M$ is a condition, then the set $B=\{x\in A\colon x$ is $P_I$-generic over the model $M\}$ is Borel and $I$-positive.
\end{claim}

\begin{proof}
The Borelness of the set of generics over countable models is a general fact, proved in \cite[Fact 1.4.8]{z:book2}. To see that $B\notin I$, suppose that $\{C_n\colon n\in\gw\}$ are Borel anticliques for some hypergraphs in $\cG$; I must produce a point $x\in B\setminus\bigcup_nC_n$. To this end, let $N$ be a countable elementary submodel of large structure containing $M, C_n$ for $n\in\gw$ and $\cG$ as elements, let $H\subset P_I\cap M$ be a filter generic over the model $N$ containing the condition $A$, and let $x$ be its associated generic real. By Claim~\ref{keyclaim} applied in the model $N$, $N[H]\models x\in B\setminus\bigcup_nC_n$, and by the Mostowski absoluteness between the models $N[H]$ and $V$, $x\in B\setminus\bigcup_nC_n$ holds as required.
\end{proof}

The theorem now immediately follows by the characterization of properness in \cite[Proposition 2.2.2]{z:book2}.
\end{proof}

\begin{example}
\label{silvergraphexample}
Let $G$ be the \emph{Silver graph} on $\cantor$, connecting points $x, y$ just in case they disagree on exactly one entry. The graph $G$ is invariant under the usual action of the rational points of the Cantor group $\cantor$ on the whole group. The quotient poset
is well known to have a dense subset naturally isomorphic to the Silver forcing \cite[Theorem 2.3.37]{z:book}.
\end{example}

\begin{example}
\label{vitaligraphexample}
Let $G$ be the \emph{Vitali graph} on $\cantor$, connecting points $x, y$ just in case they disagree on only finite number of entries. The graph $G$ is invariant under the usual action of the rational points of the Cantor group on the whole group. The quotient poset is well-known
to have a dense subset naturally isomorphic to the Vitali forcing, or $E_0$-forcing as it is called in \cite[Section 4.7.1]{z:book2}. The main difference between the Vitali and Silver forcings is that Vitali forcing adds no independent reals while Silver forcing does, even though below I identify another profound iterable difference--Corollary~\ref{neverclosedcorollary}.
\end{example}

\begin{example}
\label{kstexample}
Let $G$ be the \emph{KST graph} (for Kechris--Solecki--Todorcevic \cite{kechris:chromatic}) on $\cantor$. To define it, let $s_n\in 2^n$ be binary strings for each $n\in\gw$ such that the set $\{s_n\colon n\in\gw\}$ is dense in $\bintree$. Put $\langle x, y\rangle\in G$ if $x, y$ differ in exactly one entry and their longest common initial segment belongs to the set $\{s_n\colon n\in\gw\}$. It is well-known and easy to prove that the KST graph is closed, acyclic, and spans the Vitali equivalence relation. Borel anticliques of $G$ must be meager.
Let $\cG$ be the family of all rational shifts of the graph $G$. By the definitions, the family $\cG$ is actionable. The quotient poset does not depend on the initial choices  as proved in Subsection~\ref{kstsubsection} and I call it the KST forcing. The main iterable difference between the KST forcing and the Silver forcing is that in the Silver extension, ground model coded compact anticliques of any closed acyclic graph still cover their domain Polish space--Corollary~\ref{silveracycliccorollary}. In the KST extension this clearly fails for the initial graph $G$.
\end{example}

\begin{example}
Let $\Gamma$ be a countable group acting continuously on a Polish space $X$ and let $\mu$ be an invariant probability measure on $X$. The hypergraph $G$ consisting of all elements $e\in X^\gw$ such that the set $\rng(e)$ consists of pairwise orbit equivalent points and has $\mu$-positive closure is certainly invariant. Does the quotient forcing depend on the initial choice of the action and the measure?
\end{example}

\subsection{The canonical intermediate extension}

The generic extensions associated by the actionable posets share a certain important feature: there is a large, natural intermediate forcing extension. The main theorem of this section identifies the most important features of this extension.

\begin{theorem}
\label{intermediatetheorem}
Let $X$ be a Polish space and $I$ an actionable $\gs$-ideal on it; let $G\subset P_I$ be a generic filter. Then there is an intermediate extension $V\subset W\subset V[G]$ such that a set $a\in V[G]$ of ordinals belongs to $W$ just in case its intersection with every ground model countable set belongs to the ground model. Moreover, $V[G]$ is a c.c.c.\ extension of $W$.
\end{theorem}

\noindent The second sentence identifies the intermediate model uniquely; a moment's thought will show that $W$ must be the largest, in the sense of inclusion, intermediate model of ZFC with no new reals. The last sentence implies that $W$ is a nontrivial extension of the ground model except in the case that $P_I$ is equivalent to the Cohen forcing: if $W=V$ then $P_I$ is c.c.c.\ below some condition, and the only definable c.c.c.\ poset preserving Baire category is the Cohen forcing by \cite{Sh:630}. A good part of the theorem is the precise identification of the intermediate model $W$; the details of this have been incorporated into the proof rather than the statement of the theorem.

\begin{proof}
Start in the ground model $V$. Let $\cG$ be a family of analytic hypergraphs generating the $\gs$-ideal $I$, with the associated action of a countable group $\Gamma$ on $X$. Write $E$ for the resulting orbit equivalence relation on $X$; $E$ is Borel and all its classes are countable. Mover to the generic extension $V[G]$ and write $\xgen\in X$ for the generic point added by the filter $G\subset P_I$. The model $W$ is the class of all sets in $V[G]$ which such that every element of the transitive closure of $x$ is in $V[G]$ definable from parameters in $V$ and the additional parameter $[\xgen]_E$. It is well-known that classes of this type are models of ZFC. The following claim provides the central piece of information about the model $W$.

\begin{claim}
$\cantor\cap V=\cantor\cap W$.
\end{claim}

\begin{proof}
 Move to the ground model and let $B\in P_I$ be a condition, $\tau$ a $P_I$-name, and $B\Vdash\tau\in W$; I will find a condition $D\subset B$ and a point $z\in\cantor$ such that $D\Vdash\tau=\check z$. Strengthening the condition $W$, I may identify the parameters in $V$ and the formula $\phi$ which defines $\tau$ in $V[G]$ with those parameters plus the parameter $[\dotxgen]_E$. Strengthening the condition $B$ even further, I can find a Borel function $f\colon B\to\cantor$ such that $B\Vdash\tau=\dot f(\dotxgen)$. I will find a Borel $I$-positive set $C\subset B$ such that $f\restriction C$ is constant on $E$-classes, and then I argue that for every such a function there must be an $I$-positive Borel set $D\subset C$ on which $f$ is constant. Clearly, if $z\in\cantor$ is the constant value, then $D\Vdash\tau=\check z$ as required.

For the construction of the set $C$, let $M$ be a countable elementary submodel of a large structure containing all the objects named so far. Let $C\subset B$ be the set of all points $P_I$-generic over the model $M$; by Theorem~\ref{maintheorem}, this set is Borel and $I$-positive. To see that the function $f\restriction C$ is constant on $E$-classes, suppose that $x, y\in C$ are $E$-related points and $\gamma\in\Gamma$ is a group element such that $\gamma\cdot x=y$. Both points $x, y$ are $P_I$-generic over $M$ and so (identifying $M$ with its transitive collapse) I may consider the generic extensions $M[x]$, $M[y]$. Since $\gamma\cdot x=y$, these two generic extensions coincide. By the forcing theorem, the formula $\phi$ applied to the parameters and the $E$-class containing \emph{both} $x, y$ gives a point $z\in\cantor$ which must be equal to both $f(x)$ and $f(y)$. Thus $f(x)=f(y)$ as desired.

For the construction of the set $D$, suppose for contradiction that for every $z\in\cantor$, the set $\{x\in C\colon f(x)=z\}$ belongs to the ideal $I$.  Then the set $F\subset\cantor\times C$ of all pairs $\langle y, x\rangle$ such that $f(x)=y$ has all vertical sections in the ideal $I$. By Fact~\ref{millerfact}, there are Borel sets $F_n\subset\cantor\times B$ and hypergraphs $G_n\in\cG$ for each $n\in\gw$ such that $F\subset \bigcup_nF_n$ and each vertical section of the set $F_n$ is a $G_n$-anticlique. By the $\gs$-additivity of the $\gs$-ideal $I$, there is a number $n\in\gw$ such that the Borel set $A=\{x\in C\colon \langle f(x), x\rangle\in D_n\}$ is $I$-positive. Find a $G_n$-edge $e$ in the set $A$. Since the edge $e$ is a subset of a single $E$-class, the function $f$ is constant on the edge $e$ with value $y$, then a contradiction with the assumption that the section $(F_n)_y$ is $G_n$-anticlique appears. 
\end{proof}

It immediately follows that if $a\in V[G]$ is a set of ordinals such that for some countable set $b\in V$, $a\cap b\notin V$, then $a\notin W$: if $a$ were an element of $W$, so would be the set $a\cap b$, contradicting the claim. On the other hand, suppose that $a\in V[G]$ is a set of ordinals such that $a\cap b\in V$ for every countable set $b\in V$. Let $\tau$ be a $P_I$-name in the ground model such that $a=\tau/G$. In the model $V[G]$, the set $d=\{\tau/\gamma\cdot G\colon\gamma\in\Gamma\}$ is countable and definable from $\tau$ and the equivalence class $[\xgen]_E$. There is a countable set $c$ of ordinals such that whenever $a_0, a_1\in d$ are distinct sets of ordinals then they disagree on the membership of some ordinal in $c$. Since $V[G]$ is a proper extension of $V$, there is a countable set $b$ of ordinals in the ground model such that $c\subset b$. By the assumption on the set $a$, the intersection $e=a\cap b$ belongs to the ground model, and so $a$ can be defined from $[\dotxgen]_E, \tau, b, e$ as the only set of ordinals in $d$ whose intersection with $b$ is equal to $e$. Thus, $a\in W$ and the second sentence of the theorem is proved.

To evaluate the properties of the forcing that leads from $W$ to $V[G]$, one has to evaluate the forcing that leads from $V$ to $W$. This is a result of an entirely general procedure. Step back into the ground model. Let $Q_I\subset P_I$ be the poset of all $I$-positive Borel $E$-invariant sets, ordered by inclusion.

\begin{claim}
$G\cap Q_I$ is a filter on $Q_I$ generic over $V$.
\end{claim}

\begin{proof}
Move to the ground model. To prove the genericity, suppose that $B\in P_I$ is a condition and $C\subset [B]_E$ is a Borel $E$-invariant $I$-positive set. It will be enough to find a condition $D\subset B$ in $P_I$ such that its $E$-saturation is a subset of $C$. That way, it will be confirmed that the map $B\mapsto [B]_E$ is a projection of $P_I$ to $Q_I$, and the genericity follows.
\end{proof}

Now, the remainder poset  leading from the model $V[G\cap Q_I]$ to $V[G]$ is the poset $R$ of all Borel $I$-positive subsets $B\subset X$ coded in the ground model such that $[B]_E\in G\cap Q_I$. The following is easy to show via a genericity argument.

\begin{claim}
Let $x\in X$ be a point $E$-related to $\xgen$. The collection $H_x\subset R$ of all sets containing $x$ is a filter generic over $V[G\cap Q_I]$. Moreover, $R=\bigcup_xH_x$.
\end{claim}

It follows that $W=V[G\cap Q_I]$. On one hand, the filter $G\cap Q_I$ is definable in $V[G]$ from the parameter $[\xgen]_E$ as the set of all $E$-invariant Borel sets coded in the ground model which contain $[\xgen]_E$ as a subset. This shows that $V[G\cap Q_I]\subset W$. To prove the opposite inclusion, move to the model $V[G\cap Q_I]$. Suppose that $\tau$ is an $R$-name for a set of ordinals. Suppose that $B\in R$ is a condition forcing that $\tau$ is definable in the $R$-extension using a formula $\phi$ with parameters in $V$ and an additional parameter $[\dotxgen]_E$. It will be enough to show that the condition $B$ decides the membership of all ordinals in $\tau$, since then $B\Vdash\tau\in V[G\cap Q_I]$ as desired. To see this, note that the filters $H_x$ for $x\in [\xgen]_E$  in the claim cover the whole poset $R$ and since their generic points are all $E$-related, the filters all evaluate the formula $\phi$ and so the name $\tau$ in the same way.

It also follows that $W\models R$ is c.c.c.\ since the poset $R$ is covered by countably many filters in the extension $V[G]$ which has the same $\aleph_1$ as $W$. The theorem has just been proved.
\end{proof}

\subsection{Nearly open hypergraphs}

There is an entirely different class of hypergraphable proper forcings which is perhaps more frequent in the literature, even though never in the following most convenient presentation:

\begin{definition}
Let $G$ be an analytic hypergraph on a Polish space $X$. The hypergraph is \emph{nearly open} if for each edge $e\in G$ and every $i\in\dom(e)\setminus \{0\}$ there are open sets $O_i\subset X$ containing $e(i)$ such that for every choice of points
$x_i\in O_i$, the sequence $\langle e(0), x_i\colon i\in\dom(e)\setminus\{0\}\rangle$ is an edge in the hypergraph $G$. 
\end{definition}

It is clear from the definition that the first vertex in an edge of a nearly open hypergraph may carry priviledged information about the edge, and in the more interesting examples this is indeed the case. Note that I do not demand the hypergraphs to be invariant under the permutation of vertices in the edges. There is an important subclass of the nearly open hypergraphs, namely the hypergraphs which are open in the box topology on $X^\gw$. In these hypergraphs no vertex in an edge is clearly priviledged, and the treatment becomes
much simpler; in particular, the associated $\gs$-ideal is $\gs$-generated by closed sets.

Unlike the case of actionable ideals, the definition of a nearly open hypergraph depends on the choice of the topology on the underlying space $X$, and it may occur that one needs to make a rather unnatural change of topology to present a given hypergraph
as nearly open. Note that the poset $P_I$ and all its properties do not depend on the topology of $X$ as long as the algebra of Borel sets remains the same.

\begin{theorem}
Let $\cG$ be a family of nearly open hypergraphs on a Polish space $X$. The quotient poset $P_{I_{\cG}}$ is proper.
\end{theorem}

\begin{proof}
Write $I=I_{\cG}$.
Let $M$ be a countable elementary submodel of a large structure and consider the (countable) poset $P_I\cap M$.  As in the case of actionable ideals, the following claim will be central.

\begin{claim}
Let $G\in\cG$ be a hypergraph and $C\subset X$ be a Borel $G$-anticlique. Then $P_I\cap M\Vdash\dotxgen\notin\dot C$.
\end{claim}

\begin{proof}
Suppose on the contrary that some condition $B\in P_I\cap M$ forces $\dotxgen\in\dot C$. Shrinking the set $B$ if necessary, I may assume that for every open set $O\subset X$, $B\cap O=0$ or $B\cap O\notin I$.
The set $B'=\{x\in B\colon\lnot\exists e\in G\ e(0)=x\land\rng(e)\subset B\}$ is a coanalytic $G$-anticlique. Its complement $B\setminus B'$ cannot belong to $I$, since then it would be covered by a Borel
set $B''\in I$ and $B$ would be the union of $B''$ and the $G$-anticlique $B'\setminus B''$ and therefore in $I$. Now, by Theorem~\ref{millerfact}(2), the $I$-positive set $B\setminus B'$ has a Borel $I$-positive subset $D\in P_I\cap M$.
Now, in $V$ choose a sufficiently generic filter on the poset $P_I$ containing $D$ and let $x$ be its associated generic point. By the forcing theorem, it will be the case that $x\in D$. 

Now, since $x\in C$, there is an edge $e\in G$ such that $x=e(0)$ and $\rng(e)\subset B$.
Use the openness of the graph $G$ to find basic open sets $O_i\subset X$ for $i\in\dom(e)\setminus \{0\}$ such that $e(i)\in O_i$ and for every choice of points
$x_i\in O_i$, the sequence $\langle e(0), x_i\colon i\in\dom(e)\setminus\{0\}\rangle$ is an edge in the hypergraph $G$. For every $i\in \dom(e)\setminus\{0\}$, the set $B\cap O_i$ is nonempty, containing $e(i)$, and so it is $I$-positive. It also belongs to the model $M$. Thus, I can choose a sufficiently generic filter on the poset $P_I\cap M$ containing the set $B\cap O_i$
and let $x_i$ be its generic point. By the forcing theorem, $x_i\in D$ holds.

Finally, consider the tuple $\langle e(0), x_i\colon i\in\dom(e)\setminus\{0\}\rangle$. It is an edge in the hypergraph $G$, and it consists of points in the set $C$. This contradicts the assumption that $C$ was a $G$-anticlique.
\end{proof}

Let $B\in P_I$ be a condition in the model $M$. I must show that the set $\{x\in B\colon x$ is $P_I$-generic over the model $M\}$ is $I$-positive. To this end, let $C_n$ for $n\in\gw$ be Borel subsets of $X$ which are anticliques in at least one hypergraph
on the generating family $\cG$. I must produce a point $x\in B$ which is $P_I$-generic over $M$ and does not belong to the set $\bigcup_nC_n$. To do this, choose a sufficiently generic filter on the poset $P_I\cap M$ containing the set $B$, and let $x\in X$ be its generic point.
Clearly, $x\in B$, and by the claim $x\notin\bigcup_nC_n$ as desired. The proof is complete!
\end{proof}

\begin{example}
If $I$ is a \pioneoneonsigmaoneone\ $\gs$-ideal on a Polish space $X$ generated by closed sets, then there is a nearly open and in fact box open hypergraph $G$ on $X$ such that $I=I_G$. Just let $e\in G$ just in case the closure of $\rng(e)$ does not belong to $I$.
\end{example}

\begin{example}
Let $X, Y$ be Polish spaces and $f\colon X\to Y$ be a Borel function. Let $G$ be the hypergraph on $X$ consisting of those tuples $e\in X^\gw$ such that $e(0)=\lim_ie(i)$ but there is an open set $O\subset Y$ containing $f(e(0))$ and no other points of $f''\rng(e)$.
The hypergraph is not nearly open as it stands, but becomes nearly open if the topology of the space $X$ is updated to make the function $f$ continuous. The $\gs$-ideal $I_G$ is generated by Borel subsets of $X$ on which the function $f$ is continuous (with the original topology on the space $X$).
\end{example}

The $\gs$-ideals generated by nearly open graphs differ from the actionable ideals in many significant ways. One remarkable difference is Theorem~\ref{coincidencetheorem} proved below, which provides for many iterable preservation properties
of the quotient forcings of the nearly open hypergraphs that the actionable ideals can never have.

\subsection{Limitations}

 In the only result in this subsection, I isolate a rather humble class of hypergraphable posets which are not proper.

\begin{theorem}
\label{nonpropertheorem}
Let $G$ be a locally countable, acyclic, analytic graph on a Polish space $X$ and let $I$ be the $\gs$-ideal $\gs$-generated by Borel $G$-anticliques. The quotient forcing $P_I$ is either trivial or not proper.
\end{theorem}

\begin{proof}
Let $E$ be the equivalence relation connecting points $x, y\in X$ if there is a $G$-path from $x$ to $y$. By the second reflection theorem, $E$ is a subset of a countable Borel equivalence relation $F$. By the Feldman--Moore theorem, there is a Borel action of some countable group $\Gamma$ which induces $F$ as the orbit equivalence relation. Consider the following a small claim of independent interest which does not use the acyclicity assumption:

\begin{claim}
If $B\in P_I$ is a Borel $I$-positive set, then there is a Borel $I$-positive set $C\subset B$ such that two elements of $C$ are connected by a $G$-path if and only if they are connected by a $G$-path whose vertices are all in $C$.
\end{claim}

\begin{proof}
Since the Borel chromatic number of $G$ on $B$ is uncountable, by a result of Miller \cite{lecomte:basis} there is a continuous map $h\colon\cantor\to B$ which is a homomorphism of the KST graph to $G$ and antihomomorphism of the Vitali equivalence relation to $F$.
Since the KST graph spans the Vitali equivalence relation, the compact set $C=\rng(h)$ is a witness to the validity of the claim.
\end{proof}

Now, let $M$ be a countable elementary submodel of a large structure containing $G, X$, and the action. By the properness criterion \cite[Proposition 2.2.2]{z:book2}, it will be enough to show that the set $C=\{x\in X\colon x$ is $P_I$-generic over $M\}$ is a $G$-anticlique.
Suppose towards contradiction that there are points $x, y\in C$ which are $G$-connected. Let $\gamma\in\Gamma$ be an element such that $\gamma\cdot x=y$. By the forcing theorem applied in the model $M$, there must be a Borel set $B\in P_I\cap M$
which forces $\gamma\cdot\dotxgen$ to be $P_I$-generic and $G$-connected to $\dotxgen$. Now, thinning down the set $B$ successively, I may arrange that for each $x\in B$, $\langle x, \gamma\cdot x\rangle\in G$, $\gamma\cdot B\cap B=0$, and finally, by the claim,
that two elements of $B$ are connected by a $G$-path if and only if they are connected by a $G$-path whose vertices are all in $B$. 

Look at the Borel set $\gamma\cdot B$. Since $B\Vdash\gamma\cdot\dotxgen$ is $P_I$-generic, the Borel set $\gamma\cdot B$ cannot be a $G$-anticlique. Thus, there are points $x\neq y\in B$ such that $\langle \gamma\cdot x, \gamma\cdot y\rangle\in G$.
Thus, $x, \gamma\cdot x, \gamma\cdot y, y$ is a $G$-path from $x$ to $y$. By the choice of the set $B$, there must be also such a path using exclusively points from the set $B$. The two paths together form a cycle in the graph $G$, which is a contradiction.
\end{proof}

\section{Fubini-type preservation properties}
\label{fubinisection}

In this section, I provide several results which, from simple combinatorial properties of the collection of hypergraphs $\cG$, obtain central forcing properties of the quotient posets. The theorems are proved through the Fubini property with various
$\gs$-ideals generated by Suslin ergodic forcings. This approach is in an important class of cases optimal (Theorem~\ref{fubinitheorem}) and has the advantage of automatically yielding preservation theorems for countable support iteration and product (Corollary~\ref{aacorollary} and~\ref{bbcorollary}). Recall the central definition:

\begin{definition}
\textnormal{\cite[Definition 3.2.1]{z:book2}}
Let $I, J$ be $\gs$-ideals on respective Polish spaces $X, Y$. The ideals have the \emph{Fubini property} ($I\not\perp J$) if for every Borel $I$-positive set $B_X\subset X$, every Borel $J$-positive set $B_Y\subset Y$,
and every Borel set $C\subset B_X\times B_Y$, either $C$ has a $J$-positive vertical section or the complement of $C$ has a horizontal $I$-positive section.
\end{definition}

\noindent The Fubini property is always going to be verified through the following property of Suslin partial orders and the attendant theorem:

\begin{definition}
\label{perpdefinition}
Let $\cG$ be an analytic family of hypergraphs on a Polish space $X$. Let $P$ be a Suslin forcing. The notation $\cG\not\perp P$ denotes the following  statement: for every Borel $I_{\cG}$-positive set $B\subset X$, every Borel function $h\colon B\to P$, and every hypergraph $G\in\cG$, there is a condition $p\in P$ which forces that there is an edge $e\in G$ which consists of ground model elements of the set $B$, and $h''\rng(e)$ is a subset of the generic filter.
\end{definition}

\noindent One good way to satisfy the statement $\cG\not\perp P$ is to actually find an edge $e\in G$ such that the set $h''\rng(e)$ has a lower bound in the poset $P$, which then will serve as the condition $p$. This is clearly the only way to act if the edges in the hypergraph $G$ are finite. However, in the case of hypergraphs of infinite arity, this approach may not be flexible enough, and the sought edge $e\in G$ may be found only in the $P$-extension.

The following theorem is one of the reasons why hypergraphable forcings are so special. To state it, if $P$ is a forcing, $Y$ is a Polish space and $\dot y$ is a $P$-name for an element of $Y$, write $J_{\dot y}$ for the $\gs$-ideal generated by the analytic sets $A\subset Y$ such that $P\Vdash\dot y\notin\dot A$.

\begin{theorem}
\label{keytheorem}
Let $\cG$ be an analytic family of hypergraphs on a Polish space $X$ such that the quotient forcing $P_{I_\cG}$ is proper. Let $P$ be a Suslin c.c.c.\ forcing, $Y$ a Polish space, $\dot y$ a $P$-name for an element of $Y$. Then, $\cG\not\perp P$ implies $I_{\cG}\not\perp J_{\dot y}$. 
\end{theorem}

\begin{proof}
Suppose that $\cG\not\perp P$ holds. To prove the theorem, suppose that $B_X\subset X$ and $B_Y\subset Y$ are Borel $I_\cG$ and $J_{\dot y}$-positive sets respectively and $C\subset B_X\times B_Y$ is a Borel set with all vertical sections in the ideal $J_{\dot y}$. I must prove that the complement $D=(B_X\times B_Y)\setminus C$ has an $I_{\cG}$-positive horizontal section.

Suppose for contradiction that this fails. Then, by Theorem~\ref{millerfact}(3), there are Borel sets $D_n\subset D$ and hypergraphs $G_n\in\cG$ for $n\in\gw$ such that $D=\bigcup_nD_n$ and each vertical section of $D_n$ is a $G_n$-anticlique. By Theorem~\ref{uniformizationtheorem}, there is a Borel $I_{\cG}$-positive set $B\subset B_X$, a natural number $n\in\gw$, and a Borel function $h\colon B\to P$ such that for all $x\in B$, the condition $h(x)\in P$ forces $\langle\check x, \dot y\rangle\in \dot D_n$. 

Now, use the assumption $\cG\not\perp P$ to find a condition $p\in P$ and a $P$-name $\dot e$ such that $p\Vdash\dot e\in\dot G_n$, $\rng(\dot e)$ consists of ground model elements of $\dot B$, and for each $x\in\rng(\dot e)$, $h(x)$ belongs to the generic filter. Let $M$ be a countable elementary submodel of a large structure containing all objects named so far, and let $H\subset P\cap M$ be a filter $P$-generic over $M$, and let $y=\dot y/H$ and $e=\dot e/H$. By the forcing theorem, $M[H]$ satisfies $e\in G_n$ and $\forall x\in\rng(e)\ \langle x, y\rangle\in D_n$. By the Mostowski absoluteness, these statements transfer from $M[H]$ to $V$ and so $e$ is a $G_n$-edge consisting of points in the horizontal section $(D_n)^y$. This contradicts the choice of the set $D_n$.
\end{proof}

\subsection{The localization property}

The quotient posets arising from families of analytic hypergraphs in which every edge is finite with a fixed bound on its arity share a strong preservation property.

\begin{definition}
A Suslin poset $P$ is \emph{analytic $\gs$-$n$-linked} if it can be covered by the union of countably many analytic sets such that if $\{p_m\colon m\in n\}$ are conditions in one of the analytic sets, then they have a common lower bound.
\end{definition}

\begin{theorem}
\label{localizationtheorem} 
Suppose that $n\in\gw$ is a number and $\cG$ is a countable family of analytic hypergraphs on a Polish space $X$ such that each hypergraph in $\cG$ has arity $\leq n$. Suppose that $P$ is a Suslin poset which is analytic $\gs$-$n$-linked. Then
$\cG\not\perp P$.
\end{theorem}

\begin{proof}
Let $P=\bigcup_mP_m$ be a countable cover of the poset $P$ by analytic $n$-linked pieces.
Suppose $B\subset X$ is a Borel $I_{\cG}$-positive set, $h\colon B\to P$ is a Borel function, and $G\in \cG$ is a hypergraph. 
For each number $m\in\gw$, the preimage $h^{-1}P_m\subset B$ is analytic. If it is a $G$-anticlique, by the first reflection theorem it can be extended to a Borel $G$-anticlique $B_m\subset B$. If this occurred for each number $m\in\gw$, I would have $B=\bigcup_mB_m$, contradicting the assumption that $B\notin\cG$. Thus, there is a number $m\in\gw$ and an edge $e\in G$ such that $h''\rng(e)\subset P_m$. Since the set $P_m$ is $n$-linked and the set $\rng(e)$ has size at most $n$, there is a lower bound $p\in P$ of $h''\rng(e)$. This lower bound witnesses the statement ${\cG}\not\perp P$.
\end{proof}

\noindent Among the corollaries of the theorem, one appears to be exceptionally powerful:

\begin{definition}
\textnormal{\cite{roslanowski:localization}}
Let $m$ be a natural number. A poset $P$ has the $m$-localization property if every element of $\baire$ in the $P$-extension is a branch through some $\leq m$-branching tree in the ground model.
\end{definition}

\begin{corollary}
\label{localizationcorollary}
Suppose that $n\in\gw$ is a number and $\cG$ is a countable family of analytic hypergraphs on a Polish space $X$ such that each hypergraph in $\cG$ has arity $\leq n$. If the poset $P_{I_\cG}$ is proper, then it has the $n$-localization property.
\end{corollary}

\begin{proof}
I will need a poset adding a large $n$-branching tree. Let $P$ be the partial order of all pairs $p=\langle t_p, a_p\rangle$ where $t_p\subset\gwtree$ is a finite $n$-branching tree, $a_p\subset\baire$ is a finite set, every element of $a_p$
contains a terminal node of $t_p$ as an initial segment, and every terminal node of $t_p$ is an initial segment of at most one element of $a_p$. The ordering is defined by $q\leq p$ if $t_q$ is an end-extension of $t_p$ and $a_p\subset a_q$. 
The poset $P$ is Suslin c.c.c., and in fact any $n$ many conditions sharing the same first coordinate have a lower bound. This also proves that the poset is analytic $\gs$-$n$-linked. 

The poset $P$ adds a generic $n$-branching tree $\dottgen$ which is the union of the first coordinates of conditions in the generic filter. A simple genericity argument shows that every ground model element $z\in\baire$ has a finite modification which is a branch through $\dottgen$. Let $Y$ be a the space of all $n$-branching trees on $\gw$, and let $J$ be the $\gs$-ideal generated by those analytic sets $A\subset Y$ such that $P\Vdash\dottgen\notin \dot A$. 

Suppose that $B\in P_{I_\cG}$ is a condition and $\tau$ is a $P_{I_\cG}$-name
for an element of $\baire$. By the Borel reading of names in proper forcing, it is possible to thin out the condition $B$ to find a Borel function $f\colon B\to\baire$ such that $B\Vdash\tau=\dot f(\dotxgen)$. Let $C\subset B\times Y$ be the set of all pairs $\langle x, T\rangle$ such that no finite modificaton of $f(x)$ is a branch through the tree $T$. The vertical sections of $C$ belong to the $\gs$-ideal $J$. By Theorems~\ref{localizationtheorem} and~\ref{keytheorem}, there has to be an $n$-branching tree $T$ such that the horizontal section of the complement of the set $C$ associated with $T$,
the Borel set $B'=\{x\in B\colon$ a finite modification of $f(x)$ is a branch through the tree $T\}$ is $I_{\cG}$-positive. The condition $B'$ forces some finite modification of $\tau$ to be a branch through $T$, and so $\tau$ to be a branch to some finite modification of the tree $T$.
\end{proof}

\begin{example}
For a fixed number $n\in\gw$ consider the hypergraph $G_n$ of arity $n$ on the space $n^\gw$ which contains exactly all tuples $\{x_i\colon i\in n\}$ such that for some $m\in\gw$, $x_i(m)=i$ while all functions $x_i$ are identical on the set $\gw\setminus\{m\}$. The resulting poset $P_n$ is the wider version of the Silver forcing. It does have the $n$-localization property, while it fails to have the $n-1$-localization property.
\end{example}

\begin{example}
Let $G_0$ be the graph on $X=\cantor$ connecting points $x\neq y$ if the smallest $n\in\gw$ such that $x(n)\neq y(n)$ is even. Let $G_1$ be the graph on $X=\cantor$ connecting points $x\neq y$ if the smallest $n\in\gw$ such that $x(n)\neq y(n)$ is odd, and let $\cG=\{G_0, G_1\}$. The resulting forcing is the $\cmin$ forcing studied by Kojman among others \cite[Section 4.1.5]{z:book2}, \cite{geschke:coloring}. Since the two graphs are open, the poset is proper, and by Corollary~\ref{localizationcorollary} it has the $2$-localization property.
\end{example}

\subsection{The Sacks property}

If one drops the demand that there be a uniform bound on the arities of edges in all hypergraphs in the generating collection, while each hypergraphs contains only edges of finite and bounded arity, then the resulting poset  still maintains
strong preservation properties.

\begin{theorem}
\label{sackstheorem}
Suppose that $\cG$ is a countable family of analytic hypergraphs on a Polish space $X$ such that for each $G\in\cG$ there is $n\in\gw$ such that $G$ has arity $\leq n$. and $I$ is a hypergraphable $\gs$-ideal on $X$. Suppose that $P$ is a Suslin poset which is analytic $\gs$-$n$-centered for every $n$. Then $\cG\not\perp P$. 
\end{theorem}

\begin{proof}
Suppose $B\subset X$ is a Borel $I_{\cG}$-positive set, $h\colon B\to P$ is a Borel function, and $G\in \cG$ is a hypergraph. 
Let $n\in\gw$ be the arity of $G$, and let $P=\bigcup_mP_m$ be a countable cover of the poset $P$ by analytic $n$-linked pieces.
For each number $m\in\gw$, the preimage $h^{-1}P_m\subset B$ is analytic. If it is a $G$-anticlique, by the first reflection theorem it can be extended to a Borel $G$-anticlique $B_m\subset B$. If this occurred for each number $m\in\gw$, I would have $B=\bigcup_mB_m$, contradicting the assumption that $B\notin\cG$. Thus, there is a number $m\in\gw$ and an edge $e\in G$ such that $h''\rng(e)\subset P_m$. Since the set $P_m$ is $n$-linked and the set $\rng(e)$ has size at most $n$, there is a lower bound $p\in P$ of $h''\rng(e)$. This lower bound witnesses the statement ${\cG}\not\perp P$.
\end{proof}

Among the corollaries of the theorem, one again appears to be dominant. The requisite definitions:

\begin{definition}
An \emph{narrow tunnel} is a function $T\colon\gw\to\power(\gw)$ such that for all $n\in\gw$, $|f(n)|\leq 2^n$. A point $x\in\baire$ is \emph{enclosed} by the tunnel $T$ if for all $n\in\gw$, $x(n)\in T(n)$.
A poset $P$ has the \emph{Sacks property} if every point in $\baire$ in the $P$-extension enclosed by a ground model narrow tunnel.
\end{definition}

\begin{corollary}
\label{sackscorollary}
All the hypergraphable posets in the class identified in Theorem~\ref{sackstheorem} have the Sacks property.
\end{corollary}

\begin{proof}
I will need a Suslin poset which adds a generic narrow tunnel. Let $P$ be the set of all pairs $p=\langle o_p, a_p\rangle$ such that $o_p$ is a function with domain $n_p\in\gw$ such that for all $m\in n_p$, $o_p(m)\subset\gw$ is a set of size $\leq 2^m$.
Also, $a_p\subset\baire$ is a set of size $\leq 2^{n_p}$. The ordering is defined by $q\leq p$ if $o_p\subset o_q$, $a_p\subset a_q$, and for all $m\in n_q\setminus n_p$ and all $x\in a_p$, $x(m)\in o_q(m)$ holds.  It is not difficult to see that the forcing $P$ is c.c.c.\ and in fact analytic $n$-centered for every $n\in\gw$.

The forcing adds a generic narrow tunnel $\dottgen$ which is the union of the first coordinates of all conditions in the generic filter. A genericity argument shows that every ground model element of $\baire$ has a finite modification which is enclosed by the generic tunnel. 
 Let $Y$ be a the space of all narrow tunnels, and let $J$ be the $\gs$-ideal generated by those analytic sets $A\subset Y$ such that $P\Vdash\dottgen\notin \dot A$. 

Suppose that $B\in P_{I_\cG}$ is a condition and $\tau$ is a $P_{I_\cG}$-name
for an element of $\baire$. By the Borel reading of names in proper forcing, it is possible to thin out the condition $B$ to find a Borel function $f\colon B\to\baire$ such that $B\Vdash\tau=\dot f(\dotxgen)$. Let $C\subset B\times Y$ be the set of all pairs $\langle x, T\rangle$ such that no finite modification of $f(x)$ is not enclosed by the tunnel $T$. The vertical sections of $C$ belong to the $\gs$-ideal $J$. By Theorems~\ref{sackstheorem} and~\ref{keytheorem}, there has to be a narrow tunnel $T\in Y$ such that the horizontal section of the complement of the set $C$ associated with $T$,
the Borel set $B'=\{x\in B\colon$ a finite modification of $f(x)$ is enclosed by $T\}$ is $I_{\cG}$-positive. The condition $B'$ forces some finite modification of $\tau$ to be enclosed by $T$, and so $\tau$ to be enclosed by some finite modification of the tunnel $T$.
\end{proof}

\begin{example}
Let $X=\cantor$ and $\cG=\{G_n\colon n\in\gw\}$ be the family of analytic hypergraphs on $X$ in which a $G_n$-edge is a set $a\subset\cantor$ such that for some interval $i\subset\gw$ of length $n$, the sequences in $a$ are equal off $i$ while every function in $2^i$ is a subset of some element of $a$. Then $\cG$ is actionable as witnessed by the standard group action. The associated poset is the poset of all partial functions from $\gw$ to $2$ whose codomain contains arbitrarily long intervals, ordered by reverse inclusion.
\end{example}

Note that in the previous example there is no number $n\in\gw$ such that all hypergraphs in the generating family have arity $\leq n$. Thus, the conclusion is that the poset has the Sacks property by Corollary~\ref{sackscorollary}, but does not have the $n$-localization property for any number $n$. To see the failure of the $n$-localization property, consider the function $g_n$ in the generic extension which to each $i\in\gw$ assigns the binary string $g_n(i)\in 2^n$ such that for each $j\in n$, the value of the generic point at $i\cdot n+j$ is equal to $g_n(i)(j)$.

\subsection{The weak Sacks property}

If one further loosen the demands on the hypergraphs by allowing each of the hypergraphs to contain edges of all finite arities, the Sacks property may fail. Instead, I obtain the following.

\begin{definition}
Let $P$ be a Suslin poset. $P$ is \emph{analytic $\gs$-centered} if it can be written as a union of countably many analytic centered pieces.
\end{definition}

\begin{theorem}
\label{weaksackstheorem}
Suppose that $\cG$ is a countable family of analytic hypergraphs on a Polish space $X$ such that all of the edges of all hypergraphs in $\cG$ are finite. Suppose that $P$ is a Suslin poset which is analytic $\gs$-centered. Then $\cG\not\perp P$ holds. 
\end{theorem}

\begin{proof}
Let $P=\bigcup_mP_m$ be a countable cover of the poset $P$ by analytic centered pieces. Suppose $B\subset X$ is a Borel $I_{\cG}$-positive set, $h\colon B\to P$ is a Borel function, and $G\in \cG$ is a hypergraph. 
For each number $m\in\gw$, the preimage $h^{-1}P_m\subset B$ is analytic. If it is a $G$-anticlique, by the first reflection theorem it can be extended to a Borel $G$-anticlique $B_m\subset B$. If this occurred for each number $m\in\gw$, I would have $B=\bigcup_mB_m$, contradicting the assumption that $B\notin\cG$. Thus, there is a number $m\in\gw$ and an edge $e\in G$ such that $h''\rng(e)\subset P_m$. Since the set $P_m$ is centered and the set $\rng(e)$ is finite, there is a lower bound $p\in P$ of $h''\rng(e)$. This lower bound witnesses the statement $I_{\cG}\not\perp P$.
\end{proof}

Again, one of the consequences of the theorem seems to be dominant. For the definitions,

\begin{definition}
A point $x\in\baire$ is \emph{enclosed} by a narrow tunnel $T$  on an infinite set $b\subset\gw$ if for all  $n\in b$, $x(n)\in T(n)$.
A poset $P$ has the \emph{weak Sacks property} if every point in $\baire$ in the $P$-extension enclosed by a ground model narrow tunnel on a ground model infinite set.
\end{definition}

\begin{corollary}
\label{weaksackscorollary}
The hypergraphable posets in the class of Theorem~\ref{weaksackstheorem} have the weak Sacks property.
\end{corollary}

\begin{proof} 
I will need a Suslin poset which adds a generic narrow tunnel and an infinite set. Let $P$ be the set of all triples $p=\langle o_p, a_p, b_p\rangle$ such that $o_p$ is a function with domain $n_p\in\gw$ such that for all $m\in n_p$, $o_p(m)\subset\gw$ is a set of size $\leq 2^m$. Also $b_p\subset n_p$ and $a_p\subset\baire$ is a finite set. The ordering is defined by $q\leq p$ if $o_p\subset o_q$, $a_p\subset a_q$, $b_p=b_q\cap n_p$ and for all $m\in b_q\setminus b_p$ and all $x\in a_p$, $x(m)\in o_q(m)$ holds.It is not difficult to see that the forcing $P$ is analytic $\gs$-centered; in fact, any finite set of conditions with the same first and third coordinate has a lower bound.

 The forcing $P$ adds a generic narrow tunnel $\dottgen$ which is the union of the first coordinates of all conditions in the generic filter, and a generic infinite set $\dotbgen\subset\gw$ which is the union of the third coordinates of the conditions in the generic filter. A genericity argument shows that the  every ground model element of $\baire$ has a finite modification which is enclosed bythe tunel $\dottgen$ on the set $\dotbgen$.  
Let $Y$ be the space of all pairs $\langle T, b\rangle$ where $T$ is a narrow tunnel and $b\subset\gw$ is an infinite set, and let $J$ be the $\gs$-ideal of all analytic set $A\subset A$ such that $P\Vdash\langle\dottgen, \dotbgen\rangle\notin\dot A$.

To prove the corollary, suppose that $B\in P_I$ is a condition and $\tau$ is a $P_I$-name for an element of $\baire$. Thinning the condition $B$ if necessary, I may find a Borel function $h\colon B\to\baire$ such that $B\Vdash\dot h(\dotxgen)=\tau$. Let $C\subset B\times Y$ be the set
of all triples $\langle x, T, b\rangle$ such that no finite modification of $h(x)\in\baire$ is enclosed by the tunnel $T$ on the set $b$. Then the vertical sections of the Borel set $C$ are $J$-small. By Theorems~\ref{sackstheorem} and~\ref{keytheorem}, there must be an $I$-positive horizontal section of the complement.
This horizontal section, call it $B'$, is attached to a pair $\langle T, b\rangle\in Y$. Clearly, $B'\Vdash$ a finite modification of $\tau$ is enclosed by the tunnel $T$ on the set $b$, and so $\tau$ is enclosed by the tunnel $T$ on a finite modification of the set $b$.
\end{proof}

It is easy to see that the weak Sacks property implies the bounding property. Thus, I can conclude

\begin{corollary}
\label{finitarycorollary}
Let $I$ be a hypergraphable $\gs$-ideal on a Polish space $X$. If every edge of every hypergraph in the generating family is finite and the quotient poset is proper, then the quotient poset is bounding.
\end{corollary}

The following corollary uses an entirely different $\gs$-centered poset, derived from the classical Solovay coding.

\begin{corollary}
\label{splittingcorollary}
Let $I$ be a hypergraphable $\gs$-ideal on a Polish space $X$ such that every edge of every hypergraph in the generating family is finite and the quotient poset is proper. In the $P_I$-extension, for every infinite set $b\subset Z$ there is a ground model coded Borel set $D\subset Z$ such that both sets $b\cap D$, $b\setminus D$ are infinite.
\end{corollary}

\begin{proof}
I need the requisite Suslin poset. Fix a Borel function $h\colon Z\to\power(\gw)$ such that the range of $h$ consists of pairwise almost disjoint infinite subsets of $\gw$. The poset $P$ can be viewed as an iteration of adding a Cohen generic function from $Z$ to $2$ and then coding the resulting partition of $Z$ into a Borel set. More specifically, a condition $p\in P$ is a pair $\langle f_p, a_p\rangle$ such that $f_p$ is a finite function from $Z$ to $2$ and $a_p\subset\gw$ is a finite set. The ordering is defined by $q\leq p$ if $f_p\subset f_q$, $a_p\subset a_q$, and for all $n\in a_q\setminus a_p$ and all $z\in\dom(f_p)$ such that $f_p(z)=0$, $n\notin h(z)$. 

The poset $P$ is clearly Suslin. It is also analytic $\gs$-centered. To see this, if $a\subset\gw$ is finite, $o$ is a finite collection of pairwise disjoint basic open subsets of $Z$ and $f\colon o\to 2$ is a function, consider the set $A\subset P$ of those conditions $p$ such that
$a=a_p$, $\dom(f_p)\subset\bigcup o$ and for every $O\in o$ and every $x\in O\cap\dom(f_p)$ $f_p(x)=f(O)$ holds. It is not difficult to see that the set $A$ is analytic and centered. Moreover, the poset $P$ is covered by the countably many sets $A\subset P$ obtained in this way. 

Let $\dot y$ be the $P$-name for the union of the second coordinates of the conditions in the generic filter. A simple genericity argument shows that for every infinite set $b\subset Z$, the set $\{z\in b\colon \dot y\cap h(z)$ is infinite$\}$ is forced to be infinite co-infinite in $b$.
Let $Y=\power(\gw)$ and let $J$ be the $\gs$-ideal of all analytic sets $A\subset Y$ such that $P\Vdash\dot y\notin\dot A$.

To prove the corollary, suppose that $B\in P_I$ is a condition and  $B\Vdash\tau\subset\dot Z$ is an infinite set. Shrinking $\tau$ to a countable set and thinning out $B$ if necessary, it is possible to find Borel functions $g\colon B\to X^\gw$ such that
$B\Vdash\dot g(\dotxgen)$ is an injective enumeration of the set $\tau$. Let $C\subset B\times Y$ be the set of those pairs $\langle x, y\rangle$ for which the set $\{n\in\gw\colon y\cap h(g(x)(n))$ is infinite$\}$ is either finite or cofinite. The vertical
sections of the set $C$ belong to the ideal $J$, and so by Theorems~\ref{weaksackstheorem} and ~\ref{keytheorem} there is $y\in Y$ such that the horizontal section $B'\subset B$ of the complement of $C$ corresponding to $y$ is $I$-positive.
Let $D=\{z\in Z\colon y\cap h(z)$ is infinite$\}$ and observe that $B'\Vdash \tau\cap D$ and $\tau\setminus D$ are both infinite sets.
\end{proof}

\begin{example}
Let $f\in\baire$ be a function and let $X\subset\baire$ be the closed subset of all functions pointwise $\leq f$. Let $G\subset X^{<\gw}$ be the graph of all tuples $\langle x_i\colon i\in m\rangle$ such that there is $n\in\gw$ such that the functions $x_i$ agree everywhere except on $n$, and the set $\{x_i(n)\colon i\in m\}$ contains all numbers $\leq f(n)$. The hypergraph $G$ is actionable as witnessed by the group $\Prod_nf(n)$ of finite support product of the cyclic groups of size $f(n)$. Corollary~\ref{finitarycorollary} shows that the quotient poset is bounding. It is the wide version of the Silver forcing introduced in \cite[Definition 7.4.11]{bartoszynski:set}. Note that the hypergraph has edges of arbitrarily large finite sizes, so Theorem~\ref{sackstheorem} does not apply and indeed, the poset does not have the Sacks property as is immediately obvious from looking at its natural generic point.
\end{example}

\begin{example}
Another family of hypergraphs with finite edges only, whose quotient poset fails to have the Sacks property, is associated with packing measures and it is discussed in Example~\ref{packingexample}.
\end{example}

\subsection{The bounding property}

It turns out that there is a useful criterion for the hypergraphs to generate a bounding quotient forcing.

\begin{definition}
Let $G$ be an analytic hypergraph on a Polish space $X$. Call $G$ \emph{diagonalizable} if for every collection $\{d_n\colon n\in\gw\}$ of edges in $G$ such that the points $d_n(0)\in X$ are identical for all $n\in\gw$, there are finite sets $a_n\subset\rng(e_n)$ and an edge $e\in G$ such that $\rng(e)\subset\bigcup_na_n$.
\end{definition}

\noindent Note that I do not expect the hypergraph $G$ to be closed under permutation of vertices in the edges. Thus, $G$ can be formulated so that the first vertex $e(0)$ of an edge $e$ carries some specific information--for example, it may be the unique cluster point of $\rng(e)$ etc. This instrumental Suslin forcing in teis section is the \emph{Hechler forcing} $P$. This is the set of all pairs $p=\langle t_p, z_p\rangle$ such that $t_p\in\gwtree$ and $z_p\in\baire$, and the ordering is defined by $q\leq p$ if $t_p\subset t_q$, $z_q$ is larger than $z_p$ on all entries, and for all $n\in\dom(t_q\setminus t_p)$, $t_q(n)>z_p(n)$ holds. It is a well-known Suslin c.c.c.\ poset. It is also optimal for its designed purpose by the following

\begin{theorem}
\label{hechleroptimaltheorem}
Let $\cG$ be a countable family of  analytic hypergraphs on a Polish space $X$ such that the quotient forcing $P_{I_\cG}$ is proper. The following are equivalent:

\begin{enumerate}
\item $\cG\not\perp P$;
\item $P_{I_{\cG}}$ is bounding.
\end{enumerate}
\end{theorem}

\begin{proof}
To see that (1) implies (2),
Let $\dot y$ be the Hechler name for the union of the first coordinates of conditions in the generic filter; thus $\dot y$ is forced to be an element of $\baire$ which modulo finite dominates all ground model elements of $\baire$. Let $J$ be the $\gs$-ideal on $\baire$ generated by the analytic sets $A\subset\baire$ such that $P\Vdash\dot y\notin\dot A$. Theorems~\ref{boundingtheorem} and~\ref{keytheorem} imply that $I_{\cG}\not\perp J$ holds.

Now, suppose that $B\subset X$ is a Borel $I_{\cG}$-positive set and $\tau$ a $P_{I_{\cG}}$-name for an element of $\baire$. By the Borel reading of names in proper idealized forcing, thinning the set $B$ out if necessary one may find a Borel function $f\colon B\to\baire$ such that $B\Vdash\tau=\dot f(\dotxgen)$. Let $C\subset B\times\baire$ be the Borel set $\{\langle x, y\rangle\colon f(x)$ is not modulo finite dominated by $y\}$. The vertical sections of the set $C$ belong to the $\gs$-ideal $J$, and so there must be $y\in\baire$ such that the horizontal section $B'$ of the complement of $C$ indexed by $y$ is $I_{\cG}$-positive. It is immediate that the condition $B'$ forces $\tau$ to be modulo finite dominated by the function $y$.

To see why (2) implies (1), suppose that the quotient poset is bounding, $B\subset X$ is a Borel $I_{\cG}$-positive set, $f\colon B\to P$ is a Borel function and $G\in \cG$ is a hypergraph. Thinning out the set $B$ if necessary, I may assume that the first coordinates of the conditions $f(x)$ for $x\in B$ are all equal to some fixed $t\in\gwtree$, and (using the bounding property) the second coordinates are all bounded by some function $g\in\baire$. Let $e\in G$ be any edge all of whose vertices belong to the set $B$. Then the condition
$\langle t, g\rangle\in P$ forces $f''\rng(e)$ to be a subset of the generic filter as required.
\end{proof}

\begin{theorem}
\label{boundingtheorem}
Let $\cG$ be a countable family of diagonalizable analytic hypergraphs on a Polish space $X$. Then $\cG\not\perp P$.
\end{theorem}

\begin{proof}
Suppose that $B\subset X$ is a Borel $I_{\cG}$-positive set , $f\colon B\to P$ is a Borel function, and $G\in\cG$ is a hypergraph. Let $f_0\colon B\to\gwtree$ and $f_1\colon B\to\baire$ be the Borel functions which for each $x\in\baire$ indicate the first and second coordinate
of the condition $f(x)\in P$. By the $\gs$-additivity of the $\gs$-ideal $I_{\cG}$, I can shrink the set $B$ if necessary to find some $t\in\gwtree$ such that $f_0(x)=t$ for all $x\in B$. Now, a small claim is helpful:

\begin{claim}
The set $C=\{x\in B\colon\forall n\ \exists e\in G\ \rng(e)\subset B\land x=e(0)\land \forall y\in\rng(e)\ f_1(y)\restriction n=f_1(x)\restriction n\}$ does not belong to the ideal $I_{\cG}$.
\end{claim}

\begin{proof}
If $C\in I$ then $C$ is covered by a Borel set $D\in I$. The Borel set $B\setminus D$ is covered by the union of the coanalytic sets $E_n=\{x\in B\colon\forall e\in G\ \rng(e)\not\subset B\lor x\neq e(0)\lor\exists y\in\rng(e)\ f_1(y)\restriction n\neq f_1(x)\restriction m\}$.
Use the separation theorem  to find Borel sets $F_n\subset E_n$ such that $B\setminus D=\bigcup_nF_n$. By the $\gs$-additivity of the $\gs$-ideal $I_{\cG}$ there is a number $n$ such that $F_n\notin I_\cG$. By the $\gs$-additivity
of $I_\cG$ again, there must be a string $s\in\gw^n$ such that the Borel set $\{x\in F_n\colon f_1(x)\restriction n=s\}$ is $I_\cG$-positive. Choose an edge $e\in G$ whose range is contained in this set. Clearly $e(0)\in E_n$, yet the edge $e$ violates the defining
property of the set $E_n$. This is a contradiction.
\end{proof}

Now, let $x\in C$ be any point, and for each number $n\in\gw$ select an edge $d_n\in G$ such that $\rng(d_n)\subset B$, $d(0)=x$, and for all $y\in\rng(d)$, $f_1(y)\restriction n=f_1(x)\restriction n$. Finally, use the diagonalizability of the hypergraph $G$
to find finite sets $a_n\subset\rng(d_n)$ and an edge $e\in G$ such that $\rng(e)\subset\bigcup_na_n$. Observe that the set $\{f_1(y)\colon y\in\rng(e)\}$ has an upper bound in the total domination ordering on $\baire$, since for every number $n\in\gw$,
among the numbers $\{f_1(y)(n)\colon y\in\rng(e)\}$ only the finitely many numbers $\{f_1(y)\colon y\in\bigcup_{m\in n}a_m\}$ are distinct from $f_1(x)(n)$. Let $g\in\baire$ be the upper bound, and observe that the condition $p=\langle t, g\rangle$
is stronger than all conditions $f(x)$ for $x\in\rng(e)$. This proves the theorem.
\end{proof}

\begin{corollary}
\label{bounding1corollary}
Suppose that $\cG$ is a countable family of diagonalizable hypergraphs on a Polish space $X$. The forcing $P_{I_{\cG}}$, if proper, is bounding.
\end{corollary}

\begin{corollary}
\label{boundingcorollary}
Let $\cG$ be a countable family of diagonalizable analytic hypergraphs on a Polish space $X$ such that the quotient poset is proper. Let $I$ be the associated $\gs$-ideal. Every analytic $I$-positive set contains a compact $I$-positive subset, and every Borel function with $I$-positive domain is continuous on an $I$-positive set.
\end{corollary}

\begin{proof}
This follows from the standard characterization of the bounding property \cite[Theorem 3.3.2]{z:book2} and Theorem~\ref{boundingtheorem}.
\end{proof}

\begin{example}
Let $G$ be the hypergraph on $X=\baire$ such that $e\in X^\gw$ belongs to $G$ just in case there is $n\in\gw$ such that for every $m>n$ the set $\{x(m)\colon x\in \rng(e)\}$ has size greater than $2^m$. The hypergraph is box-open and so its $\gs$-ideal
is generated by closed sets and the quotient forcing is proper. It is easy to see that $G$ is diagonalizable. To see this, for every sequence $\langle e_n\colon n\in\gw\rangle$ of edges in the hypergraph $G$ find increasing sequence of numbers $m_n\in\gw$ such that 
for all $m>m_n$ the set $\{x(m)\colon x\in \rng(e_n)\}$ has size greater than $2^m$. Find finite sets $a_n\subset\rng(e_n)$ such that for all $m_n<m\leq m_{n+1}$ the set $\{x(m)\colon x\in \rng(a_n)\}$ has size greater than $2^m$ , and let $e\in X^\gw$ be any sequence enumerating the set $\bigcup_na_n$. It is immediate that $(\gw\setminus m_0)\times 2\subset \bigcup\rng(e)$ and so $e\in G$. Thus, the associated quotient poset is bounding. In contradistinction to the finitary hypergraphs, the poset does not have the weak Sacks property essentially by the definition of the hypergraph $G$.
\end{example}

\begin{example}
Let $G$ be the hypergraph on $X=\cantor$ such that $e\in X^\gw$ belongs to $G$ just in case $\bigcup\rng(e)$ contains all but finitely many elements of $\gw\times 2$. The hypergraph is box-open.
Similarly to the previous example, $G$ is diagonalizable. Thus, the associated quotient poset is bounding. In fact, the quotient poset is the
forcing introduced by Shelah \cite{Sh:326} and further investigated by Spinas \cite{spinas:splitting}.
\end{example}

\begin{example}
Let $G$ be the hypergraph on $X=\cantor$ such that $e\in X^\gw$ belongs to $G$ just in case $\rng(e)$ contains all but finitely many of the points which differ from $e(0)$ at exactly one entry. This is an actionable hypergraph, clearly invariant under the usual action of the rational Cantor group on $X$. The hypergraph is also diagonalizable and so the quotient forcing is bounding. Let $I$ be the $\gs$-ideal generated by Borel $G$-anticliques. In contradistinction to the finitary hypergraphs, the $P_I$-extension violates the conclusion of Corollary~\ref{splittingcorollary}; writing $\dot b$ for the set of all points in $X$ which differ from $\dotxgen$ in exactly one entry,  $P_I$ forces that every ground model coded Borel subset of $X$ contains either finitely many or cofinitely many elements of $\dot b$. To see this, suppose towards contradiction that $B$ is a condition in the quotient poset $P_I$ and $D\subset\cantor$ is a Borel set
and $B\Vdash\dot D$ splits the set $\dot b$ into two infinite pieces. Thin down $B$ to decide whether $\dotxgen\in \dot D$ holds or not; for definiteness assume the decision is affirmative. Let $M$ be a countable elementary submodel of a large structure containing $B$ and $D$. The set $C=\{x\in B\colon x$ is $P_I$-generic over the model $M\}$ is Borel and $I$-positive by the properness criterion \cite[Proposition 2.2.2]{z:book2}. Thus, it must contain an edge $e\in G$. All vertices of the edge are $P_I$-generic and meet the condition $B$, so they all belong to the set $D$.
But then, $e(0)\in B$ is a generic point such that almost all points which differ from it in exactly one entry (namely all points in $\rng(e)$) belong to $D$. This contradicts the forcing theorem.
\end{example}

\begin{example}
\label{porousexample}
Let $\langle X, d\rangle$ be a compact metric space. For every set $A\subset X$ and a point $x\in X$ say that $A$ is \emph{porous} at $x$ if there is an $\eps>0$ such that for each $\delta>0$ there is a point $y\in X$ such that $0<d(x, y)<\delta$, and the open ball centered
at $y$ with radius $\eps\cdot d(x, y)$ contains no elements of the set $A$. A set $A\subset X$ is \emph{porous} if it is porous at all of its points, and it is $\gs$\emph{-porous} if it is a union of countably many porous sets. Then, the $\gs$-ideal generated by Borel porous sets is hypergraphable
and the generating hypergraph is nearly open and diagonalizable. The quotient poset is proper and bounding.
\end{example}

\begin{proof}
Let $G\subset X^\gw$ consist of all sequences $e\in X^\gw$ such that the set $\rng(e)$ is not porous at $e(0)$. It is not difficult to see that this is in fact a Borel, nearly open hypergraph. Clearly, a set $A\subset X$ is porous just in case $A$ is a $G$-anticlique.
It remains to show that the hypergraph $G$ is diagonalizable. Suppose that $e_n\in G$ are edges for each $n\in\gw$, and for each $n\in\gw$ the value $e_n(0)$ is the same, equal to some point $x\in X$. For each number $n\in\gw$ find a real number $\gd_n>0$
such that for every point $y\in X$ with $0<d(x, y)\leq\gd_n$, the open ball around $y$ with radius $2^{-n}\cdot d(x, y)$ contains some element of $\rng(e_n)$. Shrinking the numbers $\delta_n$ if necessary, I may assume that they form a decreasing sequence. The compactness of the space $X$ provides a finite set $a_n\subset\rng(e_n)$ such that for every point $y\in X$ with $\gd_{n+1}\leq d(x, y)\leq\gd_n$, the open ball around $y$ with radius $2^{-n}\cdot d(x, y)$ contains some element of $a_n$. Let
$e\in X^\gw$ be any sequence with $e(0)=x$ and $\rng(e)=\{x\}\cup\bigcup_na_n$ and observe that $e$ is a $G$-edge with the required diagonalization properties.
\end{proof}

\begin{corollary}
\textnormal{\cite{dolezal:games}}
Let $\langle X, d\rangle$ be a compact metric space. Every non-$\gs$-porous analytic set has a compact non-$\gs$-porous subset. Every Borel function with non-$\gs$-porous domain is continuous on a non-$\gs$-porous domain.
\end{corollary}

\begin{proof}
Consider the ideals $I, J$, the former $\gs$-generated by Borel porous sets, and the latter $\gs$-generated by arbitrary porous sets. The hypergraph generating the two ideals is nearly open, and so by Theorem~\ref{coincidencetheorem}
the two ideals coincide on analytic sets.
The corollary now follows immediately from Example~\ref{porousexample} and the general characterization of the bounding property in quotient forcings in \cite[Theorem 3.3.2]{z:book2} applied to $P_I$.
\end{proof}

\subsection{Outer measure preservation}

Preserving outer measure is a desirable forcing property, and its verification often involves tricky fusion arguments. In the case of hypergraphable forcings, it is implied by a simple and useful criterion. Everywhere below, the letter $\mu$ denotes the usual Borel probability measure on $\cantor$.

\begin{definition}
Let $\cG$ be a countable family of analytic hypergraphs on a Polish space $X$. Say that $\cG$ is \emph{Fubini} if for every $G\in\cG$ there is $H\in\cG$ and $\eps>0$ such that for every edge $d\in H$ and every Borel set $D\subset\rng(d)\times\cantor$ all of whose vertical sections have $\mu$-mass larger than $1-\eps$,
there is a point $y\in\cantor$ and an edge $d\in G$ such that $\rng(e)\subset\rng(d)$ and for all $x\in\rng(e)$, $\langle x, y\rangle\in D$. If $\cG$ is Fubini and contains just one hypergraph $G$, I will call $G$ Fubini as well. 
\end{definition}

The instrumental Suslin forcing is of course the random forcing, i.e.\ the poset $P$ of closed subsets of $\cantor$ of positive $\mu$-mass, ordered by inclusion. This is the optimal poset for the purposes of this section, as is obvious from the following:

\begin{theorem}
\label{randomoptimaltheorem}
Let $\cG$ be a countable family of analytic hypergraphs on a Polish space $X$ such that the poset $P_{I_\cG}$ is proper. Then (1) implies (2), where

\begin{enumerate}
\item $\cG\not\perp P$;
\item $P_{I_\cG}$ preserves the outer Lebesgue measure.
\end{enumerate}

\noindent If $\cG$ contains only finitary hypergraphs, then (2) implies (1).
\end{theorem}

\begin{proof}
To see how (1) implies (2), first use Theorem~\ref{keytheorem} to show that (1) implies $I_{\cG}\not\perp J$ where $J$ is the $\gs$-ideal on $\cantor$ of sets of zero $\mu$-mass. Then, \cite[Proposition 3.2.11]{z:book2} implies that $P_{I_\cG}$ preserves the outer Lebesgue measure.

For the opposite direction, suppose that $\cG$ contains only finitary hypergraphs, and the poset $P_{I_\cG}$ preserves the outer Lebesgue measure. To prove (1), suppose that $B\subset X$ is a Borel $I_{\cG}$-positive set, $f\colon B\to P$ is a Borel function, and $G\in\cG$ is a hypergraph.
Let $C\subset B\times\cantor$ be the Borel set of all pairs $\langle x, y\rangle$ such that $\lim_n 2^{-n}\cdot \mu(f(x)\cap [y\restriction n])=1$ and use the Lebesgue density theorem to see that all vertical sections of the set $C$ have positive $\mu$-mass. Let $\Gamma$ be the rational Cantor group with its usual action on $\cantor$. Note that the $\Gamma$-saturation of any vertical section of the set $C$ is of $\mu$-mass $1$. Since the poset $P_{I_\cG}$ preserves the outer Lebesgue measure, it forces that there must be a ground model point of $\cantor$ in the $\Gamma$-saturation of the set $\dot C_{\dotxgen}$. Standard manipulations then show that thinning out the set $B$ if necessary one can find
a point $y\in\cantor$ such that for all $x\in B$, $\langle x, y\rangle\in C$. Let $e\in\cG$ be an edge with all vertices in the set $B$. Find a natural number $n\in\gw$ such that $2^{-n}\cdot \mu(f(x)\cap [y\restriction n])>1-1/|e|$ for all $x\in\rng(e)$. It follows immediately that
$[y\restriction n]\cap \bigcap_{x\in\rng(e)}f(x)$ is a set of positive $\mu$-mass, and therefore a condition in the poset $P$. This condition forces $f(x)$ into the generic filter for all $x\in\rng(e)$, confirming (1).
\end{proof}

\begin{theorem}
\label{measuretheorem}
If $\cG$ is a Fubini countable family of analytic hypergraphs on a Polish space $X$, then $\cG\not\perp P$ where $P$ is the random forcing.
\end{theorem}

\begin{proof}
For the purposes of this proof, I present the random forcing as the collection of all $\mu$-positive closed subsets of $\cantor$, ordered by inclusion. Suppose that $B\subset X$ is a Borel $I_{\cG}$-positive set , $f\colon B\to P$ is a Borel function,
and $G\in \cG$ is a hypergraph. Let $\eps>0$ and $H\in\cG$ witness the Fubini property for $G$. By the Lebesgue density theorem and countable additivity of the $\gs$-ideal $I_{\cG}$, it is possible to shrink the set $B$ to find a basic open set
$O\subset\cantor$ such that for all $x\in B$, $\mu(B\cap O)>(1-\eps/2)\mu(O)$. Let $d\in H$ be an edge consisting of points in the set $B$. Let $D\subset\rng(d)\times O$ be the Borel set of all pairs $\langle x, y\rangle$ such that $y\in f(x)$ holds.
The choice of $H$ and $\eps$ shows that the analytic set $C=\{y\in O\colon\exists e\in G\ \rng(e)\subset\rng(d)\land \rng(e)\times\{y\}\subset D\}$ has $\mu$-mass at least $\eps/2$. Let $p\in P$ be some condition with $p\subset C$. Since $p$ forces the random point to belong to $C$, it forces the existence of some edge $e\in G$ such that $\rng(e)\subset\rng(d)$ and for all $x\in\rng(e)$, $f(x)$ contains the random point, i.e. $f(x)$ belongs to the generic filter. This completes the proof.
\end{proof}

\begin{corollary}
\label{measurecorollary}
If $\cG$ is a Fubini countable family of analytic hypergraphs on a Polish space $X$, then the quotient forcing $P_{I_\cG}$, if proper, preserves the outer Lebesgue measure.
\end{corollary}

\begin{example}
\label{millerexample}
Let $X=\baire$ and let $G\subset X^\gw$ be the hypergraph of all infinite sets without any accumulation point. The hypergraph has the Fubini property, as any real number $0<\eps<1$ will witness. For if $a\subset X$ is an infinite set without any accumulation point
and $D\subset a\times\cantor$ is a Borel set whose vertical sections have $\mu$-mass at least $1-\eps$, the the limsup, the set $\{y\in\cantor\colon\exists^\infty x\in a\ \langle x, y\rangle\in D\}$ has $\mu$-mass at least $1-\eps$. Whenever $y\in\cantor$
is a point in the limsup, the horizontal section $D_y$ contains an infinite set and therefore a $G$-edge. The $\gs$-ideal $\gs$-generated by Borel $G$-anticliques coincides with the $\gs$-ideal generated by compact subsets of $X$. The quotient forcing is the Miller forcing.
\end{example}

\begin{example}
Let $X, Z$ be Polish spaces and $f\colon X\to Z$ be a Borel function. Let $G\subset X^\gw$ be the hypergraph given by $e\in G$ if $\lim_ie(i)=e(0)$ and $\lim_if(e(i))\neq f(e(0))$.  The hypergraph has the Fubini property as any number $0<\eps<1/2$ will witness.
The $\gs$-ideal $\gs$-generated by Borel $G$-anticliques coincides with the $\gs$-ideal generated by sets on which the function $f$ is continuous. The quotient forcing is the Pawlikowski forcing studied in \cite[Section 4.2.3]{z:book2}.
\end{example}

\begin{example}
Let $f\colon [0,1]\to [0,1]$ be a continuous function. Let $G\subset X^\gw$ be the hypergraph given by $e\in G$ if $\lim_ie(i)=e(0)$ and $\lim_i\frac{f(e(i))-f(e(0))}{e(i)-e(0)}$ does not exist.  The hypergraph has the Fubini property as any number $0<\eps<1/3$ will witness. 
To see this, let $e\in G$ be an edge and $D\subset\rng(e)\times\cantor$ be a Borel set whose vertical sections have a $\mu$-mass larger than $1-\eps$. Use the fact that $e$ is an edge to find a real number $\gd>0$ and numbers $i_n, j_n$ tending to infinity
such that $|\frac{f(e({i_n}))-f(e(0))}{e({i_n})-e(0)}-\frac{f(e({j_n}))-f(e(0))}{e({j_n})-e(0)}|>\gd$. For each $n\in\gw$ let $C_n=D_{e(i_n)}\cap B_{e(j_n)}$; so $\mu(C_n)>1-2\eps$. The limsup $\{y\in\cantor\colon\exists^\infty n\ y\in C_n\}$ has $\mu$-mass at least $1-2\eps$,
and its intersection with $D_{e(0)}$ has mass at least $1-3\eps$. If $y\in\cantor$ is any point in the intersection then the horizontal section $D^y$ contains the desired $G$-edge. The $\gs$-ideal $\gs$-generated by Borel $G$-anticliques coincides with the $\gs$-ideal generated by sets on which the function $f$ is differentiable. The quotient forcing is proper. 
\end{example}

\begin{example}
Let $H$ be an analytic Fubini hypergraph on a Polish space $X$. Let $I$ be the $\gs$-ideal on $X$ $\gs$-generated by closed $H$-anticliques. Then the poset $P_I$ preserves outer Lebesgue measure.
\end{example}

\begin{proof}
Note that the poset $P_I$ is proper since $I$ is $\gs$-generated by closed sets. The outer Lebesgue measure preservation property will be proved by identifying a hypergraph generating the $\gs$-ideal $I$ which has the Fubini property.

Let $G\subset X^\gw$ be the analytic set of all sequences $e$ such that one can compose an $H$-edge from the accumulation points of $\rng(e)$. I claim that the $\gs$-ideals $I$ and $I_G$ coincide. For the left-to-right inclusion note
that a closed $H$-anticlique is also a $G$-anticlique. For the right-to-left inclusion suppose that $B$ is a Borel $I_G$-anticlique. Throwing out countably many points if necessary I may assume that $B$ has no isolated points. The definition of the graph $G$ then shows
that the closure of the set $B$ must be an $H$-anticlique.

The Fubini property of the hypergraph $G$ is witnessed by any number $\eps>0$ smaller than the witness for the Fubini property of $H$. Given an edge $e\in G$ and a Borel set $D\subset\rng(e)\times\cantor$ with all vertical sections of mass $>1-\eps$, 
first find an edge $f\in H$ whose range consists of accumulation points of $\rng(e)$. Let $C\subset\rng(f)\times\cantor$ be the set consisting of all pairs $\langle x, y\rangle$ such that for every open neighborhood $O$ of $x$ there is a point $z\in\rng(e)\cap O$
with $\langle z, y\rangle\in D$. It is easy to see that the vertical sections of the set $C$ have $\mu$-mass $\geq 1-\eps$. Thus, there is a point $y\in\cantor$ and an edge $g\in H$ such that $\rng(g)\subset C^y$. The definitions immediately imply that
any enumeration of the countable set $D^y$ is the desired $G$-edge.
\end{proof}

\begin{example}
\label{packingexample}
Let $\langle X, d\rangle$ be a compact metric space and $h\colon \R^+_0\to\R^+_0$ a gauge function; $h$ is continuous, increasing, and $h(0)=0$. The $\gs$-ideal $I$ of sets of $\gs$-finite $h$-dimensional packing measure is generated by a countable Fubini family of hypergraphs, and so the quotient forcing
preserves outer Lebesgue measure.
\end{example}

\begin{proof}
For every $k, n\in\gw$ let $G_{nk}$ be the hypergraph on $X$ consisting of finite edges $e$ such that the range of $e$ consists of pairwise distinct points such that for each $x\in\rng(e)$ there is a real number $0<r_x<2^{-k}$ such that the balls with centers $x$ and radii $r_x$ are pairwise disjoint and $\gS_xh(r_x)>n$. A review of the definitions \cite[Section 4.1.6]{z:book2} shows that the packing premeasure of a set $A\subset X$ is $\leq n$ if there is $k\in\gw$ such that $A$ is a $G_{nk}$-anticlique. It is easy to check that the closure of a $G_{nk}$-anticlique is again a $G_{nk}$-anticlique. Thus, the $\gs$-ideal of sets of finite packing measure is $\gs$-generated by Borel subsets of $X$ which are $G_{nk}$-anticliques for at least one pair of values $n, k\in\gw$.

Now, let $\cG=\{G_{nk}\colon n, k\in\gw\}$; I will verify that $\cG$ is a Fubini family. Let $n, k\in\gw$ be numbers; I will show that for every edge $e\in G_{2n, k}$ and every Borel set $B\subset \rng(e)\times\cantor$ with vertical sections of $\mu$-mass $>1/2$, there is an
edge $f\in G_{nk}$ and a point $y\in\cantor$ such that for every $x\in\rng(f)$, $\langle x, y\rangle\in B$. This is a simple exercise in Fubini theorem. For each point $x\in\rng(e)$ choose a real number $r_x$ witnessing the fact that $e\in G_{2n, k}$,
and Let $\lambda$ be the measure on the finite set $\rng(e)$ assigning each singleton $x$ mass $h(r_x)$. Thus, $\lambda(\rng(e))>2n$. The set $B\subset\rng(e)\times\cantor$ has $\lambda\times\mu$-mass greater than $n$ and so it has to have a horizontal section
$B^y$ of $\lambda$-mass greater than $n$. Let $f$ be any sequence enumerating the set $B^y$ and observe that $f\in G_{nk}$ as required.
\end{proof}

\subsection{Category preservation}

The main forcing restriction on the class of hypergraphable posets is that they have to preserve the Baire category as long as they are proper.

\begin{theorem}
\label{categorytheorem}
Let $X$ be a Polish space and $\cG$ be a countable family of analytic hypergraphs on $X$. Then $\cG\not\perp P$ where $P$ is the Cohen forcing.
\end{theorem}

\begin{proof}
Let me present the Cohen forcing as the poset of finite binary strings ordered by inclusion. If $B\notin I_{\cG}$ is a Borel set, $f\colon B\to P$ is a Borel function and $G\in\cG$ is a hypergraph, then the sets $\{x\in B\colon f(x)=p\}$ form a partition of the set $B$ into countably many Borel sets
as $p$ varies over all the countably many conditions in $P$. By the positivity of the set $B$, one of these sets must contain all vertices of some edge $e\in G$, and then the corresponding condition $p\in P$ forces $f''\rng(e)=\{p\}$ to be a subset of the generic filter. This
completes the verification of $\cG\not\perp P$ and the proof of the theorem.
\end{proof}

\begin{corollary}
\label{categorycorollary}
Let $X$ be a Polish space and $\cG$ be a countable family of analytic hypergraphs on $X$. If the quotient poset $P_{I_{\cG}}$ is proper, then it preserves the Baire category.
\end{corollary}

\begin{proof}
Let $Y=\cantor$ and let $J$ be the $\gs$-ideal of meager sets on $Y$. By \cite[Proposition 3.2.2]{z:book2}, it is enough to show that $I_{\cG}\not\perp J$. This follows from Theorems~\ref{categorytheorem} and~\ref{keytheorem}.
\end{proof}

\subsection{Anticliques in closed graphs}
\label{anticliquesubsection}

In this subsection, I will prove a theorem concerning the covering of the extension by ground model coded compact anticliques of closed graphs. There is a natural poset for adding large compact anticliques to closed graphs, identified previously by numerous authors:

\begin{definition}
\label{rhdefinition}
\textnormal{\cite{geschke:graphs}} Let $Z$ be a Polish space with a fixed countable basis, and let $H$ be a closed graph on $Z$.

\begin{enumerate}
\item $R_H^0$ is the poset of all pairs $r=\langle n_r, o_r, a_r\rangle$ where $a_r\subset Z$ is a finite $H$-anticlique and $o_r$ is a function with domain $a_r$, associating to each point in $a_r$ one of its basic open neighborhoods in such a way that
$z_0\neq z_1\in a_r$ implies $H\cap o_r(z_0)\times o_r(z_1)$. The ordering is defined by $s\leq r$ if $n_r\leq n_s$, $a_r\subseteq a_s$ and $\bigcup\rng(o_s)\subseteq \bigcup\rng(o_r)$.
\item $a_r$ is the \emph{working part} of a condition $r\in R_H^0$, while $n_r, \rng(o_r)$ is the \emph{side condition} of $r$.
\item the $R_H^0$-generic anticlique is the closure of the union of the second coordinates of the conditions in the $R_H^0$-generic filter.
\end{enumerate}

 In addition, I define

\begin{enumerate}
\item[4.] $R_H$ is the finite support product of countably many copies of the poset $R_H^0$, so the conditions in $R_H$ are just finite sequences of conditions in $R_H^0$ with coordinatewise ordering;
\item[5.] the side condition of $r\in R_H$ is the finite set of all pairs $\langle k, O\rangle$ where $O$ is an open set in the working part of $r(k)$;
\item[6.] for every $r\in R_H$ and a pair $\langle k, O\rangle$, write $r(k, O)\in Z$ for the unique element of the working part of $r(k)$ in the set $O$;
\item[7.] the $R_H$-name for the $\gw$-sequence of the generic compact $H$-anticliques at the various coordinates of the product is denoted by $\dotygen$;
\item[8.] let $Y=(K(Z))^\gw$ and let $J_H$ be the $\gs$-ideal generated by those analytic sets $A\subset Y$ such that $R_H\Vdash\dotygen\notin\dot A$.
\end{enumerate}
\end{definition}

An important attendant fact is provided by Todorcevic \cite[Theorem 19.6]{todorcevic:dichotomies}: the closed graph $H$ contains no perfect cliques iff $R_H$ is c.c.c. iff $R_H$ preserves $\aleph_1$. Thus, if the graph $H$ has no perfect anticliques then
$R_H$ is a Suslin c.c.c.\ poset. The poset $R_H$ is optimal for resolving the question whether in various forcing extensions, every point of $Z$ belongs to a ground model coded compact $H$-anticlique. This follows from the following theorem:

\begin{theorem}
\label{fubinitheorem}
Let $\cG$ be a countable family of analytic hypergraphs on a Polish space $X$ such that the poset $P_{I_\cG}$ is proper. Let $H$ be a closed graph on a Polish space $Z$ without perfect cliques. Then (1) implies (2) implies (3):

\begin{enumerate}
\item $\cG\not\perp R_H$;
\item $I_\cG\not\perp J_H$;
\item every element of $Z$ in the $P_{I_\cG}$-extension belongs to a ground model coded compact $H$-anticlique.
\end{enumerate}

\noindent If the hypergraphs generating the ideal $I$ have only finite edges, then (3) implies (1).
\end{theorem}

\begin{proof}
(1) implies (2) by Theorem~\ref{keytheorem}. (2) implies (3) by virtue of the fact that $R_H\Vdash\bigcup\dotygen$ contains all ground model elements of the space $Z$, which is proved by an elementary density argument.
The nontrivial part of the theorem is the last sentence. Suppose that (3) holds, $B\subset X$ is a Borel $I_{\cG}$-positive set, $f\colon B\to R_H$ is a Borel function, and $G\in\cG$ is a hypergraph.
By the $\gs$-additivity of the $\gs$-ideal $I_{\cG}$, it is possible to shrink the set $B$ if necessary to find some $o$ which is the side condition of every condition $f(x)$ for $x\in X$. For each pair $\langle k, O\rangle\in o$ the function $f_{kO}\colon B\to Z$ defined by
$f_{kO}(x)=f(x)(k, O)$ is Borel. By (3), one can shrink the set $B$ further to find compact $H$-anticliques $A_{kO}\subset Z$ such that $f_{kO}''B\subset A_{kO}$. Let $e\in G$ be an edge consisting of vertices in the set $B$.
Since the edge $e$ is finite and for each $\langle k, O\rangle\in o$ the points $f_{kO}(x)$ for $x\in\rng(e)$ are pairwise $H$-disconnected, the conditions $f(x)$ for $x\in\rng(e)$ have a lower bound $r\in R_H$. This condition $r$ witnesses the statement $\cG\not\perp R_H$.
\end{proof}

The conclusion of Theorem~\ref{fubinitheorem} gloriously fails for hypergraphable $\gs$-ideals which need infinite edges in their defining hypergraphs:

\begin{example}
Let $X=\cantor$ with the usual minimum difference metric $d$. Let $I$ be the $\gs$-ideal of $\gs$-porous sets on $X$, as defined in Example~\ref{porousexample}. Let $H$ be the Silver graph on $\cantor$. Then $I$ is hypergraphable,
the quotient forcing $P_I$ is proper, the poset $R$ is c.c.c., every point of $\cantor$ in the $P_I$-extension belongs to a ground model coded $H$-anticlique, and yet $I\perp J_H$.
\end{example}

\begin{proof}
There is a number of things to verify, but they are all either easy or previously known. The hypergraph generating the $\gs$-ideal $I$ is produced in Example~\ref{porousexample}. The graph $H$ is locally countable and so the associated poset $R_H$ it is c.c.c. The proof that every point in $\cantor$ in the $P_I$-extension belongs to a ground model coded compact $H$-anticlique is difficult, but known \cite[Theorem 6.33]{z:book3}. Now, to show that $I\perp J_H$ holds, use a trivial genericity argument to show that the poset $R_H^0$ forces its generic anticlique to be a compact porous subset of $\cantor$. Thus, writing $Y=(K(X))^\gw$ for the underlying space of the $\gs$-ideal $J_H$, there is a Borel $J_H$-positive set $B\subset Y$ consisting just of sequences of porous sets.
Lett $C\subset X\times B$ be the Borel set $C=\{\langle x, y\rangle\colon x\notin \bigcup\rng(y)\}$. The vertical sections of the set $C$ belong to the $\gs$-ideal $J_H$, since $R_H\Vdash\bigcup\rng(\dotygen)$ contains all ground model points of $X$. The vertical sections of the complement of the set $C$ are $\gs$-porous by the choice of the set $B$. The proof is complete!
\end{proof}

I will now prove a theorem which helps with making distinctions between various classes of graphs. The moral of the story is, if $P_I$ is a quotient forcing generated by a family of analytic graphs and $H$ is a closed graph on a Polish space $X$, then if the graphs
in the generating family are much denser than $H$, then in the $P_I$ extension, every point of $Z$ belongs to a ground model coded compact $H$-anticlique. The various ways of measuring the density of the graphs provide various preservation results.

\begin{theorem}
\label{closedanticliquetheorem}
Let $\cG$ be a countable family of analytic graphs on a Polish space such that the quotient poset $P_{I_\cG}$ is proper and adds a minimal real degree. Suppose one and the same of the following holds for every Borel $I_{\cG}$-positive set $B\subset X$ and every graph $G\in\cG$:

\begin{enumerate}
\item for every $n\in\gw$ there is a finite set $a\subset B$ such that $|G\cap [a]^2|>n\cdot |a|$;
\item there is a point $x\in B$ such that the set $\{y\in B\colon\langle x, y\rangle\in G\}$ has uncountable closure;
\item there is a countable set $a\subset B$ with uncountable closure such that $G$ is dense in the set $a^2$.
\end{enumerate}

\noindent Then $\cG\not\perp R_H$ holds for all closed graphs $H$ in the following corresponding classes:

\begin{enumerate}
\item[1*.] acyclic graphs;
\item[2*.] locally countable graphs;
\item[3*.] graphs without perfect cliques.
\end{enumerate}
\end{theorem}

\noindent I do not know if the minimal degree assumption is necessary for the conclusion of the theorem.

\begin{proof}
Let $B\subset X$ be a $I_{\cG}$-positive set, $G\in\cG$ be a graph, and $f\colon B\to R_H$ be a Borel function. By the countable additivity of the $\gs$-ideal $I_{\cG}$, I may assume that there is a side condition $o$ such that for all $x\in X$, the condition $f(x)$ has side condition $o$. For each pair $\langle k, O\rangle\in o$, write $f_{kO}\colon B\to Z$ for the Borel function given by $f_{kO}(x)=f(x)(k,O)$. By the minimal degree assumption, thinning out the set $B$ if necessary I may assume that each of of these functions is either constant
or injective. The quotient forcing is bounding by Corollary~\ref{boundingcorollary}, and so by thinning out the set $B$ if necessary I may assume that the set $B$ is compact and the functions $f_{kO}$ are all continuous on it. The treatment now divides according to the three cases.

Suppose that (1) holds and the graph $H$ is acyclic. Find a finite set $a\subset B$ such that $|G\cap [a]^2|>n\cdot |a|$ where $n=|o|$. Since the functions $f_{kO}$ are all injective or constant, the preimage $f_{kO}^{-1}H$ is either an acyclic graph or an empty graph on $a$.
Since acyclic graphs on $a$ have at most $|a|-1$ many edges, a counting argument shows that there is an edge $\langle x, y\rangle\in G$ consisting of vertices in $a$ which does not belong to $\bigcup_{k, O}f_{kO}^{-1}H$. It follows from the definitions
that the conditions $f(x)$ and $f(y)$ have a lower bound, which witnesses $\cG\not\perp R_H$.

Suppose that (2) holds and the graph $R_H$ is locally countable. Find a point $x\in B$ such that the closure $C\subset B$ of the set $\{y\in B\colon\langle x, y\rangle\in G\}$ is uncountable. If there is $y\in B$ such that $\langle x, y\rangle\in G$
yet for no $\langle k, O\rangle\in o$ $\langle f_{kO}(x),  f_{kO}(y)\rangle H$ holds, then the conditions $f(x), f(y)\in R_H$ are compatible, confirming $\cG\not\perp R_H$. Thus, assume that this fails and work towards a contradiction. Since the functions $f_{kO}$ are all continuous and the graph $H$ is closed, it would have to be the case that for every $y\in C$ distinct from $x$,
there is $\langle k, O\rangle\in o$ such that $\langle f_{kO}(x), \langle f_{kO}(y)\rangle H$ holds. One of the pairs $\langle k,O\rangle\in o$ works for uncountably many elements $y\in C$, which contradicts the assumption that the point $f_{kO}(x)$ has only countably
many $H$-neighbors.

Suppose that (3) holds and the graph $H$ has no perfect cliques. Find a set $a\subset B$ with uncountable closure $C\subset B$ such that $G$ is dense in the set $a^2$. If there are points $x\neq y\in a$ such that $\langle x, y\rangle\in G$ yet
for no $\langle k, O\rangle\in o$ $\langle f_{kO}(x), f_{kO}(y)\rangle\in H$ holds, then the conditions $f(x), f(y)$ are compatible, confirming $\cG\not\perp R_H$. Thus, assume that this fails and work towards contradiction.
Since the functions $f_{kO}$ are continuous and the graph $H$ is closed, for any two distinct points
$x\neq y\in C$ there is $\langle k, O\rangle\in o$ such that $\langle f_{kO}(x), f_{kO}(y)\rangle\in H$ holds. By a classical partition theorem due to Galvin \cite{galvin:real}, there is a perfect set $D\subset C$ and a pair $\langle k, O\rangle\in o$ such that $x\neq y\in D$ implies $\langle f_{kO}(x), f_{kO}(y)\rangle\in H$.
It follows that $f_{kO}''D\subset Z$ is a perfect $H$-clique, contradicting the assumptions on the graph $H$.
\end{proof}

I provide two applications of the theorem, which in no sense exhaust its potential.

\begin{corollary}
\label{silveracycliccorollary}
Let $H$ be an acyclic closed graph on a Polish space $Z$. In the Silver extension, every element of $Z$ belongs to some compact ground model coded compact $H$-anticlique. 
\end{corollary}

\begin{proof}
I will use Theorem~\ref{closedanticliquetheorem}(1). It is well-known that the Silver forcing adds a minimal real degree. For every Silver-positive Borel set $B\subset\cantor$, there is a continuous injection $h\colon\cantor\to B$
reducing the Silver graph to itself. Now for every number $m\in\gw$, consider the set $a\subset\cantor$ consisting of all binary sequences which are zero at all entries $\geq m$. It is not difficult to see that $|a|=2^m$ and the number
of Silver edges between the various elements of $a$ is $2^{m-1}\cdot m=\frac{1}{2}|a|\cdot\log |a|$. Transporting these sets $a$ by the injection $h$ to the set $B$, I conclude that the assumption (1) of the theorem is satisfied. 
\end{proof}

\begin{corollary}
\label{neverclosedcorollary}
Let $H$ be a closed graph on a Polish space $Z$ with no perfect cliques. In the Vitali extension, every element of $Z$ belongs to some ground model coded compact $H$-anticlique.
\end{corollary}

\begin{proof}
I will use Theorem~\ref{closedanticliquetheorem}(3). It is well-known the Vitali forcing adds a minimal real degree. For every Vitali-positive Borel set $B\subset\cantor$, there is a continuous injection $h\colon\cantor\to B$
which is a reduction of the Vitali graph to itself. For any Vitali class $a\subset\cantor$, the set $h''a\subset B$ shows that the assumption (3) is satisfied.
\end{proof}

\noindent Other classes of closed graphs are discussed in Subsection~\ref{smallsubsection}.

\subsection{The small anticlique property}
\label{smallsubsection}

Reviewing the proofs in the previous sections, a question arises. Is it possible to drop the definability conditions on the c.c.c.\ poset and the partition and still get a theorem? The answer is yes, but one has to
pay for this option with the following interesting assumption on the hypergraphs in question.

\begin{definition}
A countable collection $\cG$ of analytic hypergraphs on a Polish space $X$ has the \emph{small anticlique property} if for every sequence $\langle B_n, G_n\colon n\in\gw\rangle$ such that $G_n\in\cG$ and $B_n\subset X$ is a $G_n$-anticlique, the union
$\bigcup_nB_n$ does not contain an analytic $I_{\cG}$-positive subset.
\end{definition}

\noindent Note that I do not require the anticliques $B_n$ to be Borel, and that is the whole point. The small anticlique property says that the two $\gs$-ideals, one generated by Borel anticliques and the other by arbitrary anticliques, coincide on the class of analytic sets.
Thus, no actionable family has the small anticlique property, since the whole space $X$ can be covered by countably many sets, each of which is a selector on the underlying orbit equivalence relation and so an anticlique with respect to any of the graphs on the
actionable family. On the other hand, the nearly open hypergraphs do have the small anticlique property:

\begin{theorem}
\label{coincidencetheorem}
Let $X$ be a Polish space and $\cG$ a countable family of analytic, nearly open hypergraphs on $X$. Then $\cG$ has the small anticlique property.
\end{theorem}

\begin{proof}
The key claims:

\begin{claim}
\label{nearlyopenclaim}
Let $G$ be a nearly open analytic hypergraph. Every $G$-anticlique is a subset of a coanalytic $G$-anticlique.
\end{claim}

\begin{proof}
Let $A\subset X$ be a $G$-anticlique, let $C\subset X$ be its closure, and consider the set $B=\{x\in C\colon \forall y\in C^{\leq\gw}\ y(0)\neq x\lor y(0)\notin G\}$. The definition shows that $B\subset C$ is a coanalytic set, and it is a $G$-anticlique. I will be finished
if I show that $A\subset B$. Indeed, suppose that $x\notin B$ and argue that $x\notin A$ must hold. There must be an edge $y\in C^{\leq\gw}\cap G$ such that $y(0)=x$. For each $n\in\dom(y)$ with $y>0$ there must be an open set $O_n\subset X$ containing
$y(n)$ such that the product $\{x\}\times\prod_nO_n$ consists only of $G$-edges. Since each point $y(n)$ belongs to the closure of $A$, the intersection $O_n\cap A$ must be nonempty, containing some point $x_n\in A$. Then, $\langle x, x_n\colon n\in\dom(y), n>0\rangle$ is an edge of the hypergraph $G$ and so one of its points must be outside of the anticlique $A$. That point must be $x$.
\end{proof}

\begin{claim}
\label{coanalyticclaim}
Let $G$ be an analytic hypergraph on a Polish space $X$, and let $I$ be a $\gs$-ideal on $X$ containing all Borel $G$-anticliques, such that the quotient poset $P_I$ is proper. Whenever $A\subset X$ is a coanalytic $G$-anticlique then $P_I\Vdash\dotxgen\notin\dot A$.
\end{claim}

\begin{proof}
Suppose for contradiction that $B\in P_I$ is a condition forcing $\dotxgen\in\dot A$. Let $M$ be a countable elementary submodel of a large structure and et $C=\{x\in B\colon x$ is $P_I$-generic over the model $M\}$. By the properness criterion \cite[Proposition 2.2.2]{z:book2},
the set $C\subset B$ is Borel and $I$-positive. By the forcing theorem applied in the model $M$, for every $x\in C$ $M[x]\models x\in A$ holds. Since the model $M[x]$ is wellfounded, it is correct about membership in analytic sets, and so $x\in A$ holds in $V$.
In conclusion, $C\subset A$ and so $C$ is a $G$-anticlique. This contradicts the initial assumption that $I$ contains all Borel $G$-anticliques.
\end{proof}

Now, suppose that $A\subset X$ is an analytic set. Suppose that $A\notin I_{\cG}$, $G_n\in\cG$ are hypergraphs and $B_n\subset X$ are $G_n$-anticliques; I must produce a point in $A\setminus\bigcup_nB_n$. By Claim~\ref{nearlyopenclaim}, the sets $B_n$ may be assumed to be coanalytic; by Theorem~\ref{millerfact}, $A$ may be assumed to be Borel. Let $M$ be a countable elementary submodel of a large structure containing $\cG$, $A, B_n$ for $n\in\gw$, and let $x\in A$ be a point $P_{I_\cG}$-generic over the model $M$, meeting the condition $A$. By Claim~\ref{coanalyticclaim}, the point $x$ does not belong to any of the coanalytic anticliques $B_n$ for $n\in\gw$. The proof is complete!
\end{proof}

\noindent It is not all that easy to find other examples. The small anticlique property is preserved by the (wellfounded) countable support iterations, which provides a strategically placed source of hypergraphs with the small anticlique property.

There are several theorems that greatly improve the conclusions of Theorem~\ref{localizationtheorem} etc.\ if the families of hypergraphs in question have the small anticlique property. I will use the following definition, to be contrasted with Definition~\ref{perpdefinition}.

\begin{definition}
Suppose that $\cG$ is a countable family of analytic hypergraphs on a Polish space $X$. Suppose that $P$ is a partial ordering. The symbol $\cG\not\perp^* P$ denotes the following statement: for every Borel $I$-positive set $B\subset X$, every function $f\colon B\to P$ and every hypergraph $G\in \cG$ there is a condition $p\in P$ which forces that there is an edge $e$ consisting of vertices from the ground model elements of $B$ such that $f''\rng(e)$ is a subset of the generic filter.
\end{definition}

\noindent Note the lack of definability demands on the poset $P$ and the function $f$. Stated this way, it may seem that the defined property is not satisfied in any interesting cases, but alas! it is. The following theorem is proved in the same way as Theorem~\ref{keytheorem}.

\begin{theorem}
Suppose that $\cG$ is a countable family of analytic hypergraphs on a Polish space $X$. Suppose that $P$ is a partial ordering, $Y$ is a Polish space and $\dot y$ is a $P$-name for an element of $Y$. $\cG\not\perp^* P$ implies $I_{\cG}\not\perp J_{\dot y}$.
\end{theorem}

\noindent I will now state several interesting theorems and corollaries which do not fall into the context of the previous subsections.

\begin{theorem}
\label{smalltheorem}
Suppose that $\cG$ is a countable family of analytic hypergraphs on a Polish space $X$ with the small anticlique property, all edges in all hypergraphs finite. Suppose that $P$ is a $\gs$-centered forcing. Then ${\cG}\not\perp^* P$ holds.
\end{theorem}

\begin{proof}
Suppose that $B\subset X$ is an $I_{\cG}$-positive Borel set, $f\colon B\to P$ is a function, and $G\in\cG$ is a hypergraph. Fix a cover $P=\bigcup_mP_m$ of $P$ by countably many centered pieces. The preimages $f^{-1}P_m$ for $m\in\gw$ cover the set $B$, and by the
small anticlique property of $\cG$, there is an edge $e\in G$ such that $\rng(e)$ is a subset of one of them, say of $f^{-1}P_m$. The set $P_m$ is centered, the set $f''\rng(e)\subset P_m$ is finite, and therefore it has a lower bound $p$. This condition
witnesses the validity of ${\cG}\not\perp^* P$.
\end{proof}

\begin{corollary}
\label{boocorollary}
The quotient posets in the Theorem~\ref{smalltheorem} category, if proper, do not add an independent real.
\end{corollary}

\begin{proof}
Let $U$ be a nonprincipal ultrafilter on $\gw$, and let $P$ be the standard poset diagonalizing the ultrafilter $U$, i.e.\ adding an infinite set $\dot y\subset\gw$ which is modulo finite included in every element of $U$. Let $Y=[\gw]^{\aleph_0}$ and let $J$ be the $\gs$-ideal of all analytic sets $A\subset Y$ such that $P\Vdash\dot y\notin \dot A$. Theorem~\ref{smalltheorem} says that $I_{\cG}\not\perp J$. 

Now suppose that $B\in P_{I_{\cG}}$ is a condition forcing $\tau\in\cantor$; I have to find an infinite set $y\in Y$ such that some stronger condition
forces $\tau$ to be constant on $y$ with perhaps finitely many exceptions. Thinning out the set $B$ if necessary I may assume that there is a Borel function $f\colon B\to\cantor$ such that $B\Vdash\tau=\dot f(\dotxgen)$. Let $C\subset B\times Y$ be the Borel
set of all pairs $\langle x, y\rangle$ such that $f(x)$ is not constant on $y$ modulo finite. A simple density argument with $P$ shows that the vertical sections of the set $C$ belong to the $\gs$-ideal $J$. Since $I_{\cG}\not\perp J$ holds,
there must be $y\in Y$ such that the set $\{x\in B\colon f(x)$ is constant modulo finite on $y\}$ is $I_{\cG}$-positive. This concludes the proof.
\end{proof}

\begin{corollary}
Let $H$ be a closed graph on a Polish space $Z$, with countable chromatic number. In the extension by the quotient posets in the Theorem~\ref{smalltheorem} category, if proper, every point in $Z$ belongs to a ground model coded compact $H$-anticlique.
\end{corollary}

\begin{proof}
This uses the poset $R_H$ isolated in Definition~\ref{rhdefinition}, adding countably many compact $H$-anticliques covering the ground model elements of $Z$. It is not difficult to see that the poset $R_H$ is $\gs$-centered just in case $H$ has countable chromatic number \cite[Theorem 3.9(2)]{z:graphs}.
\end{proof}

The following theorem uses a concept which is often useful for proving that various posets do not add dominating reals. Let $P$ be a poset. A set $Q\subset P$ is \emph{liminf-centered} if for every infinite set $a\subset Q$ there is a condition $p\in P$ forcing that infinitely many elements of $a$ belong to the generic filter. The poset $P$ is $\gs$-liminf-centered if it is the union of countably many liminf-centered pieces \cite[Definition 3.8]{z:graphs}.

\begin{theorem}
\label{millertheorem}
Suppose that $\cG$ is a countable family of analytic hypergraphs on a Polish space $X$ with the small anticlique property, such that every edge is infinite and an infinite subset of an edge is again an edge. Suppose that $P$ is a $\gs$-liminf-centered forcing. Then ${\cG}\not\perp^* P$ holds.
\end{theorem}

\begin{proof}
Suppose that $B\subset X$ is an $I_{\cG}$-positive Borel set, $f\colon B\to P$ is a function, and $G\in\cG$ is a hypergraph. Fix a cover $P=\bigcup_mP_m$ of $P$ by countably many liminf-centered pieces. The preimages $f^{-1}P_m$ for $m\in\gw$ cover the set $B$, and by the small anticlique property of $\cG$, there is an edge $e\in G$ such that the infinite set $\rng(e)$ is a subset of one of them, say of $f^{-1}P_m$. Now, if the set $f''\rng(e)\subset P_m$ is finite, then there is $p\in P_m$ such that for infinitely many $x\in\rng(e)$
$f(x)=p$ holds, and then the condition $p$ witnesses the validity of $\cG\not\perp^*P$. If the set $f''\rng(e)$ is infinite, then as the set $P_m$ is liminf-centered, there is a condition $p\in P$ forcing that for infinitely many $x\in\rng(e)$, $f(x)$ belongs to the generic filter. This condition
witnesses the validity of ${\cG}\not\perp^* P$ in this case as well.
\end{proof}

My main source of $\gs$-liminf-centered posets are closed graphs on compact metrizable spaces with countable loose number. This is a concept developed in \cite{z:graphs}. Let $H$ be a graph on a topological space $Z$. A subset $A\subset Z$ is \emph{loose} if every point $z\in Z$
has a neighborhood $O$ such that no points of $A\cap O$ are in $H$-relation with $z$. The loose number of a graph is the smallest cardinality of a collection of loose sets covering the whole space $Z$. It is proved in \cite{z:graphs} that graphs with countable loose number include acyclic graphs and locally countable graphs among others. Moreover, if $H$ is a closed graph on a compact metrizable space $Z$ with countable loose number, then the forcing $R_H$ of Definition~\ref{rhdefinition} is $\gs$-liminf-centered \cite[Theorem 3.9(3)]{z:closed}.

\begin{corollary}
\label{loosecorollary}
Let $H$ be a closed graph on a compact metrizable space $Z$, with countable loose number.  In the extension by the quotient posets in the Theorem~\ref{millertheorem} category, if proper, every point in $Z$ belongs to a ground model coded compact $H$-anticlique.
\end{corollary}

\begin{example}
In the Miller extension, every point in $\cantor$ belongs to some ground model coded compact KST-anticlique. This happens since the KST graph is locally countable and so of countable loose number. Moreover, the Miller forcing can be hypergraphed so that any infinite subset of an edge is again an edge; this is shown in Example~\ref{millerexample}.
\end{example}

\subsection{Adding an independent real}

I close this enormous section with an anti-preservation theorem, showing that certain property of a hypergraph implies a lack of a preservation property of the quotient forcing. While I hope for more results of this type, so far there is only one. Recall that a function $x\in\cantor$ is an \emph{independent real} over some model of ZFC if it has no infinite subfunction in that model.

\begin{theorem}
\label{independenttheorem}
Let $\cG$
be a family of analytic hypergraphs on a Polish space $X$ such that for some $G\in\cG$, $G$ can be written as an increasing countable union of hypergraphs, each with Borel chromatic number two.
Then the quotient forcing $P_{I_{\cG}}$ adds an independent real.
\end{theorem}

\begin{proof}
Let $G=\bigcup_nH_n$ be an increasing union and for each $n\in\gw$, let $X=B_n^0\cup B_n^1$ be a partition into Borel sets neither of which contains an $H_n$-edge.
Let $f\colon X\to\cantor$ be the Borel function defined by $f(x)(n)=b$ if $x\in B_n^b$. I claim that $\dot f(\dotxgen)$ is a name for an independent real.

Suppose that this fails, and find a condition $B\subset X$ and an infinite partial function $g\colon\gw\to 2$ such that $B\Vdash\check g\subset\dot f(\dotxgen)$. Thinning out the condition
$B$ if necessary, I may assume that for all $x\in B$, $g\subset f(x)$. Let $e\in G$ be an edge consisting of points in $B$, and let $n\in\dom(g)$ be a number such that $e\in H_n$.
Then, there must be a point $x\in\rng(e)$ such that $x\in B_n^{1-g(n)}$, and for such a point $x\in B$, it is clearly not the case that $g\subset f(x)$. A contradiction.
\end{proof}

\begin{example}
Let $G$ be the Silver graph of  Example~\ref{silvergraphexample} above, connecting two elements of $\cantor$ if they differ in exactly one entry. For each $n\in\gw$ let $H_n\subset G$ be the graph connecting two elements in $\cantor$ if they differ on a single entry smaller than $n$. It is clear that $G$ is the increasing union of $H_n$'s, and each graph $H_n$ has Borel chromatic number two--just let $B_n^0=\{x\in \cantor\colon |\{m\in n\colon x(m)=1\}|$ is even$\}$ and $B_n^1=\{x\in \cantor\colon |\{m\in n\colon x(m)=1\}|$ is odd$\}$.
 The quotient forcing $P_{I_G}$ is well-known to add an independent real.
\end{example}

\begin{example}
Let $X=\baire$ and let $G\subset X^\gw$ be the hypergraph such that $e\in G$ if for some $m=m_e$, all functions in $\rng(e)$ agree on all inputs different from $m$, but they attain infinitely many values at input $m$. Writing $I$ for the $\gs$-ideal generated by Borel $G$-anticliques, it turns out that $P_I$ is proper (since $G$ is actionable) and preserves outer Lebesgue measure (Theorem~\ref{measuretheorem}). It also adds an independent real. For each $n\in\gw$ let $H_n$ be the set of all edges $e\in G$ such that $m_e<n$ and for some $x\in\rng(e)$, $x(m_e)<n$. It is not difficult to see that the hypergraphs $H_n$ increase with $n$ and exhaust all of $G$. The Borel chromatic number of $H_n$ is equal to two: just let $B_n^0=\{x\in\baire\colon |\{m\in n\colon x(m)<n\}|$ is even$\}$ and $B_n^1=X\setminus B_n^0$.
\end{example}

\begin{example}
Let $I$ be the ideal of $\gs$-porous sets on the interval $X=[0,1]$ and the usual Euclidean metric $d$, with the generating hypergraph $G$ described in Example~\ref{porousexample}. For each $n\in\gw$ let $H_n$ be the hypergraph of all edges $e\in G$ such that there is no point $x\in X$ with $x\neq e(0)$, $d(x, e(0))<2^{-n}$ and $d(x, e(0))\leq 4d(x, \rng(e))$. It is immediate that the hypergraphs $H_n$ increase with $n$ and exhaust all of $G$. To show that the Borel chromatic number of $H_n$ is two, cover $X$ by successive intervals of length $2^{-n-1}$;
the treatment of endpoints is irrelevant. It is not difficult to observe that for each edge $e\in H_n$ there must be $m\in\gw$ such that $e(0), e(m)$ belong to two successive intervals. Therefore, the partition $B_n^0=$the union of the intervals coming at even position in the successive order, and $B_n^1=X\setminus B_n^0$, shows that the Borel chromatic number of $H_n$ is two.
\end{example}

\section{Dichotomy results}
\label{invariantsection}

It is natural to hope that the quotient posets arising from especially restrictive classes of hypergraphs can be classified. In this section, I prove a couple of dichotomies for actionable families of graphs which do exactly that. The proofs proceed through the dreaded
fusion arguments, for which I first establish a common notation.

\begin{definition}
\label{fusiondefinition}
Let $I$ be a $\gs$-ideal on a Polish space $X$, let $\Gamma$ be a countable group acting continuously on $X$ in a way that preserves the $\gs$-ideal $I$, and let $A\subset X$ be an analytic $I$-positive set. An \emph{fusion sequence} below $A$ is a sequence $\langle B_n, \gamma_n\colon n\in\gw\rangle$ such that

\begin{enumerate}
\item $A\supset B_0\supset B_1\supset\dots$ and all $B_n$'s are compact $I$-positive sets;
\item $\gamma_n\in\Gamma$. For a binary string $t\in 2^n$ I will always write $\gamma_t$ for the product of all $\gamma_m$ where $m\in n$ is such that $t(m)=1$, and the product is taken in the increasing order;
\item $B_{n+1}$ and $\gamma_n\cdot B_{n+1}$ are disjoint subsets of $B_n$;
\item For each $t\in 2^n$, the diameter of each set $\gamma_t\cdot B_n$ in some fixed complete metric for the space $X$ is smaller than $2^{-n}$.
\end{enumerate}
\end{definition}

Note that in all situations in this section, the quotient poset $P_I$ will be proper and bounding. Thus, every every analytic $I$-positive set will have a compact $I$-positive subset by virtue of Corollary~\ref{finitarycorollary} and the standard characterization of bounding quotient forcings. The following is easy to prove by induction on $n$, using the third item of the
fusion definition:

\begin{proposition}
If $m\in n$ are numbers and $s\in 2^m$, $t\in 2^n$ are binary strings such that $s\subset t$, then $\gamma_t\cdot B_n\subset\gamma_s\cdot B_m$.
\end{proposition}

If $\langle B_n, \gamma_n\colon n\in\gw\rangle$ is a fusion sequence, then the intersection $\bigcap_nB_n$ is a singleton containing a unique point $x\in X$, called the \emph{resulting point} of the fusion. I will be always interested in the map $h\colon\cantor\to X$ defined by $h(z)=\lim_n\gamma_{z\restriction n}\cdot x$. It is not difficult to argue, using the proposition, that $h$ is a well-defined continuous injection from $\cantor$ to $X$. The map $h$ will be referred to as the \emph{resulting map} of the fusion.

\subsection{Invariant graphs}

In this section, I provide a dichotomy which shows that the actionable forcings generated by a single invariant graph inherit all the preservation features of the Silver forcing.

\begin{theorem}
\label{firstdichotomytheorem}
Let $X$ be a Polish space, $\Gamma$ a countable group acting on it, and let $G$ be an analytic $\Gamma$-invariant graph which is a subset of the orbit equivalence relation. Let $I$ be the $\gs$-ideal generated by Borel $G$-anticliques. Let $A\subset X$ be an analytic set. Exactly one of the following occurs:

\begin{enumerate}
\item $A\in I$;
\item there is a continuous injective function $h\colon \cantor\to A$ which is a homomorphism of the Silver graph to $G$.
\end{enumerate}
\end{theorem}

Note that the theorem does \emph{not} say that the quotient poset $P_I$ is densely isomorphic to the Silver forcing. For that conclusion, one would need to get a \emph{reduction} of the Silver graph to $G$ as opposed to just a homomorphism.

\begin{proof}
It is clear that (2) implies the negation of (1), since the $h$-preimages of Borel $G$-anticliques are Borel Silver anticliques, which are meager in $\cantor$, and so countably many of them do not cover the whole space $\cantor$ by the Baire category theorem.

To see why the negation of (1) implies (2), let $A\subset X$ be an $I$-positive analytic set. 
Construct a fusion sequence $\langle B_n, \gamma_n\colon n\in\gw\rangle$ below $A$ such that for each $n\in\gw$ and each $x\in B_{n+1}$, $\langle x, \gamma_n\cdot x\rangle\in G$. To perform the induction step, suppose that the set $B_n$ has been found.
The set $C=\{x\in B_n\colon \exists \gamma\in \Gamma\colon \gamma\cdot x\in B_n\land \langle x, \gamma\cdot x\rangle\in G\}$ is Borel and $I$-positive, since its complement in $B$ is a $G$-anticlique and therefore in the ideal $I$.
Using the $\gs$-additivity of the $\gs$-ideal $I$, it is possible to thin out the set $C$ so that for some specific $\gamma=\gamma_n$ it is the case that for all $x\in C$, $\langle x, \gamma\cdot x\rangle\in G$ and $\gamma\cdot x\in B_n$ hold.
Now use the $\gs$-additivity of the ideal $I$ again to find a compact $I$-positive subset $B_{n+1}$ satisfying the small diameter demands of Definition~\ref{fusiondefinition}.

It is now easy to show that the map $h:\cantor\to A$ resulting from the fusion sequence is a homomorphism of the Silver graph to the graph $G$. Suppose that $y_0, y_1\in\cantor$ are Silver connected points, and let $n\in\gw$ be the unique number such that $y_0(n)=0$ and $y_1(n)=1$. Let $t=y_0\restriction n$, and  let  $z\in\cantor$ be such that $z(m)=0$ for all $m\leq n$ and $z(m)=y_0(m)$ otherwise. It is immediate that $z\in B_{m+1}$. Thus, $\langle h(z), \gamma_n\cdot h(z)\rangle\in G$ and by the invariance of the graph $G$, $\langle \gamma_t\cdot h(z) , \gamma_t\cdot\gamma_n\cdot h(z)\rangle\in G$. However, these are exactly the points $f(y_0), f(y_1)$ in the bracket. Thus, the function $f$ is a homomorphism of the Silver graph to $G$ as desired.
\end{proof}

\begin{corollary}
\label{invariantgraphcorollary}
Let $X$ be a Polish space, $\Gamma$ a countable group acting on it, and let $G$ be an analytic $\Gamma$-invariant graph which is a subset of the orbit equivalence relation. Let $I$ be the $\gs$-ideal generated by Borel $G$-anticliques. In the $P_I$-extension:

\begin{enumerate}
\item there is exactly one real degree over the ground model;
\item if $H$ is an acyclic closed graph on a Polish space $Z$ in the ground model, every point of $Z$ belongs to a ground model coded compact $Z$-anticlique.
\end{enumerate}. 
\end{corollary}

\noindent The conclusion cannot be generalized much farther. If the $\gs$-ideal is generated by more than one invariant graphs, then there may be more real degrees than one, even though generation by finitely many graphs results in finitely many degrees. A good example is provided by the product of two copies of the Silver forcing, which is generated by two invariant graphs by Theorem~\ref{producttheorem}. If hypergraphs of higher arity than two are used, there again may be more real degrees. Even when the assumptions of the theorem are satisfied, the quotient poset will not add a minimal forcing extension, as an intermediate $\gs$-closed extension typically exists by Theorem~\ref{intermediatetheorem}.

\begin{proof}
I will start with (1). Suppose that $B\in P_I$ is a condition and $\tau$ is a $P_I$-name for an element of $\cantor$. Strengthening the condition $B$ if necessary, one can find a Borel (even continuous, since the poset $P_I$ is bounding by Theorem~\ref{boundingtheorem}) function $g\colon B\to\cantor$ such that $B\Vdash\tau=\dot g(\dotxgen)$. Let $f\colon\cantor\to B$ be a continuous injective homomorphism of the Silver graph to $G$ posited by Theorem~\ref{firstdichotomytheorem}. It is well-known that the Silver poset adds a minimal real degree, and so there is a Silver-positive compact set $C\subset\cantor$ such that the function $g\circ f$ is either one-to-one or constant on $C$. The set $f''C\subset B$ is compact and $I$-positive since the $f$-preimages of $G$-anticliques are Silver anticliques.
The function $g$ is then either one-to-one or constant on the set $f''C$. In the former case, the condition $f''C\in P_I$ forces $\tau$ to be equal to the unique point in the range of $g$.
In the latter case, the condition $f''C$ forces that $\dotxgen$ can be recovered from $\tau$ as the unique $g$-preimage of $\tau$ in $f''C$, and therefore the model $V[\tau]$ is equal to the whole generic extension.

The proof of (2) is similar, using Corollary~\ref{silveracycliccorollary}.
\end{proof}

\subsection{Closed graphs}
\label{closedsubsection}

It turns out that among the forcings generated by a single invariant graph, the Silver forcing is the only one which adds a point that falls out of all compact ground model anticliques for some locally countable closed graph. In particular, if $G$ is an invariant locally countable closed graph, then the quotient poset is densely naturally isomorphic to the Silver forcing. This is an immediate conclusion of the following dichotomy:

\begin{theorem}
\label{seconddichotomytheorem}
Let $X$ be a Polish space, $\Gamma$ a countable group with a Borel action on $X$ and $G$ an analytic graph on $X$ invariant under the action such that every $G$-edge consists of a pair of orbit equivalent points. Let $I$ be the $\gs$-ideal on the space $X$ generated by Borel $G$-anticliques. For every analytic set $A\subset X$, exactly one of the following occurs:

\begin{enumerate}
\item $A\in I$;
\item there is a continuous injection $h\colon\cantor\to A$ which is a reduction of the Silver graph to $G$;
\item the condition $A$ (as a nonzero element of the completion of the poset $P_{I_G}$) forces that for every locally countable closed graph $H$ on a Polish space $Z$, every element of $Z$ belongs to a ground model coded compact $Z$-anticlique.
\end{enumerate}
\end{theorem}

\begin{proof}
It is immediate that the three options are exclusive, since if $h$ is an injection as in (2), then $h$ transports the Silver ideal to the ideal generated by Borel $G$-anticliques, and so it generates a natural isomorphism of the Silver forcing with the poset $P_I$ below $\rng(h)$. Thus, assume that (1) fails, and argue that (2) or (3) has to occur; this will complete the proof of the theorem. The dividing line revolves around the following question: is it true that for every Borel $I$-positive set $B\subset A$ there is a point $x\in B$ such that the set $\{y\in B\colon \langle x, y\rangle\in G\}$ has uncountable closure?

If the answer to the question is affirmative, then (3) occurs as proved in Theorem~\ref{closedanticliquetheorem}(2). Note that the poset $P_I$ adds a minimal degree by Corollary~\ref{invariantgraphcorollary}. Thus, suppose that the answer is negative, as witnessed by some Borel $I$-positive set $B\subset A$, and work to get a continuous injection $h\colon\cantor\to B$ which is a reduction of the Silver graph to $G$. By induction on $n\in\gw$ build a fusion sequence $\langle B_n, \gamma_n\colon n\in\gw\rangle$ for the ideal $I$ below the set $B$ such that for every $n$, the intersection of $G$ with $B_{n+1}\times\gamma_n B_{n+1}$ is exactly the set
$\{\langle x, \gamma_n\cdot x\rangle\colon x\in B_{n+1}\}$.

Once the induction is performed, the resulting function $h\colon \cantor\to X$ will be the desired reduction of the Silver graph to $G$. To see this,  suppose that $y_0, y_1\in\cantor$ are distinct points and let $n$ be the smallest number where they differ; say $y_0(n)=0$ and $y_1(n)=1$. Let $t=y_0\restriction n$ and let $z\in\cantor$ be a point such that $z(m)=0$ for $m\leq n$ and $z(m)=y_0(m)$ for $m>n$. Then $h(z)\in B_{n+1}$ and by the induction demand, $\gamma_n\cdot h(z)$ is the only point in $\gamma_n\cdot B_{n+1}$ which is $G$-connected with $h(z)$. Observe that $\gamma_t\cdot h(z)=h(y_0)$, and by invariance of the graph $G$, $\gamma_t\cdot \gamma_n\cdot h(z)$ is the only point in $\gamma_t\cdot\gamma_n\cdot B_{n+1}$ which is $G$-connected to $\gamma_t\cdot h(z)=h(y_0)$.
Thus, if $y_1(m)=y_0(m)$ for all $m>n$, then $h(y_1)=\gamma_t\cdot\gamma_n\cdot h(z)$ is $G$-connected to $h(y_0)$. On the other hand, if there is a number $m>n$ such that $y_1(m)\neq y_0(m)$, then $h(y_1)$ is an element of $\gamma_t\cdot\gamma_n\cdot B_{n+1}$ which is distinct from $\gamma_t\cdot\gamma_n\cdot h(z)$ and therefore is not $G$-connected to $h(y_0)$.

To perform the induction, start with an arbitrary compact $I$-positive set $B_0\subset B$ of small diameter. Suppose that the set $B_n$ has been found and work to produce $\gamma_n\in \Gamma$ and $B_{n+1}\subset B$. By the case assumption, for every point $x\in B$, the closure of the set $N_x=\{y\in B_n\colon x=y\lor y\mathrel G x\}\subset X$ is countable. The boundedness theorem says that there is a countable ordinal $\ga\in\gw_1$ such that all  of these countable closed sets have Cantor--Bendixson rank $<\ga$. Now, for every ordinal $\gb\in\ga$, every group element $\gd\in\Gamma$ and every basic open set $O\subset X$, let $C_{\gd,\gb, O}=\{x\in B_n\colon\gd\cdot x\in B_n$ and $\gd\cdot x\mathrel G x$ and the rank of $\gd\cdot x\in O$ in the closure of the set $N_x$ is equal to $\gb$, and there are no other points of $N_x$ of rank $\geq\gb$ in the set $O\}$. The sets $C_{\gd, \gb, O}\subset B$ are Borel, and the set $B\setminus C_{\gd, \gb}\setminus\bigcup_{\gd, \gb, O}C_{\gd, \gb}$ is a $G$-anticlique. It follows that one of the sets $C_{\gd, \gb, O}$ must be $I$-positive; select one with minimal ordinal $\gb$. Let $C=C_{\gd, \gb, O}\setminus\bigcup\{C_{\gd', \gb', O'}\colon \gd'\in\gd, \gb'\in\Gamma, O'\subset X$ basic open$\}$ and let $\gamma_n=\gd$. A review of the definitions shows that the set $C$ works as required. To complete the induction step, it is not difficult to refine the set $C$ to $B_{n+1}$ to satisfy the diameter demand of Definition~\ref{fusiondefinition}. This completes the induction step and the proof.
\end{proof}

\subsection{Open graphs}
\label{opensubsection}

 \begin{theorem}
\label{thirddichotomytheorem}
Let $\Gamma$ be a countable group acting on a Polish space $X$ in a continuous free way. Let $H$ be an invariant open graph on $X$, and let $G$ be the intersection of $H$ with the orbit equivalence relation. Let $I$ be the $\gs$-ideal generated by Borel $G$-anticliques. For every analytic set $A\subset X$, exactly one of the following occurs:

\begin{enumerate}
\item $A\in I$;
\item there is a continuous injection $f\colon\cantor\to A$ reducing the Vitali equivalence relation to $G$.
\end{enumerate}
\end{theorem}

\noindent A mildly interesting case of such graphs arises when $\Gamma$ acts on the space $X$ by isometries, $O\subset\mathbb{R}$ is an open set, and the graph $H$ connects points $x, y$ if their distance belongs to the set $O$. Of course, the theorem exactly says that the resulting quotient poset is either trivial or below a dense set of conditions, it is naturally isomorphic to the Vitali forcing.

\begin{proof}
As before, (2) implies the negation of (1) since the $f$-preimages of Borel $G$-anticliques are Borel Vitali anticliques and as such meager in $\cantor$. To show that the negation of (1) implies (2), let $M$ be a countable elementary submodel of a large enough structure containing all objects named in the assumptions of the theorem.  By induction on $n\in\gw$ build a fusion sequence $\langle B_n, \gamma_n\colon n\in\gw\rangle$ so that 

\begin{itemize}
\item[(*)] $B_{n+1}\times\gamma_n\cdot B_{n+1}\subset H$;
\item[(**)] whenever $t, s\in 2^n$ and $\gb\neq 1$ is among the first $n$ elements of $\Gamma$ in some fixed enumeration of $\Gamma$, then $\gamma_s\cdot \gamma_n B_{n+1}\cap\gamma_t\cdot B_{n+1}=0$.
\end{itemize}

To show how the induction step is performed, suppose that the set $B_n$ and the group elements $\gamma_i\in\Gamma$ for $i<n$ have been found. The first order of business is to find a Borel $I$-positive set $C\subset B_n$ such that whenever $t, s\in 2^n$ and $\gb\neq 1$ is among the first $n$ many elements of $\Gamma$, then $\gamma_t\gb\gamma_s^{-1}\cdot C\cap C=0$ or else $\gamma_t\gb\gamma_s=1$. This is possible since the $\gs$-ideal $I$ is countably additive. Having produced $C$, find an $I$-positive Borel set $D\subset C$ and a group element $\gamma_n$ such that for all $x\in D$, $\gamma_n\dot x\in C$ and moreover $x\mathrel G\gamma_n\cdot x$ holds. By the countable additivity of the ideal $I$ and the fact that the graph $G$ is open, it is possible to thin the set $D$ further so that there are open sets $O, P\subset X$ such that $O\times P\subset H$ and $D\subset O$ and $\gamma_n\cdot D\subset P$. Finally, thin out the set $D$ to $B_{n+1}$ so that the diameter demand of Definition~\ref{fusiondefinition} are satisfied.

Suppose that the induction has been carried out; I claim that the associated function $h\colon \cantor\to B$ is a reduction of the Vitali equivalence to the graph $G$. To see this, suppose that $y_0, y_1\in\cantor$ are Vitali-related but not equal. Then let $n\in\gw$ be some number such that $\forall i\geq n\ y_0(i)=y_1(i)$ holds and let $z\in\cantor$ be such that $z(i)=0$ for $i<n$, and $z(i)=y_0(i)$ for $i\geq n$. Then, $h(y_0)=\gamma_{y_0\restriction n}\cdot h(z)$ and $h(y_1)=\gamma_{y_0\restriction n}\cdot\gamma_n\cdot h(z)$ and so the points $h(y_0), h(y_1)\in X$ belong to the same $\Gamma$-orbit. To see that the are $H$-related, let $n\in\gw$ be the smallest number such that $y_0(n)\neq y_1(n)$, say $y_0(n)=0$ and $y_1(n)=1$. Let $t=y_0\restriction n$. Then by the definitions, $h(y_0)\in \gamma_t\cdot B_{m+1}$ and $h(y_1)\in \gamma_t\cdot\gamma_n\cdot B_{m+1}$. The product $B_{m+1}\times \gamma_m\cdot B_{m+1}$ is a subset of the graph $H$ by (*), the graph $H$ is invariant under the group action, and so $h(y_0)\mathrel Hh(y_1)$ as requested.

Suppose on the other hand that $y_0, y_1\in\cantor$ are not Vitali related; I will show that $f(y_0)$ and $f(y_1)$ are not in the same $\Gamma$-orbit. Suppose towards a contradiction that they are, and let $\gb\in\Gamma$ be such that $\gb\cdot f(y_0)=f(y_1)$. Let $n\in\gw$ be a number such that both $\gb, \gb^{-1}$ are among the first $n$ many elements of $\Gamma$ and $y_0(n), y_1(n)$ differ, say $y_0(n)=0$ and $y_1(n)=1$. Let $t=y_0\restriction n$ and $s=y_1\restriction n$. Note that $h(y_0)\in \gamma_t\cdot B_{n+1}$ and $h(y_1)\in \gamma_s\gamma_n\cdot B_{n+1}$, and so $h(y_1)=\gb\cdot h(y_0)\in\gb\gamma_t\cdot B_{n+1}\cap\gamma_s\gamma_n\cdot B_{n+1}$ and
$\gamma_s^{-1}\gb\gamma_t\cdot B_{n+1}\cap\gamma_n\cdot B_{n+1}\neq 0$, contradicting (**).
\end{proof}

\subsection{Group-related graphs}
\label{groupsubsection}

All of the previous examples are unsophisticated in the sense that they do not use the combinatorics of the acting group in any way. At the same time, there is a whole universe of actionable forcings which seem to reflect various group features in an intriguing way. In this section, I introduce a rich class of examples and show that a small subclass trivializes.

\begin{definition}
Let $X$ be a Polish space, $\Gamma$ a countable group acting on $X$ in a Borel way, and let $\Delta\subset\Gamma$ be a set closed under conjugation by elements of the group $\Gamma$, and under inverse. The graph $G_{\Delta}$ connects points $x, y$ if there is $\gd\in\Delta$ such that $\gd\cdot x=y$.
\end{definition}

The graph $G_{\Delta}$ is invariant under the group action and as such generates an actionable ideal and an actionable forcing. In most cases, the associated $\gs$-ideal is nontrivial: for example, if $\Gamma$ acts naturally on $2^\Gamma$ and $\Delta\subset\Gamma$ is any infinite set, the Borel anticliques of the graph $G_{\Delta}$ are meager. The basic examples below use the group $\Gamma$ of all finite subsets of $\gw$ with the symmetric difference operation, action on $X=\power(\gw)$ by the symmetric difference.

\begin{example}
The Silver forcing is obtained from $\Delta=$the set of all singletons.
\end{example}

\begin{example}
The Vitali forcing is obtained from $\Delta=\Gamma$.
\end{example}

\begin{example}
\label{vitalioddexample}
The Vitali-odd forcing forcing is obtained from $\Delta=$the set of all finite sets of odd size. The Vitali-odd forcing is distinct from both Silver forcing and Vitali forcing. In the contradistinction to the Vitali forcing, it adds an independent real. Similarly to the Vitali forcing, in the Vitali-odd extension the ground model coded compact anticliques of any closed graph without perfect cliques still cover their domain space, see the proof of Corollary~\ref{neverclosedcorollary}. This shows that the Vitali-odd forcing is distinct from the Silver forcing.
\end{example}

In general, it seems that the analysis of the quotient forcing depends on deep combinatorial properties of the set $\Delta$ in question.  I will prove a theorem which shows that if the set $\Delta$ belongs to a certain natural ideal, the quotient poset is the Silver forcing.

\begin{definition}
Let $\Gamma$ be a countable group.

\begin{enumerate}
\item A set $\Delta\subset\Gamma$ is \emph{secluded} if for every finite set $a\subset\Gamma$, for all but finitely many $\gg\in\Gamma$ it is the case that $|(a\cdot\gg\cdot a)\cap\Delta|\leq 1$;
\item the secluded ideal on $\Gamma$ is the ideal generated by secluded sets.
\end{enumerate}
\end{definition}

\noindent For example, an infinite subset of the integers is secluded if the size of the holes in it tends to infinity. The set of singletons is secluded in the group of finite subsets of $\gw$ with the symmetric difference. It is not difficult to see that no infinite group $\Gamma$ belongs to its secluded ideal. Everyone's favorite question asks whether the set of prime numbers belongs to the secluded ideal on $\gw$; recent results in number theory show that it is not one of the secluded generators.
The quotient posets resulting from sets in the secluded ideal trivialize in the sense of the following dichotomy:

\begin{theorem}
\label{fourthdichotomytheorem}
Let $X$ be a Polish space, $\Gamma$ be a countable abelian group acting on $X$ in a Borel and free way, and let $\Delta\subset\Gamma$ be a set in the secluded ideal, closed under inverse. Let $G_{\Delta}$ be the associated graph and let $I$ be the $\gs$-ideal on $X$ generated by the Borel $G_{\Delta}$ anticliques. If $A\subset X$ is an analytic set, exactly one of the following occurs:

\begin{enumerate}
\item $A\in I$;
\item there is a continuous function $f\colon \cantor\to A$ which is a reduction of the Silver graph to $G_{\Delta}$.
\end{enumerate}
\end{theorem}

\begin{proof}
The difficulty here resides in handling sets in the secluded ideal as opposed to just secluded sets. For this, I need a corollary of the Hales--Jewett theorem. To state it, let $U$ be an arbitrary set and write $Q$ for the collection of nonempty finite partial functions from $\gw$ to $U$, and for elements $q, r\in Q$ write $r<q$ if $r\neq q$ and for some $n\in\gw$, $q=r\restriction [n, \gw)$.

\begin{claim}
\label{halesjewettclaim}
Suppose that $C\subset Q$ is a set such that for every $n\in\gw$ $\{\langle n, 1\rangle\}\in C$, and the length of $<$-chains in $C$ is bounded. Then there is a sequence $\langle q_n\colon n\in\gw\rangle$ of elements of $C$ such that $\max\dom(q_n)<\min\dom(q_{n+1})$
and for each $n\in\gw$ and each $r<q_n$ with $\dom(r)\subset\bigcup_{m\leq n} q_m$, $r\notin C$ holds.
\end{claim}

\begin{proof}
Let $k\in\gw$ be a number such that every $<$-chain in $C$ is shorter than $k$. For every nonempty finite set $a\subset\gw$, let $\pi(a)\in k$ be the maximal length of a $<$-chain of elements of $C$ whose domain is a subset of $a$, and let $q(a)$ be the $<$-smallest
element of one such chain of maximal possible length. Note that $q(a)$ is well-defined by the initial assumption on the set $A$. Use the Hales--Jewett theorem to find a number $j\in k$ and a sequence of nonempty finite sets $\langle a_n\colon n\in\gw\rangle$ such that 
$\max(a_n)<\min(a_{n+1})$ and for each nonempty finite union $a$ of sets on the sequence, $\pi(a)=j$. Let $q_n=q(a_n)$ and note that the conclusion of the theorem is satisfied. If $r<q_n$ were an element of $C$ such that  $\dom(r)\subset\bigcup_{m\leq n} q_m$, then
the $<$-chain of length $j$ of elements of $C$ whose domain is a subset of $a_n$ and whose $<$-smallest element is $q_n$ could be extended by $r$, showing that $\pi(\bigcup_{m\leq n}a_n)>j$, contradicting the $\pi$-homogeneity of the sequence $\langle a_n\colon n\in\gw\rangle$.
\end{proof}

To prove the theorem, it is only necessary to show that negation of (1) implies (2). Fix a positive analytic set $A\subset X$ and similarly to Theorem~\ref{firstdichotomytheorem} find a fusion sequence $\langle B_n, \gamma_n\colon n\in\gw\rangle$ below it such that for each $n\in\gw$, $\gamma_n\in\Delta$. It is not necessarily true that the associated fusion map is a reduction of the Silver graph to $G$; I have to thin the fusion sequence out  in a quite sophisticated way.

Let $U=\{1, -1\}$ and let $Q$ be the set of all partial functions from $\gw$ to $U$ with finite domain. For each finite function $q\in Q$ write $\gamma(q)=\prod_n\gamma_{n}^{q(n)}$
where $n$'s  list the elements of $\dom(q)$ in increasing order.
Write $\Delta=\bigcup_{j\in k}\Delta_j$ for some partition of $\Delta$ into secluded sets. Use the seclusion to find an infinite set $b\subset\gw$ such that 

\begin{itemize}
\item[(*)] for every $n\in b$ and all functions $q_0, q_1\in Q$ whose domain is a subset of $b\cap n$, if $\gamma(q_0)\cdot\gamma_n$ and $\gamma(q_1)\cdot\gamma(q_0)\cdot\gamma_n$ both belong to $\Delta$, then they belong to different secluded pieces of $\Delta$, and the same for $\gamma_n^{-1}$. 
\end{itemize}

To simplify the notation, assume that $b=\gw$. Let $C\subset Q$ be the set of all $q\in Q$ such that $\gamma(q)\in\Delta$. Note that if $c$ is a $<$-chain in the set $C$ then the elements $\gamma(q)$ for $q\in C$ must come from different secluded pieces of $\Delta$ by (*) and so the length of $c$ is not greater than $k$. By Claim~\ref{halesjewettclaim}, there is a sequence $\langle q_n\colon n\in\gw\rangle$ of elements of $C$ such that $\max\dom(q_n)<\min\dom(q_{n+1})$
and for each $n\in\gw$ and each $r<q_n$ with $\dom(r)\subset\bigcup_{m\leq n} q_m$, $r\notin C$ holds. For each $n\in\gw$, write $\gd_n=\gamma(q_n)$; thus $\gd_n\in\Delta$. The following is the key property of this sequence:

\begin{itemize}
\item[(**)] For a finite function $q\in Q$, the product $\prod_n\gd_n^{q(n)}$, where $n$'s list the domain of $q$ in the increasing order,
belongs to the set $\Delta$ if and only if $|q|=1$.
\end{itemize}

Now, I am ready to find the continuous embedding of the Silver graph into $G_\Delta$. Let $x\in X$ be the resulting point of the fusion sequence. Let $c=\{m\colon\exists n\ m\in\dom(q_n)\land q_n(m)={-1}\}$. Let $y\in X$ be the point $\prod_{m\in c}\gamma_m\cdot x$. Now, for every point $z\in\cantor$ let $f(z)\in X$ be the point $\prod_{z(n)=1}\gd_n\cdot y$. This is the desired embedding; I just need to verify its properties.

The commutativity of the group $\Gamma$ implies that for each $z\in\cantor$, $f(z)=\prod\{\gamma_m\colon m\in\bigcup_nb_n\}\cdot x$ where $b_n=\{m\in\gw\colon q_n(m)=-1\}$ if $z(n)=0$ and $b_n=\{m\in\gw\colon q_n(m)=1\}$ if $z(n)=1$.
Thus, if $z_0, z_1$ are not Vitali-equivalent, the points $f(z_0)$ and $f(z_1)$ are not in the same $\Gamma$-orbit and so not $G_\Delta$-connected. If $z_0, z_1$ are Vitali connected, then write $\gamma=\prod_n\gd_n^{q(n)}$ where $n$'s list the domain of $q$ in the increasing order, and $q\in Q$ is the function with $\dom(q)=\{n\in\gw\colon z_0(n)\neq z_1(n)\}$ and $q(n)=1$ iff $z_0(n)=0$. It is immediate that $\gamma\cdot f(z_0)=f(z_1)$ holds, and so (**) implies that $f(z_0), f(z_1)$ are $G_{\Delta}$-connected just in case the points $z_0, z_1$ are Silver-connected as desired.
\end{proof}

\begin{corollary}
Under the assumptions of Theorem~\ref{fourthdichotomytheorem}, the quotient poset $P_I$, if nontrivial, is densely isomorphic to the Silver forcing.
\end{corollary} 

\subsection{The KST forcing}
\label{kstsubsection}

If $\Gamma$ is a countable group continuously acting on a Polish space $X$ and $G$ is an analytic hypergraph on $X$ such that all edges of $G$ consist of pairwise orbit equivalent points, one can form the actionable family $\cG=\{\gamma\cdot G\colon\gamma\in\Gamma\}$
and consider the resulting proper quotient poset. My understanding of posets of this kind is limited at this point. In this subsection, I will prove a dichotomy showing that the poset derived from the KST graph in this way has a central position. 

\begin{definition}
Let $\Gamma$ be a countable group acting freely and continuously on a Polish space $X$ and let $\cG$ be a countable family of hypergraphs such that each edge in each of them consists of pairwise orbit equivalent points. Say that the family
$\cG$ is \emph{free} over the action if the action naturally extends to a free action of the group on $\cG$, and for any two distinct hypergraphs $G_0, G_1\in\cG$, the Borel chromatic number of $G_0\cap G_1$ is countable.
\end{definition}

\begin{example}
\label{acyclicexample}
Let $\Gamma$ be a countable group acting freely and continuously on a Polish space $X$. Let $G$ be an acyclic analytic graph spanning the orbit equivalence relation, with uncountable Borel chromatic number. The family $\cG=\{\gamma\cdot G\colon \gamma\in\Gamma\}$ is free over the action.
\end{example}

\begin{proof}
It is enough to show that for every $\gamma\in\Gamma$, the graph $G\cap\gamma\cdot G$ has countable Borel chromatic number. To see this, for each $n\in\gw$, let $B_n=\{x\in X\colon$the $G$-distance between $x$ and $\gamma^{-1}\cdot x$ is equal to $n\}$. Note that as the graph $G$ spans the orbit equivalence relation, $X=\bigcup_nB_n$ holds. It is not difficult to verify that the sets $B_n$ are Borel. Thus, it will be enough to show that each element of $B_n$ has at most two $G\cap\gamma\cdot G$-neighbors in the set $B_n$. 

Suppose for contradiction that this is not the case and let $x, y_0, y_1, y_2\in B_n$ be distinct points such that $x$ is $G\cap\gamma\cdot G$-connected to all $y_0, y_1, y_2$. Look at the unique injective $G$-path $p$ from $x$ to $\gamma^{-1}\cdot x$. It can use at most one of the points $y_0, y_1, y_2$, since they are all $G$-connected to  $x$, and at most one of the points $\gamma^{-1}\cdot y_0, \gamma^{-1}\cdot y_1, \gamma^{-1}\cdot y_2$ since they are all $G$-connected to $\gamma^{-1}\cdot x$. Suppose for definiteness that $y_2$, $\gamma^{-1}\cdot y_2$ are not used in the path $p$. Then the unique injective $G$-path connecting $y_2$ to $\gamma^{-1}\cdot y_2$ contains $p$ as a proper subset and so is longer than $p$. This contradicts the assumption that both $x, y_2$ are in the set $B_n$.
\end{proof}

I will prove a dichotomy which shows that the quotient forcings associated with families of analytic graphs which are free over some free action of the group $\Delta=$finite subsets of $\gw$ with the symmetric difference operation are all equivalent from the forcing point of view.
It may be that changing the acting group yields different forcings. In order to state the dichotomy theorem, I need to define a suitable canonical object.

\begin{definition}
Let $\gw=\bigcup_na_n$ be a partition into infinite sets. For each $n\in\gw$, choose a dense set $b_n\subset\bintree$ such that for each $t\in b_n$, $|t|\in a_n$ holds, and for each $s\neq t\in b_n$, $|s|\neq |t|$ holds. Write $H_n$ for the graph on $\cantor$ connecting
$x, y$ if there is exactly one $m$ such that $x(m)\neq y(m)$ and in addition, $x\restriction m\in b_n$. Then $\cH$ is the $\Delta$-closure of the set $\{H_n\colon n\in\gw\}$.
\end{definition}

\noindent It is not hard to see that for each $n\in\gw$, Borel $H_n$-anticliques must be meager. Thus, the $\gs$-ideal $\cH$ on $\cantor$ consists only of meager sets. The following key dichotomy inserts the quotient forcing $P_{I_{\cH}}$ into many similar quotients.

\begin{theorem}
Let $\Delta$ act freely and continuously on a Polish space $X$ and let $\cG$ be a countable family of closed graphs, free over the action. If $A\subset X$ is an analytic set,
exactly one of the following occurs:

\begin{enumerate}
\item $A\in I_{\cG}$;
\item there is a continuous, injective near reduction of $\cH$ to $\cG$ with the range included in $A$.
\end{enumerate}
\end{theorem}

\begin{proof}
To prove the dichotomy, note that (2) implies the negation of (1), since the near reduction in (2) transports the ideal $I_{\cH}$ to $I_{\cG}$ and $\cantor\notin I_{\cH}$--see the reasoning in Corollary~\ref{uniquenesscorollary} below. The hard part is showing that the negation of (1) implies (2). Before I start the construction, observe that every analytic $I_{\cG}$-positive set has a compact $I_{\cG}$-positive subset. To see this, use Theorem~\ref{localizationtheorem} to show that the quotient poset has the $2$-localization property, in particular it is bounding, and then use the standard characterization of the bounding property \cite[Theorem 3.3.2]{z:book2}.

Let $A\subset X$ be an analytic $I$-positive set. To produce the injection $h$, I will in fact produce a fusion sequence $\langle B_n, \gamma_n\colon n\in\gw\rangle$ below the set $A$ for the ideal $I$ such that the injection $h$ is the associated fusion map. 
The fusion sequence will be constructed by induction on $n\in\gw$. I will write $\Delta_n$ for the subgroup of $\Delta$ generated by the elements $\{\gamma_m\colon m\in n\}$.  Also, in the course of the induction, I will select some graphs $G_n\in\cG$ and write $\cG_n$ for the $\Delta_n$-orbit of the set $\{G_m\colon m\in n\}$.
The induction hypothesis:

\begin{itemize}
\item with the natural action of $\Delta$ on $\cG$, the set $\{G_m\colon m\in n\}$ consists of $\Delta_n$-unrelated graphs; 
\item whenever $G\in\cG$ is $n$-th  element of $\cG$ in some fixed enumeration, then $G\in\cG_{n+1}$; 
\item whenever $m\in n$ is a number and $t\in 2^n$ is a string, then the intersection of $B_{n+1}\times \gamma_n\cdot B_{n+1}$ with the graph $\gamma_t\cdot G_m$ is either empty, or is equal to the set $\{\langle x, \gamma_n\cdot x\rangle\colon x\in B_{n+1}\}$. The latter case occurs precisely when $t\in b_m$.
\end{itemize}

The induction is not difficult to perform. To begin, let $B_0\subset B$ be any compact $I$-positive subset. To perform the induction step, suppose that $B_n$ and $\gamma_m, G_m$ for $m\in n$ have been found. First, find the graph $G_n$. If the $n$-th graph of $\cG$ in some fixed enumeration is of the $\Delta_n$-orbit of $\{G_m\colon m\in n\}$, then just let $G_n\in\cG$ be some graph which is not in this orbit; otherwise, let  $G_n$ be the $n$-th graph of $\cG$. Now, I must find $B_{n+1}$ and $\gamma_n$. The treatment now divides into two cases.

\noindent\textbf{Case 1.} There is $m\in n$ such that the set $2^n\cap b_m$ is nonempty, containing some string $t\in 2^n$. Note that this $m$ must be unique and so is $t$. Now, I let $C=B_n$ and go through a process of repeatedly shrinking the set $C$ to smaller Borel $I$-positive sets in the following steps:

\begin{itemize}
\item use the freeness of the action on $X$ to shrink the set $C$ so that for all the finitely many $\gd\in\Delta$ such that $G_n\in \gd\cdot\cG_n$, it is the case that $C\cap \gd\cdot C=0$;
\item use the freeness of the family $\cG$ to shrink $C$ so that it is an anticlique in the intersection of any two distinct graphs in $\cG_n$;
\item  use the assumption that the graphs are closed and the argument from Theorem~\ref{seconddichotomytheorem} to shrink the set $C$ and find a group element $\gamma_n\in\Delta$ such that $\gamma_n\cdot C$ is still a subset of the version of $C$ obtained in the previous step,
 and the intersection of $C\times \gamma_n\cdot C$ with the graph $\gamma_t\cdot G_m$ is exactly equal to the set $\{\langle x, \gamma_n\cdot x\rangle\colon x\in C\}$; 
\item since for all graphs $G\in\cG_n$ other than $\gamma_t\cdot G_m$ and all points $x\in C$ it is the case that $\langle x, \gamma_n\cdot x\rangle\notin G$, and all these graphs $G$ are closed, one can shrink the set
$C$ further so that the intersection of $C\times \gamma_n\cdot C$ with each of the graphs $G\in\cG_n$ distinct from $\gamma_t\cdot G_m$ is empty. 
\end{itemize}

\noindent Finally, shrink the set $C$ further to some $I$-positive compact set $B_{n+1}$ of small diameter. This completes the induction step in this case.

\noindent\textbf{Case 2.} If Case 1 fails, pick a graph $G\in\cG$ not in $\cG_n$ and proceed as in Case 1, replacing $\gamma_t\cdot G_m$ with $G$.

Suppose that the induction has been performed, let $x\in B$ be the resulting point, and let $h\colon \cantor\to A$ be the resulting continuous injective map. The map $\{n\}\mapsto \gamma_n$ from $\Delta$ to $\Delta$ induces an injective homomorphism $\xi\colon\Delta\to\Delta$. The map $H_n\mapsto G_n$ together with the homomorphism $\xi$ induce a map $\pi\colon\cH\to\cG$ given by $\pi(a\cdot H_n)=\xi(a)\cdot G_n$; this is an injection by the first item of the induction hypothesis and a surjection by the second item of the
induction hypothesis. The third item of the induction hypothesis then shows that the map $h$ is a near-reduction of $a\cdot H_n$ to $\xi(a)\cdot G_n$. This completes the proof of the theorem.
\end{proof}

\begin{corollary}
\label{uniquenesscorollary}
The quotient forcing does not depend on the choice of the action of $\Delta$ and the free family.
\end{corollary}

\begin{proof}
It is enough to show that under any condition $B\in P_{I_{\cG}}$ there is a condition $C\subset B$ such that the poset below $C$ is naturally isomorphic to $P_{I_{\cH}}$. Let $h\colon\cantor\to X$ be a map from the dichotomy theorem, with the attendant maps $\pi, \chi$.
It is enough to show that the map $h$ transports the ideal $I_{\cH}$ to the ideal $I_{\cG}$.

To see this, choose a graph $H\in\cH$. Note that if $t\in 2^{\chi(H)}$ is a binary string and $B\subset [t]$ is a set, then $B$ is an $H$-anticlique if and only if $h''B$ is a $\pi(H)$-anticlique. Therefore, if $B\subset\cantor$ is a Borel $H$-anticlique then $h''B$ is a finite union of $\pi(H)$-anticliques and vice versa, if $B\subset X$ is a Borel $\pi(H)$-anticlique then $h^{-1}B$ is a union of finitely many Borel $H$-anticliques. It follows that the function $h$ transports the ideal $I_{\cH}$ to $I_{\cG}$.
\end{proof}

\begin{corollary}
The quotient forcing is homogeneous.
\end{corollary}

\begin{corollary}
The quotient forcing is isomorphic to the countable support product of countably many copies of itself.
\end{corollary}

\begin{proof}
Use the computation of the product in Corollary~\ref{productgraphcorollary} to conclude that the product ideal is generated by a family of closed hypergraphs which is free with respect to some action of the finite support product of countably many copies of $\Delta$, which is isomorphic to $\Delta$. A reference to Corollary~\ref{uniquenesscorollary} then concludes the proof.
\end{proof}

As one consequence of the dichotomy, the following definition makes sense:

\begin{definition}
\label{kstposetdefinition}
The \emph{KST forcing} is the unique poset $P_{I_\cG}$ where $\cG$ is a free family of closed graphs over some action of $\Delta$.
\end{definition}

Using Example~\ref{acyclicexample}, one could for example use the shifts of the KST graph to generate the free family.
As for the forcing properties of the KST poset, it has the 2-localization property by Theorem~\ref{localizationtheorem} and it adds an independent real by Theorem~\ref{independenttheorem}. The main difference between the KST forcing and Silver forcing is that in the Silver extension, every point of $\cantor$ belongs to
a ground model coded compact KST-anticlique by Corollary~\ref{silveracycliccorollary} since the KST graph is closed and acyclic. At the same time, the generic point of the KST forcing is designed to fall out of all ground model coded Borel KST-anticliques. Another key property of the KST forcing
is the following. By \cite{kechris:chromatic}, for every analytic graph $H$ on a Polish space $Z$ of uncountable Borel chromatic number there is a continuous homomorphism $h\colon\cantor\to Z$
of the KST graph to $H$. It follows that for each such a graph $H$, the KST forcing adds a point of $Z$ (namely, the image of the generic point under the function $h$) which falls out of all ground model coded Borel $H$-anticliques. Clearly, for analytic graphs with countable Borel chromatic number no forcing can perform such a job, so the KST poset is in some sense extreme. 

\section{Operations}
\label{othersection}

The purpose of this section is to consider natural operations on hypergraphs and investigate whether they yield natural operations on the quotient posets. The answer turns out to be a resounding yes.

\subsection{Template product}

 In order to state the results in the broadest and most ambitious context, I will need the following definitions. They are designed to capture products, iterations along well-founded, illfounded or non-linear orders and probably many other unspeakable crimes. 

\begin{definition}
A \emph{template} is a pair $\langle J, R\rangle$ where $R$ is a two-place relation on $J$ such that no element $j\in J$ is $R$-related to itself.
\end{definition}

\begin{definition}
\label{templatedefinition}
Let $\langle J, R\rangle$ be a template. Suppose that $\langle X_j, \cG_j\colon j\in J\rangle$ is a collection of Polish spaces and countable families of analytic hypergraphs on each. For each countable set $K\subset J$, define the template products:

\begin{enumerate}
\item $X_K=\prod_{j\in K}X_j$;
\item $\prod_{j\in K}^R\cG_j=\cG_K$ is the countable family of hypergraphs on $X_K$ defined by $\cG_K=\{\hat G\colon G\in \cG_j, j\in K\}$ where $\hat G$ is the set of those edges $x\in X_K^{\leq\gw}$ all of whose vertices have the same restriction to the set $\{i\in K\colon i\mathrel R j\}$ and
the set $\langle x(n)(j)\colon n\in\dom(x)\rangle\in X_j^{\leq\gw}$ is a $G$-edge;
\item $I_K$ is the $\gs$-ideal on $X_K$ $\gs$-generated by Borel sets which are anticliques in at least one of the hypergraphs in $\cG_K$ and $P_K$ is the quotient poset of Borel $I_K$-positive sets ordered by inclusion.
\end{enumerate}
\end{definition}

\noindent It should be noted that the template product $\gs$-ideals $I_K$ do not depend on the choice of the hypergraphs $\cG_j$ but only on the $\gs$-ideals that the hypergraphs $\cG_j$ generate. This can be proved using Theorem~\ref{millerfact}(3).
A definition can be stated for arbitrary $\gs$-ideals instead of just the hypergraphable ideals. However, the class of hypergraphable ideals has a major advantage: the result of the template product is again a hypergraphable ideal and therefore to some extent
tractable. The tractability is reflected in a seemingly trivial fact, which is nevertheless entirely central for the treatment of the uncountable operations:

\begin{theorem}
\label{systemtheorem}
In the set-up of Definition~\ref{templatedefinition},  if $K\subset L$ are countable subsets of $J$ and $A\subset X_K$ is an analytic set, then $A\in I_K$ if and only if $A\times X_{L\setminus K}\in I_L$.
\end{theorem}

\begin{proof}
The left-to-right direction is immediate from the definitions, since for each index $j\in K$ and each hypergraph $G\in\cG_j$, the projection map $\pi\colon X_L\to X_K$ is a homomorphism of the hypergraph $\hat G$ (as computed on the space $X_L$) to the hypergraph $\hat G$ (as computed on the space $X_K$). For the key right-to-left direction, write $M=L\setminus K$ and suppose that $A\notin I_K$ holds. Corollary~\ref{meagercorollary} then shows that there are Polish topologies $\tau_K$ on $X_K$ and $\tau_M$ on $X_M$
such that $A$ is $\tau_K$-comeager, every set in $I_K$ is $\tau_K$-meager and every set in $I_M$ is $\tau_M$-meager. It is not difficult to see then that the generating sets of the $\gs$-ideal $I_L$ are $\tau_K\times\tau_M$-meager. Thus, all
sets in the $\gs$-ideal $I_L$ are $\tau_K\times\tau_M$-meager and  cannot contain the whole set $A\times X_M$.
\end{proof}

\noindent Theorem~\ref{systemtheorem} says that in the sense of \cite[Definition 5.5.5]{z:book2}, if the set $J$ is uncountable, the collection $\{I_K\colon K\in [J]^{\aleph_0}\}$ is a centered system of ideals. It then makes sense to define the template product $P_J$ to be the limit poset: $P_J$ is the poset of those $B$ for which there is a $K=\dom(B)\in [J]^{\aleph_0}$ such that $B$ is a Borel $I_K$-positive subset of $X_K$. The ordering is defined by $C\leq B$ if $\dom(B)\subset\dom(C)$ and the projection of $C$ to $X_\dom(B)$ is a subset of $B$. The poset $P_J$ adds an element of the product $\prod_jX_j$. A priori, it is not clear what the properties of posets of this type may be or how they depend on the template $R$. In particular, the question of properness is wide open. 

The following apparent triviality is the mother of all iteration and product preservation theorems in the class of hypergraphable forcings. To appreciate its staggering strength, it is necessary to consult the applications in Subsections~\ref{productsubsection} and~\ref{iterationsubsection}. In this class of partial orders, it implies most of the countable support iteration theorems of \cite{bartoszynski:set} and generates many new ones without a single fusion argument.  It also applies in the case of illfounded iterations where the usual fusion arguments are unavailable.

\begin{theorem}
\label{preservationtheorem}
Let $P$ be a partial order.

\begin{enumerate}
\item the property $\cG\not\perp^*P$ is preserved under countable template products;
\item if $P$ is Suslin, then the property $\cG\not\perp P$ is preserved under countable template products.
\end{enumerate}
\end{theorem}

\begin{proof}
The argument depends on a general claim. Suppose that $\langle J, R\rangle$ is a template, the set $J$ is countable, and for each $j\in J$, the countable family $\cG_j$ of analytic hypergraphs on a Polish space $X_j$.
Now, fix $j\in J$, write $I_j$ be the $\gs$-ideal on $X_j$ $\gs$-generated by Borel sets which are anticliques in at least one of the hypergraphs in $\cG_j$, and write $K=\{i\in J\colon i\mathrel R j\}$.

\begin{claim}
Whenever $B\subset X$ is a Borel $I$-positive set, there is a point $y\in \prod_{i\in K}X_i$ such that the set $B_y=\{x\in X_j\colon\exists z\in B\ y, \{\langle j, x\rangle\}\subset z\}\subset X$ does not belong to the ideal $I_j$.
\end{claim}

\begin{proof}
This follows from Theorem~\ref{millerfact}(3): if for each $y\in\prod_{i\in K}X_i$ the set $B_y$ were $I_j$-small, one could find Borel sets $C_n\subset B$ and hypergraphs $G_n\in\cG_j$ for $n\in\gw$ such that $B=\bigcup_nC_n$ and for each $y\in\prod_{i\in K}X_i$,
the set $C_y$ is a $G_n$-anticlique. But then, each set $C_n$ is a $\hat G_n$-anticlique, showing that the set $B$ belongs to $I$.
\end{proof}

For (1), suppose that for every $j\in J$, $\cG_j\not\perp^*P$ holds. Let $\cG=\prod_j^R\cG_j$ be the product family
of hypergraphs on the Polish space $X=\prod_jX_j$. I need to show that $\cG\not\perp^*P$ holds. To this end, suppose that $B\subset X$ is an $I_J$-positive Borel subset of $X$, $f\colon B\to P$ is a function, $j\in J$ is an index and $G\in\cG_j$ is a hypergraph.
I must find a condition $p\in P$ which forces that there is an edge $e\in\hat G$ consisting of ground model elements of $B$ such that $f''\rng(e)$ is a subset of the generic filter. Use the claim to find a  a point $y\in \prod_{i\in K}X_i$ such that the set $B_y\subset X$
is $I_j$-positive. For each $x\in B_y$ choose $z_x\in B$ such that $y, \{\langle j, x\rangle\}\subset z_x$ and define a map $\hat f\colon B_y\to P$ by $\hat f(x)=f(z_x)$. Apply the assumption $\cG_j\not\perp^*P$ to find a condition $p\in P$ which forces that there is an edge $e\in\hat G$ consisting of ground model elements of $B$ such that $\hat f''\rng(e)$ is a subset of the generic filter. Then $\langle z_x\colon x\in\rng(e)\rangle$ is a $\hat G$-edge whose $f$-image belongs to the generic filter as desired.

The proof of (2) is similar, using Corollary~\ref{meagercorollary} and a uniformization theorem such as \cite[Proposition 2.3.4]{z:book2} to pick the points $z_x\in B$ in a Borel way.
\end{proof}

\subsection{Product of actionable ideals}
\label{productsubsection}

A good reason to love actionable ideals is their behavior under the countable support products. It turns out that the countable support product of finitary actionable ideals is obtained from a natural template product. Let $\langle \gw, R\rangle$ be the template
defined by $n\mathrel R m$ just in case $n\neq m$; this is clearly the largest template on a countable domain.

\begin{theorem}
\label{producttheorem}
Let $X_n, \cG_n$ be Polish spaces and countable actionable families of finitary analytic hypergraphs on each, for $n\in\gw$. Let $\cG=\prod_n^R\cG_n$ be the template product, a countable family of hypergraphs on $X=\prod_nX_n$, with an associated $\gs$-ideal $I$.
For every analytic set $A\subset X$, exactly one of the following occurs:

\begin{enumerate}
\item $A\in I$;
\item there are Borel $I_n$-positive sets $B_n\subset X_n$ for each $n\in\gw$ such that $\prod_nB_n\subset A$.
\end{enumerate}
\end{theorem}

\noindent The reader should consult Definition~\ref{templatedefinition} to see that the template product hypergraph family in this case is again actionable, and if $\Gamma_n$ are countable groups acting on the respective spaces $X_n$ to witness the actionability of
the families $\cG_n$, then the finite support product group $\Gamma=\prod_n\Gamma_n$ with its natural action on $\prod_nX_n$ witnesses the actionability of $\cG$.

It is instructive to note that the theorem fails very badly for hypergraphable ideals which are not actionable. Consider the ideal $J$ of countable subsets of $\cantor$. This is an ideal $\gs$-generated by all Borel anticliques of the full graph $G=[\cantor]^2$.
The quotient is the Sacks forcing, adding a minimal forcing extension, and therefore it is impossible to present the ideal in an actionable way due to Theorem~\ref{intermediatetheorem}. The ideal $K$ on $\cantor\times\cantor$ associated with the product of two copies of Sacks forcing
consists of those Borel sets which do not contain a perfect rectangle by \cite[Theorem 5.2.6]{z:book2}. This ideal is strictly larger than the $\gs$-ideal $L$ generated by the template product of two copies of the graph $G$.
This is not entirely easy to see. A good example is a Borel set $B\subset\cantor\times\cantor$ which contains an $\aleph_2\times\aleph_2$-rectangle but no perfect rectangle, produced by Shelah in \cite{Sh:522}. The set $B$ cannot be covered by any
two sets $C, D\subset\cantor\times\cantor$ such that all vertical sections of $C$ and all horizontal sections of $D$ are countable, simply because every map from $\gw_2$ to $[\gw_2]^{\aleph_0}$ contains a free pair. Thus, the set $B$ belongs to the ideal $K$ but not to $L$.
The $\gs$-ideal $L$ belongs to the heap of ideals for which I do not know if the quotient poset is proper.

The theorem has numerous consequences in the form of product preservation theorems, absorbing many repetitive fusion arguments. Here, I give a small sample.

\begin{corollary}
\label{productgraphcorollary}
In the notation of the theorem, the poset $P_I$ is the countable support product of the posets $P_{I_n}$.
\end{corollary}

\begin{corollary}
\label{aacorollary}
Suppose that $H$ is a closed graph  on a Polish space $Z$ without a perfect clique. The property $\phi(P, H)=$``every point of $Z$ in the $P$-extension belongs to a ground model coded compact $H$-anticlique'' is preserved under the countable support product of finitary actionable forcings.
\end{corollary}

\begin{proof}
Let $R_H$ be the poset associated with the hypergraph $H$ in Subsection~\ref{anticliquesubsection}. Note that the poset $R_H$ is c.c.c. by the assumption on the graph $H$, therefore it is a Suslin c.c.c.\ forcing.  
Theorem~\ref{fubinitheorem} shows that for all countable families $\cG$ of analytic hypergraphs with only finite edges and proper quotient, $\phi(P_{I_{\cG}}, H)$ is equivalent to $\cG\not\perp R_H$. A reference to Theorem~\ref{preservationtheorem} now concludes the argument.
\end{proof}

\begin{corollary}
Let $H$ be a closed acyclic graph on a Polish space $X$. In the product Silver extension, the space $Z$ is covered by the ground model coded compact $H$-anticliques.
\end{corollary}

\begin{proof}
This is an immediate conclusion of the conjunction of Corollary~\ref{silveracycliccorollary} and Corollary~\ref{aacorollary}.
\end{proof}

\begin{corollary}
The preservation of outer Lebesgue measure is preserved under the countable support product of actionable finitary ideals.
\end{corollary}

\begin{proof}
Consider the random forcing $P$. By Theorem~\ref{randomoptimaltheorem}, the preservation of outer Lebesgue measure is equivalent to $\cG\not\perp P$ for countable families of analytic finitary hypergraphs. Now, $\cG\not\perp P$ is preserved by the template products by Theorem~\ref{preservationtheorem}; this completes the proof.
\end{proof}

\begin{corollary}
In the product Silver extension, for every infinite subset $b\subset\cantor$ there is a Borel ground model coded set $D\subset\cantor$ such that both sets $b\setminus D, b\cap D$ are infinite.
\end{corollary}

\begin{proof}
This is an immediate consequence of the computation of the product ideal, Theorem~\ref{preservationtheorem}, and Corollary~\ref{splittingcorollary}.
\end{proof}

The proof of Theorem~\ref{producttheorem} is a rather straightforward, albeit tedious, fusion argument. I will need a definition of a fusion sequence which can deal with finitary hypergraphs as opposed to just graphs.

\begin{definition}
Let $\Gamma$ be a group, and $z$ be a finite function from $\gw$ to $\Gamma$. Then $\prod z\in\Gamma$ is the product $\prod_{i\in\dom(z)}z(i)$ taken in the increasing order. I also put $\prod 0=1$.
\end{definition}

\begin{definition}
Let $X$ be a Polish space with a continuous action of a group $\Gamma$ and a $\gs$-ideal $I$ on $X$. Let $A\subset X$ be a Borel $I$-positive set. A \emph{fusion sequence} below $A$
is a sequence $\langle a_i, B_i\colon i\in\gw\rangle$ such that

\begin{enumerate}
\item $B_0\supset B_1\supset\dots$ are compact $I$-positive subsets of $A$;
\item $a_i\subset\Gamma$ are finite subsets of $\Gamma$, each including the unit element;
\item for each $i\in\gw$, the sets $\{\gamma\cdot B_{i+1}\colon \gamma\in a_i\}$ form a pairwise disjoint family of subsets of $B_i$;
\item whenever $i\in\gw$ and $z\in\prod_{j\in i}a_j$, then the diameter of the set $\prod z\cdot B_i$ in some fixed complete compatible metric is $<2^{-i}$.
\end{enumerate}
\end{definition}

The intersection $\bigcap_iB_i$ is clearly a singleton $\{x\}$ for a point $x$ which I will call the \emph{resulting point} of the fusion sequence. It is also possible to form the continuous injective map $h\colon \prod_ia_i\to A$ by $h(z)=\lim_i \prod (z\restriction i)\cdot x$. I will refer to this map as the \emph{resulting map}.

\begin{proof}[Proof of Theorem~\ref{producttheorem}]
It is rather easy to see that (2) implies the negation of (1). Suppose that $B_n\subset X_n$ are Borel $I_n$-positive sets for each $n\in\gw$ such that $\prod_nB_n\subset A$. Use Corollary~\ref{meagercorollary} to find Polish topologies $\tau_n$ on each space $X_n$ such that $B_n$ is $\tau_n$-comeager while all sets in the ideal $I_n$ are $\tau_n$-meager. Consider the product topology $\tau=\prod_n\tau_n$ on $\prod_nX_n$. A review of definitions together with the Kuratowski theorem shows that the generating sets of the $\gs$-ideal $I$ are $\tau$-meager, and so all sets in $I$ are $\tau$-meager. On the other hand, the set $\prod_nB_n$ is $\tau$-comeager, and so neither it nor its superset $A$ belong to the ideal $I$. 

To show that the negation of (1) implies (2) is more tedious. Let $A\subset X$ be an analytic $I$-positive set. Find continuous actions of countable groups $\Gamma_n$ on the spaces $X_n$ witnessing the fact that the families $\cG_n$ is actionable.
Fix also a compatible metric on each of the spaces. Let $\Gamma$ be the finite support product $\prod_n\Gamma_n$, and naturally identify each group $\Gamma_n$ with a subset of $\Gamma$. Note that the hypergraph family
$\prod_n\cG_n$ is actionable and consists of hypergraphs with finite edges. Thus, the quotient poset $P_I$ is proper and bounding by Corollary~\ref{finitarycorollary}, and as a result, every analytic $I$-positive set has a compact $I$-positive subset by the standard criterion for the bounding property \cite[Theorem 3.2.2]{z:book2}

Let $\gw=\bigcup_nb_n$ be a partition of $\gw$ into infinite sets. Let $f$ be a bookkeeping tool: a function from $\gw$ to the set of all pairs $\langle u, G\rangle$ such that  $u$ is a partial function from $\gw$ to $\Gamma$ with finite domain and $G\in\bigcup_n\cG_n$, such that on every set $b_n$ the function $f$ takes every possible value infinitely many times. For each number $n\in\gw$ and every $I$-positive compact set $B\subset X$ write $p_n(B)\subset X_n$ for the projection of the set $B$ into the $n$-th coordinate.
Observe that the set $p_n(B)$ is compact and also $I_n$-positive. By induction on $i\in\gw$ construct a fusion sequence $F=\langle a_i, B_i\colon i\in\gw\rangle$ below the set $A$ such that

\begin{enumerate}
\item for each $n\in\gw$, the sequence $F_n\langle a_i, p_n(B_i)\colon i\in b_n\rangle$ is a fusion sequence for the ideal $I_n$; in particular, $a_i\subset\Gamma_n$;
\item whenever $i\in b_n$ is such that $f(i)=\langle u, G\rangle$ where for some $k\leq i$ $u\in\prod_{j\in b_n\cap k}a_j$ and $G\in\cG_n$ then for some function $v\in\prod_{j\in b_n\cap i}a_j$ extending $u$, for every $x\in B_{j+1}$, the set $\{\gamma\cdot x(n)\colon \gamma\in a_i\}\subset X_n$ contains all vertices of some edge in the hypergraph $(\prod v)^{-1}\cdot G$.
\end{enumerate}

\noindent The induction is easy to perform. Write $B_{-1}=A$. Suppose that sets $a_j, B_j$ have been constructed for all $j\in i$, and $n\in\gw$ is such that $i\in b_n$. If the assumptions of the second item of the induction hypothesis are not satisfied, then let $a_i=\{1\}$ and shrink the set $B_{i-1}$ to an $I$-positive compact set such that the diameter demands for the fusion sequence $F$ and $F_n$ are satisfied. This completes the induction step in this case. Suppose that the assumptions of the second item of the induction hypothesis are satisfied and let $f(i)=\langle u, G\rangle$. Let $v\in\prod_{j\in b_n\cap i}a_j$ be any extension of $u$. Consider the set $C$ of all points $x\in B_{i-1}$ such that for some finite set $a\subset \Gamma_n$ containing the unit, all points $\gamma\cdot x$ for $\gamma\in a$ belong to the set $B_{i-1}$ and there exists an edge of the graph $(\prod v)^{-1}G$ consisting of points in the set $\{\gamma\cdot x(n)\colon \gamma\in a\}$. The set $C\subset B_{i-1}$ is analytic and its complement in $B_{i-1}$ is a $(\prod v)^{-1}\cdot \hat G$-anticlique, so the set $C$ is $I$-positive. By the countable additivity of the $\gs$-ideal $I$, there is a fixed finite set $a\subset\Gamma_n$ containing the unit such that $D\notin I$ where $D$ contains exactly all  points $x\in B_{i-1}$ all points $\gamma\cdot x$ for $\gamma\in a$ belong to the set $B_{i-1}$ and there exists an edge of the graph $(\prod v)^{-1}G$ consisting of points in the set $\{\gamma\cdot x(n)\colon \gamma\in a\}$. The set $D$ is analytic, and so it contains a compact $I$-positive set $B_i\subset X$. Shrink the set $B_i$ if necessary to satisfy the diameter demands on the fusion sequences $F$ and $F_n$; this completes the induction step in this case as well.

Let $x_r\in X$ be the resulting point of the fusion sequence $F$ and let $h\colon\prod_ja_j\to X$ be the associated map. For each $n\in\gw$ the point $x_r(n)\in X_n$ is the resulting point of the fusion sequence $F_n$. Let $h_n\colon \prod_{j\in b_n}a_j\to X_n$ be the resulting map of the fusion sequence $F_n$.

\begin{claim}
The compact set $\rng(h_n)\subset X_n$ is $I_n$-positive.
\end{claim}

\begin{proof}
It will be enough to show that whenever $G\in\cG$ is a hypergraph and $C\subset X_n$ is a Borel $G$-anticlique, then $h_n^{-1}C\subset\prod_ja_j$ is meager. Suppose for contradiction that it is nonmeager, and use the Baire category theorem to find
a number $k$ and a sequence $u\in\prod_{j\in b_n\cap k}a_j$ such that the set $h_n^{-1}C$ is comeager in $[u]$. Find $i\in b_n$ bigger than $k$ such that the bookkeeping tool captured $u, G$ at $i$, that is $f(i)=\langle u, G\rangle$. By the induction hypothesis,
there must be some $v\in\prod_{j\in b_n\cap i}a_j$ extending $u$ such that for every $x\in B_{i}$, the set $\{\gamma\cdot x(n)\colon \gamma\in a_i\}\subset X_n$ contains all vertices of some edge in the hypergraph $(\prod v)^{-1}\cdot G$. Since the set $h_n^{-1}C$ is comeager in $[v]$, there must be a point $z\in\prod_{j\in b_n}a_j$ extending $u$ such that $z(i)=1$ and for every $\gamma\in a_i$, the point $z_\gamma$ obtained from $z$ by rewriting the $i$-th entry to $\gamma$ belongs to $h_n^{-1}C$.
Let $y\in\prod_ja_j$ be any point such that $y\restriction i$ returns only unit values and for all $j\in b_n\setminus i$, $y(j)=z(j)$. Then $h(y)$ is a point in $B_i$, and so by the second item of the induction hypothesis the set $\{\gamma\cdot h(y)(n)\colon \gamma\in a_i\}$
contains an edge $e\in (\prod v)^{-1}G$. Then $\prod v\cdot e$ is an edge in $G$. However, $\prod v\cdot\gamma\cdot h(y)(n)=h_n(z_\gamma)$ holds for each $\gamma\in a_i$ and so these points must contain some $G$-edge. This contradicts the assumption that
the set $C$ was a $G$-anticlique to begin with.
\end{proof}

\begin{claim}
The set $\prod_n\rng(h_n)$ is a subset of the set $A$.
\end{claim}

\begin{proof}
If $z_n\in\prod_{j\in b_n}a_j$ are any points for each $n\in\gw$, let $z=\bigcup_nz_n$. Thus $z\in\prod_{j\in\gw}a_j$. It is immediate that $h(z)=\langle h_n(z_n)\colon n\in\gw\rangle$ and so this point is in the initial set $A\subset X$.
\end{proof}

\noindent This completes the proof of the theorem.
\end{proof}

\subsection{Linear iterations}
\label{iterationsubsection}

It turns out that linear iterations of hypergraphable forcings translate into template products. In order to get neat ZFC theorems, I will need a suitable notion of non-elementary properness. The machinery maintains plenty of its content even without it, and under suitable large cardinal assumptions, the non-elementary properness can be replaced by properness with no harm to the validity of the theorems.

\begin{definition}
Let $X$ be a Polish space and $I$ a $\gs$-ideal on it.

\begin{enumerate}
\item Whenever $a\subset\baire$ is a countable set, $P_I^a$ is the partial order of sets $A\subset X$ which are $I$-positive and $\gS^1_1(z)$ for some finite sequence $z\in a^{<\gw}$ of parameters;
\item The poset $P_I$ is \emph{1-non-elementary proper} with a parameter $z\in\baire$ if for every countable set $a\subset\baire$ containing $z$ and every analytic set $A\in I$, the poset $P_I^a$ forces its generic point not to belong to $\dot A$.
\end{enumerate}
\end{definition}

\begin{theorem}
\label{absotheorem}
Let $X$ be a Polish space and $I$ a $\gs$-ideal on it.

\begin{enumerate}
\item If $I$ is \pioneoneonsigmaoneone\ then the statement ``$P_I$ is 1-non-elementary proper with parameter $z$'' is $\mathbf{\Pi}^1_2$;
\item if $P_I$ is 1-non-elementary proper then $P_I$ is proper and preserves Baire category;
\item the $\gs$-ideals generated by actionable families of hypergraphs or countable families of hypergraphs are 1-non-elementary proper.
\end{enumerate}
\end{theorem}

\begin{proof}
To see (1), let $\phi(I, z)$ be the statement ``$P_I$ is 1-non-elementary proper with parameter $z$''. Argue that $\phi(I, z)$ is equivalent to the statement $\psi(I, z)=$``for every countable transitive model $M$ containing the code for $I$ as well as $z$, $M\models\phi(I, z)$''.
This is a routine proof. Clearly, the statement $\psi$ is $\mathbf{\Pi}^1_2$.

For (2), suppose that $P_I$ is 1-non-elementary proper.  Let $M$ be a countable elementary submodel and let $a=M\cap\baire$. Then $P_I\cap M$ is dense in $P_I^a$ and so $P_I\cap M$ forces its generic point not to belong to any ground model coded analytic
set in $I$. The conclusion now follows from \cite[Corollary 3.5.4]{z:book2}.

(3) follows from a trivial review of the proofs in Section~\ref{featuresection}.
\end{proof}

\noindent I do not know an example of a \pioneoneonsigmaoneone\ $\gs$-ideal whose quotient forcing is proper and preserves Baire category and at the same time the quotient forcing is not 1-non-elementary proper. However, it is not excluded that such forcings abound,
since I do not know how to check the status of 1-non-elementary properness for such basic posets as the Matet forcing or the product of two Sacks forcings. What is important though is that the hypergraphable posets discussed in this paper are 1-non-elementary proper.

The theorems on linear iterations of hypergraphable forcings need the following notation. Let $\langle J, R\rangle$ be a set with a strict linear ordering on it, serving as a template. The following theorem uses the notation established in Definition~\ref{templatedefinition}.

\begin{theorem}
\label{iterationtheorem}
Let $\langle J, R\rangle$ be a countable set with a strict linear order on it. Suppose that for each $j\in J$, $\cG_j$ is a countable family of analytic hypergraphs on some Polish space $X_j$ such that the quotient poset is 1-nonelementary proper. Then

\begin{enumerate}
\item the poset $P_J$ is 1-non-elementary proper;
\item whenever $K\subset J$ is an initial segment, then $P_J$ is naturally isomorphic to the iteration $P_K*P_{J\setminus K}$;
\item if $J$ is a well-ordering, then the poset $P_J$ is naturally isomorphic to the usual countable support iteration of the posets $P_{\cG_j}$ along $J$.
\end{enumerate}
\end{theorem}

\begin{proof}
The argument begins with a definition. For a set $A\subset X_J$ and a set $K\subset J$ write $p_K(A)=\{x\in X_K\colon \{y\in X_{J\setminus K}\colon x\cup y\in A\}\notin I_{J\setminus K}\}$. In the case of general templates, this operation behaves quite poorly.
In the present case of a linear order, this operation serves as a forcing projection:

\begin{claim}
\label{projectionclaim}
Suppose that $K\subset J$ is an initial segment and $A\subset X_J$ is an analytic set. Then $p_K(A)\subset X_K$ is an analytic set  and the following are equivalent:

\begin{enumerate}
\item $A\notin I_J$;
\item $p_K(A)\notin I_K$.
\end{enumerate}
\end{claim}

\begin{proof}
The analyticity of the set $p_K(A)$ follows from Theorem~\ref{millerfact}(1). Now, to prove (2)$\to$(1), suppose that $p_K(A)$ does not belong to $I_K$. Pick indices $j_n\in J$, hypergraphs $G_n\in \cG_j$ and Borel $\hat G_n$-anticliques $B_n\subset X_J$; I must produce a point in $A\setminus\bigcup_nB_n$. Write $c=\{n\in\gw\colon j_n\in K\}$. For each $n\in c$, the projection $p(B_n)$ of $B_n$ into the space $X_K$ is an analytic $\hat G_n$ anticlique, and so can be extended to a Borel $\hat G_n$-anticlique. Since the set $p_K(A)$ is $I_K$-positive,
there must be a point $x\in p_K(A)\setminus \bigcup_{n\in c}  p(B_n)$. Now, for each $n\notin c$, the set $q(B_n)=\{y\in X_{J\setminus K}\colon x\cup y\in B_n\}$ is a Borel $\hat G_n$-anticlique in the space $X_{J\setminus K}$, and as the set $A_x=\{y\in X_{L\setminus K}\colon x\cup y\in A\}$ is $I_{J\setminus K}$-positive, it follows that there must be a point $y\in A_x\setminus\bigcup_{n\notin c}q(B_n)$. The point $x\cup y\in X_L$ is the desired element of $A\setminus\bigcup_nB_n$.

To prove (1)$\to$(2), suppose that $p_K(A)$ belongs to $I_K$. 
Find indices $k_n\in K$, hypergraphs $H_n\in\cG_n$ and Borel $\hat H_n$-anticliques $C_n\subset X_K$ for $n\in\gw$ such that $p_K(A)\subset\bigcup_n C_n$. Write $B=X\setminus\bigcup_nC_n$. Now, the ideal $I_{J\setminus K}$ is hypergraphable, and for every point $x\in B$, the set $A_x=\{y\in X_{L\setminus K}\colon x^\cup y\in A\}$ is in it. Use Theorem~\ref{millerfact}(3) to find indices $j_n\in J\setminus K$, hypergraphs $G_n\in\cG_{j_n}$ and Borel sets $B_n\subset B\times X_{J\setminus K}$ such that each vertical section of $B_n$ is a $\hat G_n$-anticlique and for all $x\in B$, $A_x\subset \bigcup_n (B_n)_x$. Use the definition of the iteration to conclude that the Borel sets $\hat B_n=\{x^\cup y\colon x\in B, y\in (B_n)_x\}\subset X_J$ and $\hat C_n=\{x\in X^L\colon x\restriction K\in C_n\}\subset X^J$ are Borel $\hat G_n$- and $\hat H_n$-anticliques, respectively. It is clear that $A=\bigcup_n\hat B_n\cup\bigcup_n\hat C_n$ and so $A\in I_J$ as required.
\end{proof}

To prove (1), let $z\in\baire$ be a parameter coding all the objects named so far, among others the ordering $R$, the Polish spaces $X_j$ and the hypergraphs in $\cG_j$ for all $j\in J$. I will show that the parameter $z$ witnesses the fact that the poset $P_J$ is 1-non-elementary proper. To see this, let $a\subset\baire$ be a countable set containing $z$ and let $P^a_J$ be the poset of all sets which happen to be $\gS^1_1$ in parameters in the set $a$ and $I_J$-positive; the ordering is that of inclusion. Suppose that $K\subset J$ is an initial segment of $J$ which is recursive in $z$; write $L=J\setminus K$. Let $P^a_K$ be the poset of all sets which happen to be $\gS^1_1$ in parameters in the set $a$ and $I_K$-positive; the ordering is that of inclusion. Let $\pi\colon P^a_K\to P^a_J$ be the map
defined by $\pi(A)=A\times X_L$.

\begin{claim}
\label{rrclaim}
The map $\pi$ is a regular embedding of $P^a_K$ to $P^a_J$.
\end{claim}

\begin{proof}
The projection from $P^a_J$ to $P^a_K$ is just the function $p_K$ as Claim~\ref{projectionclaim} shows.
\end{proof}

\noindent Now, suppose that $x\in X_K$ is a $P^a_K$-generic point. I want to evaluate the remainder poset $P^a_J/x$.
In the model $V[x]$, let $Q$ be the poset of all subsets of $X_L$ which happen to be $\gS^1_1$ in parameters in $a\cup\{x\}$ and $I_L$-positive.
For every set $B\in Q$ find a set $\chi(B)\subset X_J$ which is $\gS^1_1$ in parameters in the set $a$ such that $B=\{y\in X_L\colon x\cup y\in\chi(B)\}$.

\begin{claim}
\label{ssclaim}
The map $\chi$ induces an isomorphism between the separative quotients of $Q$ and $P^a_J/x$.
\end{claim}

\begin{proof}
By definitions, the quotient poset $P^a_J/x$ consists of all sets $A\subset X_J$ which happen to be $\gS^1_1$ in parameters in the set $a$ and such that $x\in p_K(A)$ holds; the ordering is that of inclusion. 
Thus, the map $\chi$ is a map from $Q$ to $P^a_J/x$. The required properties of the map $\chi$ are not difficult to check.
\end{proof}

To conclude the proof of (1), suppose that $A\subset X_J$ is an analytic set in the ideal $I_J$. Find indices $j_n$, hypergraphs $G_n\in\cG_{j_n}$ and Borel $\hat G_n$-anticliques $B_n\subset X_J$ such that $A\subset \bigcup_nB_n$. Let $x\in X_K$ be a $P^a_K$-generic point; I must show that for each $n\in\gw$, $x\notin B_n$. Fix a number $n\in\gw$ and write $j=j_n$. By Claim~\ref{rrclaim} applied to $K=\{i\in J\colon i\mathrel R j\}$, it is the case that $x\restriction K\in X_K$ is $P^a_K$-generic over $V$. In the model $V[x\restriction K]$, write $Q$ for the poset of all subsets of $X_j$ which happen to be $\gS^1_1$ in parameters in $a\cup\{x\}$ and $I_j$-positive. By Claim~\ref{ssclaim}, the point $x(j)\in X_j$
is $Q$-generic over $V[x\restriction K]$. Let $C_n\subset X_j$ be the set $\{y\in X_j\colon\exists v\in B_n\ x\restriction K\subset v\land v(j)=y\}$. This is an analytic set in the model $V[x\restriction K]$ which is a $G_n$-anticlique, therefore it belongs to the ideal $I_j$.
The poset $P_{I_j}$ is 1-non-elementary proper in $V$ and by Theorem~\ref{absotheorem}(1) it is also 1-non-elementary proper in the model $V[x\restriction K]$. Thus, in the model $V[x\restriction K]$ the poset $Q$ forces its generic point to lie outside of $C_n$, in particular $x(j)\notin C_n$. It follows that $x\notin B_n$ as required. The proof of (1) is complete.

To argue for (2), note that the poset $P_K$ is naturally regularly embedded in the poset $P_J$. To see this, deal with the posets $Q_K, Q_J$ of analytic $I_K$ and $I_J$-positive sets instead (in which $P_K, P_J$ are dense by Theorem~\ref{millerfact}(2)) and use Claim~\ref{projectionclaim} to argue that the map $p_K\colon Q_J\to Q_K$ is a projection. To compute the remainder poset, write $L=J\setminus K$ and observe that whenever $B\subset X_K$ is a Borel $I_J$-positive set and $\dot C$ is a $P_K$-name
for a $I_L$-positive set, then thinning out the condition $B$ if necessary, by the Borel reading of names one can find a Borel set $D\subset B\times X_L$ so that $B\Vdash\dot C=\dot D_{\dotxgen}$. Then $D$, viewed as a subset of the space $X_J$, is a condition in the remainder poset.

To argue for (3), suppose that $\langle J, R\rangle$ is a wellordering. Let $P$ be the usual combinatorial countable support iteration of the posets $P_{\cG_j}$ for $j\in J$ and let $\dotxgen$ be the $P$-name for the generic element of $X$. Let $I^*$ be the $\gs$-ideal of analytic sets $A\subset X$ such that $P\Vdash\dotxgen\notin\dot A$. The properness assumptions imply that the poset $P$ is in the forcing sense equivalent to the poset $P_J$ \cite[Theorem 5.1.6]{z:book2}. It is now only necessary to verify that the ideals $I, I^*$ contain the same analytic sets.

Suppose first that $A\subset X$ is an analytic set in the ideal $I$. Find indices $j_n\in J$, hypergraphs $G_n\in\cG_{j_n}$ and Borel $\hat G_n$-anticliques $B_n\subset X$ such that $A\subset\bigcup_nB_n$. To conclude that $A\in J$,
suppose towards contradiction that this fails and there is some condition $p\in P$ which forces $\dotxgen\in \dot A$. Strengthening the condition $p$ if necessary, I may find a definite number $n\in\gw$  such that $p\Vdash\dotxgen\in\dot B_n$. Write $j=j_n$ and $K=\{i\in J\colon i\mathrel R j\}$. Thus $p\Vdash\dotxgen(j)\in\dot C$, where $\dot C=\{y\in X_{j}\colon\exists z\in X_{J\setminus K\cup\{j\} }\ \dotxgen\restriction K\cup \{\langle j, y\rangle\}\cup z\in\dot B_n\}$. Since the set $B_n$ is a $\hat G_n$-anticlique, the set $C$ is forced to be an analytic $G_n$-anticlique, and it is in the model $V[\dotxgen\restriction K]$. Every analytic $G_n$-anticlique is covered by a Borel $G_n$-anticlique, and the poset $P_{\cG_{j}}$ forces its generic point to avoid all such anticliques.
This is a contradiction.

Suppose now that $A\subset X$ is an analytic $I$-positive set and argue that there is a condition $p\in P$ such that $p\Vdash\dotxgen\in\dot A$. This is proved by transfinite induction on the length of the well-ordering $J$. I will treat the case $\langle J, R\rangle=\langle\gw, \in\rangle$, the general case is a routine adaptation. Let $f\colon \baire\to X$ be a continuous function whose range is $A$. For each $n\in\gw$, let $P_n$ be the iteration up to, and not including, $n$.
 By induction on $n\in\gw$, find $P_n$-names $\dot p_n$ and $\dot m_n$ for a condition in $P_{\cG_n}$ and a natural number respectively so that
$\langle p_i\colon i\leq n\rangle\in P_{n+1}$ forces in $P_{n+1}$ that $\dotxgen\restriction n+1\in p_{n+1}A_n$, where $A_n\subset X$ is the analytic set $f''[\dot m_j\colon j\in n+1]$.
The induction step is easily performed using Claim~\ref{projectionclaim} and the $\gs$-additivity of the ideal $I$. In the end, the condition $\langle p_i\colon i\in\gw\rangle\in P$ forces that $\dotxgen=\dot f(\langle m_i\colon i\in\gw\rangle)\in\dot A$.
\end{proof}

There are some immediate attractive corollaries:

\begin{corollary}
\label{boundingpcorollary}
The bounding property is preserved under linear iterations of 1-non-elementary proper hypergraphable forcings.
\end{corollary}

\begin{proof}
The bounding property is equivalent to $\cG\not\perp P$ where $P$ is the Hechler forcing by Theorem~\ref{hechleroptimaltheorem} and this is preserved under the template products by Theorem~\ref{preservationtheorem}.
\end{proof}

\noindent Note that the diagonalizability used in Theorem~\ref{boundingtheorem} to imply the bounding property is itself preserved by the template products.

\begin{corollary}
\label{fubinipcorollary}
The linear iterations of 1-non-elementary proper posets generated by families of analytic hypergraphs with the Fubini property preserve outer Lebesgue measure.
\end{corollary}

\begin{proof}
Let $P$ be the random forcing. The Fubini property of the hypergraph families implies $\cG\not\perp P$ by Theorem~\ref{fubinitheorem}, which is preserved under the template products by Theorem~\ref{preservationtheorem}, and in turn it implies the preservation of outer Lebesgue measure by Theorem~\ref{randomoptimaltheorem}. An alternative approach would be to observe that the Fubini property itself is preserved under the template products.
\end{proof}

\begin{corollary}
\label{independentpcorollary}
The linear iterations of posets generated by countable families of finitary open hypergraphs do not add independent reals.
\end{corollary}

\begin{proof}
Note that the quotient posets given by countable families of finitary open hypergraphs are 1-non-elementary proper. Let $U$ be an ultrafilter and $P$ the standard poset diagonalizing $U$. Recall that for every countable family $\cG$ of finitary open hypergraphs, $\cG\not\perp^* P$ holds by Theorem~\ref{smalltheorem}. Thus, Theorem~\ref{preservationtheorem}
guarantees that for the template product families $\mathcal{H}$ obtained, $\cH\not\perp^*P$ holds as well. This in turn implies that independent reals are not added in the template product either.
\end{proof}

\begin{corollary}
\label{uucorollary}
Suppose that $H$ is a closed graph  on a Polish space $Z$ without a perfect clique. The property $\phi(H)=$``every point of $Z$ in the extension belongs to a ground model coded compact $H$-anticlique'' is preserved under the linear iterations of finitary hypergraphable 1-non-elementary proper forcings.
\end{corollary}

\begin{proof}
Let $R_H$ be the poset associated with the graph $H$ in Definition~\ref{rhdefinition}. The property $\phi(H)$ is equivalent to $\cG\not\perp R_H$ for the hypergraphable forcings in this class by Theorem~\ref{fubinitheorem}, which is in turn preserved by template products by Theorem~\ref{preservationtheorem}.
\end{proof}

\begin{corollary}
\label{bbcorollary}
\textnormal{(\cite{chan:g0} for $H=$the KST graph)}
Let $H$ be a closed acyclic graph on a Polish space $Z$. In the iterated Silver extension, the space $Z$ is covered by the ground model coded compact $H$-anticliques.
\end{corollary}

\begin{corollary}
Let $H$ be a closed graph on a compact metrizable space $Z$ with countable loose number. In the linearly iterated Miller extension, the space $Z$ is covered by the ground model coded compact $H$-anticliques.
\end{corollary}

\begin{proof}
Let $R_H$ be the poset of Definition~\ref{rhdefinition}, and let $\cG$ be the standard hypergraph generating the Miller forcing as in Example~\ref{millerexample}. Theorem~\ref{millertheorem} shows that $\cG\not\perp R_H$.
This is preserved by the template products by Theorem~\ref{preservationtheorem}. The corollary follows.
\end{proof}

\noindent This answers a question of William Chan (personal communication) about the specific case $H=$the KST graph, which is locally countable and therefore has countable loose number.

\section{Appendix: descriptive set theoretic computations}

In various places in the paper, complexity computations are needed which do not really have much to do with the forcing context of the paper. I gather them in this section.

\subsection{Definability of the hypergraphable ideals}

The general descriptive set theoretic properties of hypergraphable ideals are encapsulated  in the following theorem. The theorem is of folkloric nature. It can be derived from slight, but unpublished, generalizations of hypergraph dichotomies of Lecomte and Miller.
I provide a straighforward proof using effective descriptive set theory and the Gandy--Harrington forcing.

\begin{theorem}
\label{millerfact}
Suppose that $\cG$ is a countable collection of analytic hypergraphs on a Polish space $X$ and write $I=I_{\cG}$ for the associated $\gs$-ideal on $X$. Then

\begin{enumerate}
\item the $\gs$-ideal $I$ is \pioneoneonsigmaoneone;
\item every analytic $I$-positive set contains a Borel $I$-positive subset;
\item if $Y$ is a Polish space and $B\subset Y\times X$ is a Borel set such that each of its vertical sections belongs to $I$, then there are Borel sets $B_n\subset Y\times X$ for $n\in\gw$ such that $B\subset\bigcup_nB_n$ and for each $n$ there is $G\in\cG$ such that every vertical section of $B_n$ is a $G$-anticlique.
\end{enumerate}
\end{theorem}

\noindent I will need two well-known basic facts about effective arguments which are quite independent of the concerns of this paper.

\begin{definition}
Let $X$ be a recursively presented Polish space. Let $z\in\baire$ be a parameter. The \emph{Gandy--Harrington forcing} $P_z$ is the poset of nonempty $\gS^1_1(z)$ subsets of $X$ ordered by inclusion.
\end{definition}

\begin{proposition}
There is a $P_z$-name $\dotxgen$ for an element of $X$ which is forced to belong to precisely those $\gS^1_1(z)$ subsets of $X$ which belong to the generic filter.
\end{proposition}

\begin{proof}
For simplicity assume that $X=\cantor$. Just let $\dotxgen$ be the name of the unique element of $X$ which belongs to all the basic open sets in the generic filter. I need to show that for every condition $p\in P_z$, $p\Vdash\dotxgen\in\dot p$.
To see this, pick a tree $T\subset (2\times\gw)^{<\gw}$ which is recursive in $z$ and projects into $p$. A genericity argument then shows that the tree $U\subset T$ consisting of those nodes $t\in T$ such that the projection of $T\restriction t$ belongs to the
generic filter, is forced to have no terminal nodes. Thus, the tree $U$ is forced to have an infinite branch $b\in (2\times\gw)^\gw$, which projects to $\dotxgen$ and witnesses that $\dotxgen$ belongs to the projection of $T$, which is $p$.
\end{proof}

\begin{proposition}
\label{projectionproposition}
Let $Q_z$ be the Gandy--Harrington forcing on $X^{\gw}$ with parameter $z$. Let $y\in X^\gw$ be a $Q$-generic point. For every natural number $n$, the point $y(n)\in X$ is $P_z$-generic.
\end{proposition}

\begin{proof}
Let $\pi\colon Q_z\to P_z$ be the map assigning to a condition $q\in Q_z$ its projection $\pi(q)\in P_z$ into the $n$-th coordinate. It is clear that the map $\pi$ preserves the ordering. Moreover, for every $q\in Q_z$ and every $p\leq\pi(q)$ there is $q'\leq q$ such that $\pi(q')=p$: just let $q'=\{y\in q\colon y(n)\in p\}$. Thus, the map $\pi$ is a projection map of partial orders, the image of the generic filter on $Q_z$ is forced to be a generic filter on $P_z$, and the proposition follows.
\end{proof}

\begin{proof}[Proof of Theorem~\ref{millerfact}]
For simplicity assume that $X=\cantor$, the hypergraphs in the family $\cG$ are lightface $\Sigma^1_1$, and have arity $\gw$. Let $z\in\baire$ be an arbitrary parameter, and let $C_z=\{A\subset X\colon A$ is $\Sigma^1_1(z)$ and for some $G\in\cG$, $A$ is an $G$-anticlique$\}$.

\begin{claim}
\label{v1claim}
The set $C_z\subset X$ is $\Pi^1_1(z)$, uniformly in $z$.
\end{claim}

\begin{proof}
First, use the effective reflection theorem \cite[Theorem 2.7.1]{kanovei:book} to show that every $\gS^1_1(z)$ set which is a $G$-anticlique for some $G\in\cG$ is covered by a $\Delta^1_1(z)$ anticlique. Thus, the set $C_z$ is in fact the union of all $\Delta^1_1(z)$ anticliques. Then, use the coding of $\Delta^1_1(z)$ sets \cite[Theorem 2.8.1]{kanovei:book} to prove the claim. Namely, $x\in C_z$ if there is a number $e\in\gw$ and $G\in\cG$ such that $e$ is a code for a $\Delta^1_1(z)$ set such that $x$ is in the set and the set is a $G$-anticlique. This is easily decoded as a $\Pi^1_1(z)$ statement.
\end{proof}

\begin{claim}
\label{v2claim}
Whenever $G\in\cG$ and $B\subset X$ is a Borel set  which is a $G$-anticlique, then in the poset $P_z$, the condition $X\setminus C_z$ forces $\dotxgen\notin\dot B$.
\end{claim} 

\begin{proof}
Suppose for contradiction that $p\in P_z$ is a condition disjoint from $C_z$ and forcing $\dotxgen\in\dot B$. The set $p\subset\cantor$ does contain some $G$-edges, since otherwise it would be a subset of $C_z$ by the definition of $C_z$ and could not belong to the poset $P_z$. Let $Q_z$ be the poset
of nonempty $\gS^1_1(z)$ subsets of $X^\gw$ adding a point $\dotygen\in X^\gw$. Note that $q=G\cap A^\gw$ is a nonempty $\gS^1_1(z)$ set and therefore a condition in the poset $Q_z$. 
By the forcing theorem applied to the poset $P_z$ and Proposition~\ref{projectionproposition}, $q$ forces $\dotygen$ to be a $G$-edge consisting of points in the set $B$. However, since the set $B$ is a $G$-anticlique, it is forced to be an anticlique in all generic extensions
by the Mostowski absoluteness. This is a contradiction.
\end{proof}

\begin{claim}
\label{v3claim}
A $\gS^1_1(z)$ set $A\subset\cantor$ belongs to the ideal $I$ if and only if $A\subset C_z$.
\end{claim}

\begin{proof}
First of all, the set $C_z$ belongs to the $\gs$-ideal $I$. It is a union of countably many analytic sets, each of which is a $G$-anticlique for some $G\in \cG$. By the first reflection theorem, every analytic $G$-anticlique is a subset of a Borel $G$-anticlique, and so $C_z\in I$.
This proves the right-to-left implication.

For the left-to-right implication, prove the contrapositive. If $A\not\subset C_z$, then the set $A\setminus C_z$ is nonempty, and it is $\gS^1_1(z)$ by Claim~\ref{v1claim}. Let $\{B_n\colon n\in\gw\}$  be a collection of Borel sets such that for each $n$ there is a hypergraph
$G_n\in\cG$ such that $B_n$ is $G_n$-anticlique. The set $A\setminus C_z$ is a condition in the poset $P_z$, and it forces the generic point to belong to $A$ and not to any of the sets $B_n$ for $n\in\gw$ by Claim~\ref{v2claim}. Thus, in the $P_z$-extension, $A\not\subset\bigcup_nB_n$, and by the Mostowski absoluteness $A\not\subset\bigcup_nB_n$ holds in the ground model as well.
\end{proof}

Now, (1) follows immediately. Suppose that $A\subset Y\times X$ is an analytic set for some Polish space $Y$. Without loss of generality assume that $Y=\baire$, and $A$ is $\gS^1_1(z)$ for some parameter $z\in\baire$. Then, for every $y\in Y$ the vertical section $A_y$
is $\gS^1_1(z, y)$ and so by Claim~\ref{v3claim} it  is $I$-positive if and only if it contains some point which is not in $C_{z,y}$. The latter statement is an analytic condition.

To prove (2), suppose that $A\subset X$ is an analytic $I$-positive set, $A\in\gS^1_1(z)$ for some parameter $z\in \baire$. Let $M$ be a countable elementary submodel of a large structure containing $z$, and consider the set $A'=\{x\in X\colon
x$ is $P_z$-generic over the model $M$, below the condition $A\setminus C_z\}$. The set $A'$ is a subset of $A$, as all sets of generic points over any countable model, the set $A'$ is Borel by \cite[Fact 1.4.8]{z:book2}, and it is $I$-positive by Claim~\ref{v2claim}.

To prove (3), I will need some additional considerations. Assume without loss of generality that $Y=\baire$. For each hypergraph $G\in \cG$, consider the $\gS^1_1$ hypergraph $\hat G$ on $Y\times X$ by setting $\langle \langle y_n, x_n\rangle\colon n\in\gw\rangle\in\hat G$ if all $y_n$'s are pairwise equal
and $\langle x_n\colon n\in\gw\rangle\in G$. For every parameter $z\in\baire$, let $\hat P_z$ be the Gandy--Harrington poset on $Y\times X$ and let $\hat C_z\subset Y\times X$ be the union of all $\gS^1_1(z)$ sets which are $\hat G$ anticliques for some $G\in\cG$. Claim~\ref{v1claim} applied to the hatted family of hypergraphs shows that the set $\hat C_z$ is $\Pi^1_1(z)$ uniformly in $z$. The poset $P_z$ adds a pair
$\langle\dotygen, \dotxgen\rangle\in Y\times X$.

\begin{claim}
\label{v4claim}
Let $A\subset Y\times X$ be a $\gS^1_1(z)$ set. The condition $(Y\times X)\setminus\hat C_z$ forces the following: if $\langle \dotygen,\dotxgen\rangle\in A$, then the vertical section $A_{\vecygen}$ is not a $G$-anticlique for any $G\in\cG$.
\end{claim}

\begin{proof}
Suppose towards a contradiction that $p\in\hat P_z$ disjoint from $\hat C_z$ is a condition forcing $\langle\dotygen, \dotxgen\rangle\in\dot A$ and the vertical section $A_{\vecygen}$ is a $G$-anticlique. 
Replacing $p$ with $p\cap A$ if necessary we may assume that $p\subset A$. Let $G\in\cG$ be a hypergraph. Since $p$ is disjoint from $\hat C_z$,
it is not a $\hat G$-anticlique. Consider the poset $\hat Q_z$ of all nonempty $\gS^1_1(z)$ subsets of $(Y\times X)^\gw$ and the condition $q=\hat G\cap [p]^\gw$ in it. By Proposition~\ref{projectionproposition}, 
the $\hat Q_z$-generic filter yields points $y\in Y$ and $x_n\in X$ for $n\in\gw$
such that each pair $\langle y, x_n\rangle$ is $\hat P_z$-generic below $p$, and $\langle x_n\colon n\in\gw\rangle\in G$. This is a contradiction with the initial choice of the condition $p$.
\end{proof}

\begin{claim}
\label{v5claim}
Let $p\in\hat P_z$ be a condition disjoint from the set $\hat C_z$. Then $p$ forces the vertical section $p_{\vecygen}$ to be $I$-positive.
\end{claim}

\begin{proof}
In view of Claim~\ref{v3claim}, it will be enough to show that $p\Vdash\dotxgen\notin C_{\vecygen}$ holds. Suppose towards a contradiction that this fails, and thinning out the condition $p$ if necessary, use the definition of the set $C_{\vecygen}$ to find a hypergraph $G\in\cG$ and a $\gS^1_1$ set $A\subset Y\times X$ such that $p\Vdash A_{\vecygen}$ is a $G$-anticlique containing the point $\dotxgen$. This contradicts Claim~\ref{v4claim} though.
\end{proof}

Towards the proof of (4), let $B\subset Y\times X$ be a Borel set with all sections in the $\gs$-ideal $I$. Let $z\in\baire$ be a parameter such that $B$ is $\gS^1_1(z)$. It must be the case that $B\subset\hat C_z$ holds. Otherwise, the condition $B\setminus C_z$
would force $B$ to have an $I$-positive vertical section by Claim~\ref{v5claim}, and by Mostowski absoluteness the set $B$ would have to have an $I$-positive section in the ground model as well, contradicting the initial assumption. Thus, $B\subset\bigcup_nA_n$
for some $\gS_1^1(z)$ sets $A_n\subset Y\times X$ such that for some hypergraphs $G_n\in\cG$, every vertical section of $A_n$ is a $G_n$-anticlique. By the first reflection theorem, there are Borel sets $B_n\subset Y\times X$ such that
$A_n\subset B_n$ and every vertical section of $B_n$ is a $G_n$-anticlique. The sets $B_n$ for $n\in\gw$ work as required.
\end {proof}

\begin{corollary}
\label{meagercorollary}
Let $I$ be a hypergraphable ideal on a Polish space $X$. For every $I$-positive analytic set $A\subset X$ there is a Polish topology $\tau$ on $X$ generating the same Borel structure as the original one, such that $A$ is $\tau$-comeager while all sets in $I$ are $\tau$-meager.
\end{corollary}

\begin{proof}
Let $z\in\baire$ be a parameter which codes the Polish space $X$ as well as $A$ and all the hypergraphs generating the $\gs$-ideal $I$. Let $P$ be the poset of nonempty $\gS^1_1(z)$ sets ordered by inclusion, and let $p\in P$ be the condition $A\setminus C_z$.
Let $M$ be a countable elementary submodel of a large structure and let $Y$ be the space of all ultrafilters on $P\restriction p$ which are generic over $M$. Let $f\colon Y\to X$ be the map given by $f(G)=\dotxgen/G$; this is a Borel injection.
Adjust $f$ on a nowhere dense subset of $Y$ to obtain a Borel bijection between $Y$ and $X$. Now, declare a set $B\subset X$ to be $\tau$ open if its $f$-preimage is. It is clear that $\tau$ is a Polish topology on $X$ generating the same Borel structure as the old one.
Clearly, $A$ is comeager in $\tau$; at the same time every set in $I$ is $\tau$-meager by Claim~\ref{v2claim}.
\end{proof}

\subsection{Suslin forcings}

A number of the arguments in the paper are stated in the language of Suslin forcings.

\begin{definition}
A poset $\langle P,\leq\rangle$ is \emph{Suslin} if $P$ is an analytic subset of some ambient Polish space $X$, $\leq$ is an analytic relation on $X$, and incompatibility of conditions in $P$ is an analytic relation on $X$ as well.
\end{definition}

\noindent The c.c.c. of a Suslin forcing is highly absolute among transitive models of set theory. This is recorded in the following standard fact. ZFC- denotes a suitable large finite fragment of ZFC.

\begin{fact}
\label{suslinfact}
Let $P$ be a Suslin forcing. The following are equivalent:

\begin{itemize}
\item in some transitive model $M$ of ZFC- containing the code for $P$, $M\models P$ is c.c.c.;
\item in every transitive model $M$ of ZFC- containing the code for $P$, $M\models P$ is c.c.c.
\end{itemize}
\end{fact}

\noindent The absoluteness of c.c.c. yields the absoluteness of the forcing relation of the Suslin forcing. I will state this feature of Suslin c.c.c. forcings in the following way:

\begin{claim}
\label{oclaim}
Suppose $P$ is a c.c.c. Suslin forcing, $Y$ is a Polish space and $A\subset Y$ is an analytic set. If $M$ is a transitive model of ZFC- containing the codes for $P, Y$ and $A$ and $M\models\dot y$ is a $P$-name
for an element of $Y$ and $p\in P$ is a condition such that $p\Vdash\dot y\in\dot A$, then this statement holds in $V$ as well.
\end{claim}

\begin{proof}
To specify the name $y$, for each basic open set $O\subset Y$ pick a maximal antichain of conditions deciding the statement $\dot y\in\dot O$. 
Let $C\subset Y\times\baire$ be a closed set projecting into $Y$, with a code in the model $M$. Let $d$ be a compatible complete metric on $Y\times\baire$, with a code in the model $M$. The statement ``$p\Vdash\dot y\in\dot A$'' is provably equivalent to 
the existence of $B_n, f_n$ for $n\in\gw$ such that

\begin{itemize}
\item each $B_n$ is a maximal antichain in $P$ below $p$, $B_{n+1}$ refines $B_n$;
\item $f_n$ is a function with domain $B_n$, the values of $f_n$ are basic open subsets of $Y\times\baire$ of $d$-diameter $\leq 2^{-n}$ with nonempty intersection with the set $C$;
\item if $q\in B_n$ and $r\in B_{n+1}$ and $r\leq q$ then the closure of the set $f_{n+1}(r)$ is a subset of $f_n(q)$;
\item if $q\in B_n$ then $q$ is below some condition forcing $\dot y\in O$ for some basic open set $O\subset Y$ which is a subset of the projection of $f_n(q)$.
\end{itemize}

Thus, there must be objects with these properties in the model $M$. However, their properties are analytic or coanalytic--note that the poset $P$ is c.c.c.\ in the model $M$ and so the maximal antichains mentioned above are countable there
by Fact~\ref{suslinfact}, and their maximality is then a coanalytic statement. Since the model $M$ is transitive, these properties survive to $V$ and therefore $p\Vdash\dot y\in A$ as desired.
\end{proof}

\noindent The complexity computations yield the following uniformization theorem. Its statement may be long, but it has the advantage of being directly applicable at several places in the main body of the paper.

\begin{theorem}
\label{uniformizationtheorem}
Let $P$ be a Suslin c.c.c.\ poset, $Y$ a Polish space, $\dot y$ a $P$-name for an element of a Polish space $Y$, and $B_Y\subset Y$ a Borel set such that for some condition $p\in P$, $p\Vdash\dot y\in \dot B_Y$.
Let $X$ be a Polish space, let $I$ be a hypergraphable $\gs$-ideal on $X$ and let $B_X\subset X$ be a Borel $I$-positive set.
Suppose that $B_X\times B_Y=\bigcup_nC_n$ a union of Borel sets.
Then there is a number $n\in\gw$, a Borel $I$-positive set $B\subset B_X$, and a Borel function $f\colon B\to P$ such that for every $x\in B$, $f(x)\Vdash_P\dot y\in (C_n)_x$.
\end{theorem}

\begin{proof}
Without loss of generality assume that the underlying Polish space $X$ is the Baire space $\baire$. Use Corollary~\ref{meagercorollary} to find a Polish topology $\tau$ on $X$ such that $B$ is $\tau$-comeager
and every set in $I$ is $\tau$-meager. Let $Q$ be the poset of nonempty $\tau$-open sets ordered by inclusion, with $\dotxgen$ its name for a generic element of $X$.
The statement $p\Vdash\dot y\in\dot B_Y$ persists into the $Q$-extension by Claim~\ref{oclaim} applied in the $Q$-extension.
Thus, $Q$ forces that there is $n\in\gw$ and a condition in $P$ which is stronger than $p$ and forces that $\dot y\in (C_n)_{\dotxgen}$. Thinning out the set $B_X$ if necessary, I may assume that it decides the value of $n$.
Let $\gs$ be a $Q$-name such that $B_X\Vdash\gs\in P_n$, $\gs\Vdash\dot y\in (C_n)_{\dotxgen}$. Let $M$ be a countable elementary submodel of a large structure containing all objects named so far, and let
$B=\{x\in B_X\colon x$ is $Q$-generic over the model $M\}$. The set $B\subset B_X$ is Borel by its definition, and it is $I$-positive since it is nonmeager in the topology $\tau$. Let $f$ be the map with domain $B$ defined by $f(x)=\gs/x$.
The map $f$ is Borel, and for each $x\in B$, $M[x]\models f(x)\in P$ and $f(x)\Vdash_P\dot y\in (C_n)_x$ holds by the forcing theorem applied in the model $M$. By Claim~\ref{oclaim} again, for every
point $x\in X$ the statement $f(x)\in P$ and $f(x)\Vdash_P\dot y\in (C_n)_x$ persists to $V$. This completes the proof.
\end{proof}

\bibliographystyle{plain} 
\bibliography{odkazy,zapletal,shelah}

\end{document}